\documentclass[preprint,12pt]{elsarticle}

\usepackage{amsmath,amssymb,amsthm}
\usepackage{enumitem}
\usepackage{geometry}
\usepackage{booktabs}
\usepackage{dsfont}
\usepackage[all]{xy}
\usepackage{mathrsfs}
\usepackage{graphicx}
\usepackage{url}

\numberwithin{equation}{section}
\newtheorem{theorem}{Theorem}[section]
\newtheorem{lemma}[theorem]{Lemma}
\newtheorem{proposition}[theorem]{Proposition}
\newtheorem{corollary}[theorem]{Corollary}
\newtheorem{definition}[theorem]{Definition}

\newtheorem{remark}[theorem]{Remark}
\newtheorem{example}[theorem]{Example}
\newtheorem{problem}[theorem]{Problem}

\newcommand{\dn}{\mathord{\downarrow}\hspace{0.05em}}
\newcommand{\up}{\mathord{\uparrow}\hspace{0.05em}}
\newcommand{\uuar}{\mathord{\Downarrow}\hspace{0.05em}}

\begin{document}
\begin{frontmatter}
\title{A new method for constructing a compact and non SI-compact space}
\author{ Zhengmao He, Bin Zhao}

\address{School of Mathematics and Statistics, Shaanxi Normal University,\\ Xi'an 710119, P.R. China}

\begin{abstract}\quad In this paper, we provide a new method for constructing a compact and non $SI$-compact space and thus give a new answer for the question posed by Zhao and Ho. Using this method, we find a countable compact space which fails to be $SI$-compact. As a corollary, some new non sober frames are obtained and thus we give a new answer for a question raised by A. Jung. Inspired by the definition of $SI$-compactness, we introduce a new notion of compactness of topological spaces, namely, $\upsilon$-compactness, which is stronger than $SI$-compactness.  Some properties of $\upsilon$-compactness, $SI$-compactness and weak $\upsilon$-compactness are discussed. Specifically, we prove that (1) all closed subspaces of a $\upsilon$-compact ($SI$-compact, weakly $\upsilon$-compact) space are $\upsilon$-compact ($SI$-compact, weakly $\upsilon$-compact); (2) a characterization of $\upsilon$-compactness for the upper topology is given; (3) some characterizations of weakly $\upsilon$-compact subsets of real line and Sorgenfrey line are given.
\end{abstract}

\begin{keyword}$SI$-compact; $\upsilon$-compact; weakly $\upsilon$-compact; sober; Scott topology
\vspace*{0.6cm}


\end{keyword}

\end{frontmatter}
\vspace{0.1cm}

\begingroup
\renewcommand{\thefootnote}{}
\footnote{This work is supported by the National Natural Science Foundation of China (No.12331016).}
\endgroup
\begingroup
\renewcommand{\thefootnote}{}
\footnote{Corresponding author: Bin Zhao.}
\endgroup\begingroup
\renewcommand{\thefootnote}{}
\footnote{E-mail addresses: hezhengmao@snnu.edu.cn,\ zhaobin@snnu.edu.cn.}
\endgroup

\renewcommand*\contentsname{\hfill Contents \hfill}
\tableofcontents

\section{Introduction}
\setlength{\parskip}{0.5\baselineskip}
\quad In Domain theory, there are serval order topologies which are frequently discussed by scholars, such as the upper topology, lower topology, Scott topology and Lawson topology of posets. In this paper, we will focus on the order topologies of open set lattice. As we know, compactness  is one  of the most important properties for topological spaces (see \cite{RE77}). In \cite{DME15}, Zhao and Ho introduce a new type of compactness, namely $SI$-compactness which is defined by using the Scott topology of the open set lattice. Specifically, a topological space $X$ is called $SI$-compact if for every Scott irreducible subset $\mathcal{U}$ of $\mathcal{O}(X)$, $\bigcup \mathcal{U}= X$ implies $X\in\mathcal{U}$. Note that A subset $F$ of a topological space $X$ is irreducible if for each pair closed sets $A$ and $B$ of $X$, $F\subseteq A\cup B$ implies $F\subseteq A$ or $F\subseteq B$. Inspired by the definition of $SI$-compactness, we propose a new
compactness, namely, $\mathcal{T}$-compactness which is defined by using the order compatible topology $\mathcal{T}$ on the open set lattice. In this paper, we mainly consider $SI$-compactness and $\upsilon$-compactness, where the $\upsilon$-compactness is defined by the upper topology of the open set lattice. The motivation for this study is based on two aspects. Firstly, from the perspective of general topology, the compactness induced by different order topologies is used to distinguish between supercompactness and compactness. This facilitates a more profound understanding of the strengths and weaknesses of several types of compactness. This serves to reinforce the connection between general topology and order theory, while also highlighting the significant role of order topology in general topology. In addition, by studying $SI$-compactness and $\upsilon$-compactness, we provide a new Scott non-Sober frame, which offers a new method to discuss sobriety for the Scott topology in Domain theory. This method provides a possibility for solving the problem raised by H. Miao, X. Xi, Q. Li and D. Zhao in \cite{AB75}.

\quad In \cite{DME15}, Ho and Zhao find that every $SI$-compact space is always compact and they posed the following problem.

{\bf Problem 1:  Is there a compact space which is not $SI$-compact?}

In \cite{BZ35}, we first solve problem 1 by giving a compact and not $SI$-compact space. However, that space is uncountable. The following problem arises naturally.

{\bf Problem 2:  Is there a countable compact space which is not $SI$-compact?}

 In this paper, we give a new example which is countable compact but not $SI$-compact space by using the lower topology of posets. Furthermore, using this example, a new method for constructing a frame whose Scott space is non sober for the Scott topology is obtained and then we solve the following problem 3 again which is proposed by A. Jung in \cite{RTT45}. We call a topological space $X$ is sober if every irreducible closed set of $X$ is the closure of a singleton point.

{\bf Problem 3:  Is there a distributive complete lattice whose Scott space is non-sober?}

The above 3 problems play an important role in strengthening the further connection among order theory, general topology and Domain theory. Concretely, these 3 problems has a great effect on solving some of open problems related to sobriety in Domain theory. The problem 3 is first solved by Xu, Xi and Zhao in \cite{HZ16}. Concretely, they prove that the Scott space of the frame $\mathcal{Q}(L)$ is non sober, where $\mathcal{Q}(L)$ is the poset of all nonempty compact saturated subsets of the Scott space of Isbell's complete lattice with inverse inclusion order. In addition, the examples of non-sobriety for the Scott topology were given by many scholars(see \cite{SOJ81,WSMJ81,GOUY03,HZT96,DFG15}). In this paper, we give a new answer for the problem 3. Specifically, we proved that the lower topology of the dcpo $\widehat{P}$ is non sober for the Scott topology, where the dcpo $\widehat{P}$ is constructed by Miao, xi, Li and Zhao in \cite{AB75}. Furthermore, we know that there is a countable $T_{1}$ topological space $X$ such that the open set lattice $\mathcal{O}(X)$ is non sober for the Scott topology. In addition, we proved that the closed subspaces of a $SI$-compact space are $SI$-compact in this paper. Unfortunately, the open subspaces of a $SI$-compact space are not $SI$-compact in general.

\quad  We further discuss the properties of $\upsilon$-compactness. It is proved that $\upsilon$-compactness is stronger than compactness and they are equivalent in locally hypercompact spaces. In addition, we show that every closed subspace of a $\upsilon$-compact space is $\upsilon$-compact and every topological space is an open subspace of a $\upsilon$-compact space. Besides, the characterization of the $\upsilon$-compactness of the upper topology is obtained.  Furthermore, we find that a $T_{1}$ space $X$ is $\upsilon$-compact if and only if $X$ is finite. So we modify the definition of $\upsilon$-compactness. The modified compactness (called weak $\upsilon$-compactness) can hold in infinite Hausdorff spaces. Simultaneously, we characterize the weakly $\upsilon$-compact subsets of the real line and Sorgenfrey line.

\section{Preliminaries}

\quad A set $P$ with a relation $\leq$ is a {\em poset} if the relation $\leq$ is reflexive, antisymmetric and transitive.

\quad  A subset $D$ of a poset $P$ is {\em directed} if for each pair $a,b\in D$, there is a $c\in D$ such that $a,b\leq c$.

\quad Given a poset $P$ and $A\subseteq P$, let $$\up A=\{x\in P\mid a\leq x\ \mbox{for some}\ a\in A\}$$ and $$\dn A=\{x\in P\mid x\leq a\ \mbox{for some}\ a\in A\}.$$
For every $x\in P$, we write $\dn x$ for $\dn\{x\}$  and $\up x$ for $\up\{x\}$ .

\quad  Let $P$ be a poset and $x,a\in P, A\subseteq P$. We say $x$ an upper(lower) bound of $A$ if $A\subseteq\dn x$ ($A\subseteq\up x$). If the set of upper (lower) bounds of $A$ has a smallest (largest) element $a$, we say $a$ the least upper bound of $A$ and write is as $\bigvee A$ (for supremum). The infimum of set $A$ can be defined dually.

\quad  A poset $P$ is called a {\em directed complete poset} ({\em  dcpo}, for short) if every
directed subset of the poset has a supremum.

\quad Let $P$ be a poset and $x\in P$, $A\subseteq P$. We say $x$ is a {\em minimal(maximal)} element of $A$ if $x\in A$ and $\dn x\cap A=\{x\}$ ($\up x\cap A=\{x\}$). The set of all minimal (maximal) elements of $A$ is denoted by ${\rm min(A)}$ (${\rm max(A)}$).

\quad Let $P$ be a poset. A subset $U$ of  $P$ is {\em Scott open} (see \cite{GG03})
 if

 (i) $U=\up U=\{x\in P\mid y\leq x \ \mbox{for some}\ y\in U\}$; and

 (ii) for any directed subset $D$, $\bigvee D\in U$ implies
$D\cap U\neq \emptyset$ whenever $\bigvee D$ exists.

The Scott open sets on $P$ form the Scott topology $\sigma(P)$. The Scott space $(P, \sigma(P))$ will be simply written as $\Sigma P$. The Scott closed set lattice of $P$ is denoted by  $\Gamma_\sigma(P)$. The lower topology (upper topology) $\omega (P)$ ($\upsilon (P)$) on a poset $P$ is the topology that has $\{P\setminus\up x\mid x\in P\}$ ($\{P\setminus\dn x\mid x\in P\}$) as a subbase. The Alexandroff topology $\alpha (P)$ on a poset $P$ is a topology that consists of all the upper sets of $P$. A topology $\tau$ on poset $P$ is order compatible if $\upsilon(P)\subseteq \tau\subseteq\alpha(P)$.

\quad Let $P$ be a poset and $x, y\in P$. We use $x\|y$ to denote that $x$ and $y$ are incomparable.  We say that $x$ is {\em way below} $y$, in symbols $x\ll y$, iff for all directed subsets $D\subseteq P$ that has a supremum $\bigvee D$, $y\leq\bigvee D$ implies $x\leq d$ for some $d\in D$.\ If  $$\uuar a=\{b\in P\mid b\ll a\}$$ is  directed  and $\bigvee \uuar a=a$\ for all \ $a\in P$, we call $P$  a {\em continuous poset}. A complete lattice $L$ is called a {\em continuous lattice} if $L$ is a continuous poset.

\quad  The open (closed) set lattice of a topological space $X$ is denoted by $\mathcal{O}(X)$ ($\Gamma(X)$). Let $X$ be a $T_{0}$ space. The {\em specialization order} $\leq$ on $X$ is defined by $x\leq y$ if and only if $x\in cl(\{y\})$ (see \cite{GG03}). Unless otherwise stated, throughout the paper, whenever an order-theoretic concept is mentioned in the context of a $T_{0}$ space $X$, it is to be interpreted with respect to the specialization order on $X$.

\quad A topological space $X$ is said to be {\em locally compact} if for every $x\in X$ and every open set $U$ containing $x$, there is an open set $V$ and a compact subset $K$ such that $x\in V\subseteq K\subseteq U$. A topological space $X$ is called {\em locally hypercompact} if for every open set $U$ and $x\in U$, there is a finite set $F\subseteq U$ such that $x\in int(\up F)\subseteq \up F\subseteq U$ (see \cite{KDFGY35}). A topological space $X$ is said to be  {\em core-compact } if  $\mathcal{O}(X)$ is a continuous lattice (see \cite{JGKD13}). Let $P$ be a dcpo. Then for each dcpo $Q$, $\sigma(P\times Q)=\mathcal{O}(\Sigma P\times\Sigma Q)$ iff $\Sigma P$ is core-compact(see \cite{GG03}). It is known that every locally compact space is core-compact. A subset $K$ of a topological space $X$ is {\em strongly compact} if for each open set $U$, $K\subseteq U$ implies that there is a finite subset $F\subseteq X$ such that $K\subseteq\up F\subseteq U$ (see \cite{WM34}).

\quad Let $X$ be a $T_{0}$ space. A subset $K$ is called  {\em supercompact} (see \cite{MU13})  if for each open cover $\{U_{i}\mid i\in I\}\subseteq\mathcal{O}(X)$ of $K$, there is  $i_0\in I$ such that $K \subseteq U_{i_0}$.

\quad A subset $A$ of a  topological space $X$ is {\em saturated} if $A$ equals the intersection of all open sets containing $A$.

\begin{lemma}(see \cite{MU13}) The supercompact saturated sets of a $T_{0}$ space $X$ are exactly the
sets $\up x$ with $x\in X$.
\end{lemma}

\quad  A nonempty subset $F$ of a topological space $X$ is {\em irreducible}, if $F\subseteq A\cup B$ for closed subsets $A$ and $B$ implies $F\subseteq A$ or $F\subseteq B$. For every $x\in X$,\ $cl(\{x\})$ is an irreducible closed set of $X$. There are servals properties of irreducible sets which is used in this paper.

(1) Let $X$ be a topological space and $A\subseteq X$. Then $A$ is irreducible iff $cl(A)$ is irreducible.

(2) Let $X$ be a topological space and $A\subseteq X$. Then $A$ is irreducible iff for each pair open sets $U,V$, $A\cap U\neq\emptyset$ and $A\cap V\neq\emptyset$ implies $A\cap U\cap V\neq\emptyset$.

(3) Let $f:X\longrightarrow Y$ be a continuous map and $A$ an irreducible set of $X$. Then $f(A)$ is an irreducible set of $Y$.

\quad A topological space $X$ is called to be {\em sober} if every irreducible closed set of $X$ is the closure of a unique singleton set. It is proved in  \cite{GG03} that the product of some sober spaces is sober.  A topological space $X$ is called {\em hyper-sober} if for any irreducible set $F$ there is a unique $x \in F$ such that $F \subseteq cl(\{x\})$.

\quad It is proved in \cite{KU13} that for a poset $P$, $(P,\upsilon(P))$ is sober if and only if every irreducible subset of $(P,\upsilon(P))$ has a supremum. So the upper topology on the open set lattice $\mathcal{O}(X)$ of a topological space $X$ is always sober.

\quad  A topological space $X$ is {\em well-filtered} if for each filtered family $\mathcal{F}$ of compact saturated subsets of $X$ and each open set $U$ of $X$, $\bigcap \mathcal{F}\subseteq U$ implies $F\subseteq U$ for some $F\in\mathcal{F}$. It is well known
that every sober space is well-filtered (see \cite{GG03}).

\quad Let $X$ be a $T_{0}$ space. We shall use $\mathcal{Q}(X)$ to denote the poset of all nonempty
compact saturated subsets of $X$ with the reverse inclusion order. The upper Vietoris topology on $\mathcal{Q}(X)$ is the topology that has $\{\Box U\mid U\in\mathcal{O}(X)\}$ as a base, where $$\Box U=\{K\in\mathcal{Q}(X)\mid K\subseteq U\}.$$ The set $\mathcal{Q}(X)$ equipped with the upper Vietoris topology is called the upper space of $X$, denoted by  $P_{S}(X)$. Then the specialization order of the upper space $P_{S}(X)$ is the reverse inclusion order (see \cite{AGL23,MU13,ANVB75}).

\quad Given a topological space $X$, the lower powerspace on $\Gamma_{0}(X)$ (the collection of nonempty closed sets of $X$) is the topology that has $\{\diamond U\mid U\in\mathcal{O}(X)\}$ as a subbase, where $$\diamond U=\{V\in\Gamma_{0}(X)\mid U\cap V\neq\emptyset\}.$$ The lower powerspace is denoted by $P_{H}(X)$. Then $\{\Box U\mid U\in\Gamma_{0}(X)\}$ is a subbase of the closed sets of $P_{H}(X)$, where $\Box U=\{F\in\Gamma_{0}(X)\mid F\subseteq U\}$ (see \cite{ANVB75}).

\begin{lemma} (see \cite{AQ21})Let $X$ be a $T_{0}$ space and $C\in \mathcal{Q}(X)$. Then $C=\up {\rm min(C)}$ and ${\rm min(C)}$ is compact.
\end{lemma}

\quad Given a set $X$,  $\mathbf{Fin}(X)$ is the collection of all the nonempty finite subsets of $X$ and let $\mathscr{F}(X)=\{A\mid X\setminus A\in\mathbf{Fin}(X)\}$.

Let $(X,\tau)$ be a topological space and $Y\subseteq X$. We use $\tau|_{Y}$ to denote the subspace topology on $Y$.

\section{The compactness defined by the compatible topology }
\quad In this section, we introduce some new kinds of compactness by using the order compatible topology on open set lattice and discuss the properties of these kinds of compactness.

\begin{definition} Let $X$ be a topological space and $\mathcal{T}$ an order compatible topology on $\mathcal{O}(X)$. $X$ is called $\mathcal{T}$-compact if for every open cover $\mathcal{U}$ of $X$ which is irreducible in the topological space $(\mathcal{O}(X),\mathcal{T})$, $\bigcup\mathcal{U}=X$ implies $X\in\mathcal{U}$.
\end{definition}

\begin{remark} {\rm (1) Note that the irreducible subsets in Alexandroff topology are exactly all directed subsets. Thus $X$ is $\alpha(\mathcal{O}(X))$-compact iff $X$ is compact iff every directed open cover $\mathcal{U}$ contains $X$. Furthermore, $X$ is always compact if $X$ is $\mathcal{T}$-compact, where $\mathcal{T}$ is an order compatible topology on $\mathcal{O}(X)$.

(2) Given a topological space $X$ and $\mathcal{T}_{1},\mathcal{T}_{2}$ two order compatible topologies on $\mathcal{O}(X)$ with $\mathcal{T}_{1}\subseteq\mathcal{T}_{2}$. Then $X$ is $\mathcal{T}_{2}$ compact if $X$ is $\mathcal{T}_{1}$ compact.

(3) Let $X$ be a topological space and $\mathcal{T}$ a order compatible topology on $\mathcal{O}(X)$. If $(\mathcal{O}(X),\mathcal{T})$ is hyper-sober, then $X$ is $\mathcal{T}$-compact. As a matter of fact, for every $\mathcal{T}$-irreducible open cover $\mathcal{U}$ of $X$, there is a $U\in\mathcal{U}$ such that $\mathcal{U}\subseteq cl_{\mathcal{T}}(\{U\})$ since $(\mathcal{O}(X),\mathcal{T})$ is hyper-sober. Then we can see $$U\subseteq\bigcup\mathcal{U}=X\subseteq\bigcup cl_{\mathcal{T}}(\{U\})=U.$$ Whence, $U=X\in\mathcal{U}$. Hence, $X$ is $\mathcal{T}$-compact.
}
\end{remark}

\quad For every poset $P$, a  nonempty subset $D$ is  countably directed if for any countbale subset $D_{0}\subseteq D$, there is a $d\in D$ such that $D_{0}\subseteq\dn d$. A subset $U\subseteq P$ is $\sigma_{\mathcal{C}}$-Scott open if (i) $U$$=\up U$ and (ii) for any countably directed set $D$, $\bigvee D\in U$ implies
$D\cap U\neq \emptyset$ whenever $\bigvee D$ exists. The the set of all $\sigma_{\mathcal{C}}$-Scott open subsets of $P$ is denoted by $\sigma_{\mathcal{C}}(P)$. Then $\sigma_{\mathcal{C}}(P)$ is still an order compatible topology on $P$ (see \cite{GGL23,KOY35,HQT116}). The space $(P, \sigma_{\mathcal{C}}(P))$ will be simply written as $\Sigma_{\mathcal{C}}P$.

\quad In the following, a $\mathcal{T}$-compact space $X$ is called $\upsilon$-compact, $\sigma_{\mathcal{C}}$-compact and $SI$-compact if  $\mathcal{T}$ is equals to $\upsilon(\mathcal{O}(X))$, $\sigma_{\mathcal{C}}(\mathcal{O}(X))$ and $\sigma(\mathcal{O}(X))$, respectively.

\begin{proposition} Let $f:X\longrightarrow Y$ be a surjective continuous map and $X$ a $\sigma_{\mathcal{C}}$-compact space. Then $Y$ is $\sigma_{\mathcal{C}}$-compact.
\end{proposition}
\begin{proof}Let $\mathcal{U}\subseteq\mathcal{O}(Y)$ be an irreducible open over of $Y$ with respect to $\sigma_{\mathcal{C}}(\mathcal{O}(Y))$. Consider the map $g:\Sigma_{\mathcal{C}}\mathcal{O}(Y)\longrightarrow\Sigma_{\mathcal{C}}\mathcal{O}(X)$ which sends $U$ to $f^{-1}(U)$. Then $g$ is continuous for $g$ preserves the suprema of all countably directed subsets. Hence $g(\mathcal{U})$ is irreducible and $\bigcup g(\mathcal{U})=X$ since $f$ is surjective. As $X$ is $\sigma_{\mathcal{C}}$-compact, $X\in g(\mathcal{U})$. It follows that $Y\in\mathcal{U}$. So $Y$ is $\sigma_{\mathcal{C}}$-compact.
\end{proof}

\begin{theorem} Let $X$ be a topological space. If $\Sigma_{\mathcal{C}}\mathcal{O}(X)$ is sober, then $X$ is compact iff $X$ is $\sigma_{\mathcal{C}}$-compact.

\end{theorem}

\begin{proof}It suffices to show the necessity. Let $\mathcal{U}$  be an irreducible subset with respect to $\Sigma_{\mathcal{C}}\mathcal{O}(X)$ and $\bigcup\mathcal{U}=X$. Assume that $X\not \in \mathcal{U}$. Clearly, $\mathcal{O}(X)\setminus \{X\}$ is a lower set in $\mathcal{O}(X)$. Let $\mathcal{B}$ be a countably directed subset of $\mathcal{O}(X)\setminus\{X\}$. Since $X$ is compact, by Remark 3.2(1), we have that $\bigcup\mathcal{B}\in\mathcal{O}(X)\setminus \{X\}$. Then $\mathcal{O}(X)\setminus \{X\}$ is closed in $\Sigma_{\mathcal{C}}\mathcal{O}(X)$. Since $\mathcal{U}$ is irreducible in $\Sigma_{\mathcal{C}}\mathcal{O}(X)$ and $\mathcal{U}\subseteq \mathcal{O}(X)\setminus \{X\}$, $cl_{\Sigma(\mathcal{O}(X))}(\mathcal{U})$ is a irreducible closed set in $\Sigma_{\mathcal{C}}\mathcal{O}(X)$  and $cl_{\Sigma_{\mathcal{C}}(\mathcal{O}(X))}(\mathcal{U})\subseteq \mathcal{O}(X)\setminus \{X\}$. It follows from the sobriety of $\Sigma_{\mathcal{C}}\mathcal{O}(X)$ that there exists a unique open set $A$ of $X$ such that $$cl_{\Sigma_{\mathcal{C}}\mathcal{O}(X)}(\mathcal{U})=cl_{\Sigma_{\mathcal{C}}\mathcal{O}(X)}\{A\}=\dn_{\mathcal{O}(X)} A\subseteq\mathcal{O}(X)\setminus \{X\}.$$ Since $$X=\bigcup \mathcal{U}\subseteq\bigcup cl_{\Sigma(\mathcal{O}(X))}(\mathcal{U})=A,$$ we have $X=A$. It is a contradiction. Thus $X$ is $\sigma_{\mathcal{C}}$-compact.\end{proof}

\quad For a complete lattice $L$, we let $\mathcal{O}(\hat{L})=\{\hat{L}\setminus K \mid K\in\mathcal{Q}(L)\}$, where $\hat{L}=L\setminus\{1_{L}\}$ and $1_L$ is the top element of $L$. Then $(\hat{L},\mathcal{O}(\hat{L}))$ is a topological space and $\mathcal{Q}(L)$ is order isomorphic to the open set lattice $(\mathcal{O}(\hat{L}),\subseteq)$ (see \cite{BZ35}). Actually, $(\hat{L},\mathcal{O}(\hat{L}))$ is constructed by using the work of W. Thron (see \cite{ARU175}).

\begin{lemma} (see \cite{BZ35}) Let $X$ be a topological space. If $\Sigma\mathcal{O}(X)$ is sober, then $X$ is compact iff $X$ is $SI$-compact.

\end{lemma}

\begin{lemma}(see \cite{BZ35}) Let $L$ be the Isbell complete lattice. Then we have the following statements

${\rm(1)}$ $(\hat{L},\mathcal{O}(\hat{L}))$ is compact;

${\rm(2)}$ $\mathcal{U}=\mathcal{O}(\hat{L})\setminus\{\hat{L}\}$ is irreducible with respect to the Scott topology of $\mathcal{O}(\hat{L})$ and $\bigcup\mathcal{U}=\hat{L}$.

\end{lemma}

\quad Note that for a $SI$-compact space $X$, $\Sigma\mathcal{O}(X)$ is non sober in general. See the following example.

\begin{example} {\rm Let $X$ be the topological space given in Lemma 3.6. Then $X$ is an uncountable compact space and $\mathcal{U}=\mathcal{O}(X)\setminus\{X\}$ is a Scott irreducible open cover of $X$. Take a $\infty\not\in X$ and let $Y=X\cup\{\infty\}$. The topology $\tau$ on $Y$ is given by
$$\tau=\mathcal{O}(X)\cup\{Y\}.$$
Obviously, the inclusion map $i:\mathcal{O}(X)\longrightarrow\mathcal{O}(Y)$ is Scott continuous and $(Y,\tau)$ is $SI$-compact(actually, supercompact). Then $i(\mathcal{U})=\mathcal{U}$ is still irreducible in $\Sigma\mathcal{O}(Y)$. Since $X$ and $(Y,\tau)$ are compact, $\mathcal{U}=\dn\mathcal{U}$ is Scott irreducible closed in $\Sigma\mathcal{O}(Y)$. However, $\dn\mathcal{U}\neq cl_{\sigma(\mathcal{O}(Y))}\{U\}$ for all $U\in\mathcal{O}(Y)$. Thus, $\Sigma \mathcal{O}(Y)$ is non-sober.

}
\end{example}

\begin{example}{\rm Let $L$ be the Isbell complete lattice (see \cite{IJ82}). It has proved in \cite{IJ82} that $\Sigma L$ is non sober. By Lemma 3.6, $(\hat{L},\mathcal{O}(\hat{L}))$ is compact but not $SI$-compact. Actually, by the proof of Example 4.6 in \cite{BZ35}, we see that $\mathcal{U}=\mathcal{O}(\hat{L})\setminus\{\hat{L}\}$ is also irreducible with respect to $(\mathcal{O}(\hat{L}),\sigma_{\mathcal{C}}\mathcal{O}(\hat{L}))$ and $\bigcup\mathcal{U}=\hat{L}$. Therefore, $(\hat{L},\mathcal{O}(\hat{L}))$ is compact but not $\sigma_{\mathcal{C}}$-compact.}
\end{example}

\begin{corollary}  Let $L$ be the Isbell complete lattice. Then $\Sigma_{\mathcal{C}}\mathcal{O}(\hat{L})$ is non-sober.
\end{corollary}
\begin{proof}It is straightforward by using Theorem 3.4 and Example 3.8.\end{proof}

\quad A topological space $X$ is called co-consonant if for every $\mathcal{F}\in\sigma(\mathcal{O}(X))$ and $U\in\mathcal{F}$, there is a finite subset $\mathcal{E}\subseteq\Gamma(X)$ such that $$U\in\bigcap\{\lozenge A\mid A\in\mathcal{E}\}\subseteq \mathcal{F},$$ where $\lozenge A=\{V\in\mathcal{O}(X)\mid A\cap V\neq\emptyset\}$. Equivalently, $X$ is co-consonant if and only if the upper topology and the Scott topology on $\mathcal{O}(X)$ coincide (see \cite{WM34}). It is evident that a $SI$-compact space $X$ is $\upsilon$-compact if $X$ is co-consonant. Given a finite topological space $Y$, it is clear that the upper topology and Alexandroff topology on $\mathcal{O}(Y)$ coincide. Thus $Y$ is always $\upsilon$-compact. Generally, we have the following Proposition 3.10.

\begin{proposition} Let $X$ be a locally hypercompact space. Then $X$ is $\upsilon$-compact iff $X$ is compact.

\end{proposition}

\begin{proof}Necessity: Obvious.

 Sufficiency: Clearly, $\Sigma \mathcal{O}(X)$ is sober since $X$ is locally compact. According to Lemma 3.5, $X$ is $SI$-compact. Then the proof is complete by checking that $X$ is co-consonant. Let $\mathcal{U}\subseteq\mathcal{O}(X)$ be a Scott open set and $U\in\mathcal{U}$. Without loss of generality, we suppose $U$ is nonempty.
As $X$ is a locally hypercompact compact space, we have $U=\bigcup\mathcal{A}$, where $$\mathcal{A}=\{int(\up F)\mid F\in\mathbf{Fin}(U)\}.$$ Since $\mathcal{U}$ is Scott open and $\mathcal{A}$ is directed, there is a $F_{0}\in\mathbf{Fin}(U)$ such that $int(\up F_{0})\in\mathcal{U}$. Consider $$\mathcal{V}=\{V\in\mathcal{O}(X)\mid F_{0}\subseteq V\}=\mathcal{O}(X)\setminus\dn_{\mathcal{O}(X)}\{X\setminus\dn f\mid f\in F_{0}\}.$$ Clearly, $U\in\mathcal{V}$ and $\mathcal{V}$ is open with respect to the upper topology on $\mathcal{O}(X)$. If $W\in\mathcal{V}$, then $int(\up F_{0})\subseteq\up F_{0}\subseteq W$. It follows from $int(\up F_{0})\in\mathcal{U}\in\sigma(\mathcal{O}(X))$ that $W\in\mathcal{U}$. Consequently, $U\in\mathcal{V}\subseteq\mathcal{U}$ and then $\mathcal{U}\in\upsilon(\mathcal{O}(X))$. Thus, $\Sigma P$ is co-consonant.\end{proof}

Let $X$ be a topological space. We consider three topologies on $\mathcal{O}(X)$. They are $\sigma (\mathcal{O}(X))$, $\upsilon(\mathcal{O}(X))$ and the topology $\tau_{X}$ having a subbase $$\{\square Q\mid Q\in\mathcal{Q}(X)\},$$ where $\square Q=\{V\in\mathcal{O}(X)\mid Q\subseteq V\}$. It is clear that $\upsilon(\mathcal{O}(X))\subseteq\tau_{X}\subseteq\sigma (\mathcal{O}(X))$. We say $X$ is consonant (co-consonant) if $\tau_{X}=\sigma (\mathcal{O}(X))$ ($\upsilon(\mathcal{O}(X))=\sigma (\mathcal{O}(X))$). It is easy to see that every co-consonant space is consonant.

\begin{lemma}(see \cite{GHY30})$X$ is locally compact if and only if $X$ is core-compact and consonant.

\end{lemma}

According to the proof of Proposition 3.10, every locally hypercompact space is co-consonant. So it is natural to ask whether every core-compact and co-consonant space is locally hypercompact. In the following, we give a positive answer in Proposition 3.13.

\begin{lemma} (see \cite{WM34}) Let $X$ be a co-consonant $T_{0}$ space. Then every compact subset of $X$ is strongly compact.

\end{lemma}

\begin{proposition} A topological space $X$ is locally hypercompact if and only if $X$ is core-compact and co-consonant.
\end{proposition}
\begin{proof}Necessity: Obviously.

Sufficiency: Let $U\in\mathcal{O}(X)$ and $x\in U$. Since every co-consonant space is consonant, by Lemma 3.11, $X$ is locally compact. Then there is a $V\in\mathcal{O}(X)$ and $K\in\mathcal{Q}(X)$ such that $x\in V\subseteq K\subseteq U$. By Lemma 3.12, $X$ is strongly compact. This means that $K\subseteq \up F$ for some finite subset $F\subseteq U$. Consequently, $$x\in V\subseteq int(\up F)\subseteq \up F\subseteq U.$$ Therefore $X$ is locally hypercompact.\end{proof}

\quad Note that the Alexandroff topology of a poset is always locally hypercompact. So we have the following Corollary 3.14.

\begin{corollary} Let $P$ be a poset. The following statements are equivalent.

{\rm (1)} $\Sigma P$ is compact;

{\rm (2)} $(P,\alpha(P))$ is compact;

{\rm (3)} $(P,\alpha(P))$ is $\upsilon$-compact;

{\rm (4)} $P=\up F$ for some finite subset $F\subseteq P$.

\end{corollary}

\begin{proof}(1)$\Longrightarrow$(4): Using Lemma 2.2, we have $P=\up {\rm min(P)}$ and ${\rm min(P)}$ is a compact subset of $\Sigma P$. Assume that ${\rm min(P)}$ is infinite. Consider $$\mathcal{S}=\{P\setminus E\mid E\subseteq {\rm min(P)}, {\rm min(P)}\setminus E \in \mathbf{Fin}({\rm min(P)})\}.$$ Then $\mathcal{S}\subseteq\sigma(P)$ is directed and $\bigcup \mathcal{S}=P$. However, $P\not\in \mathcal{S}$, a contradiction. Thus, ${\rm min(P)}$ is finite.

(2)$\Longleftrightarrow$(3): By Proposition 3.10.

(2)$\Longrightarrow$(1),(4)$\Longrightarrow$(1) and (4)$\Longrightarrow$(2): Clearly.
\end{proof}

\quad  For every continuous poset $P$, $\Sigma P$ is always locally hypercompact and hence by proposition 3.10, $\Sigma P$ is co-consonant. It is proved in \cite {BWR30} that for a continuous poset $P$, $P_{H}(\Sigma P)$ is still co-consonant. In \cite {BWR30}, He and Zhao asked whether the lower powerspace of a core-compact and co-consonant space is co-consonant. We will give a solution for this problem in Theorem 3.15.

\begin{theorem} Let $X$ be a locally hypercompact $T_{0}$ space. Then $P_{H}(X)$ is co-consonant.
\end{theorem}
\setlength{\baselineskip}{1.1\baselineskip}
\begin{proof}Let $\mathcal{F}\in\sigma(\mathcal{O}(P_{H}(X)))$ and $\mathfrak{F}\in\mathcal{F}$. Then there are $\{V_{j}\mid j\in J_{i},i\in I\}\subseteq\mathcal{O}(X)$ and  finite sets $J_{i}$ ($\forall\ i\in I$), such that $\mathfrak{F}=\bigcup\limits_{i\in I}(\bigcap\limits_{j\in J_{i}}\diamond V_{j})$. Since $\mathcal{F}$ is Scott open, there exists a finite subset $K_{0}\subseteq I$ such that $\bigcup\limits_{i\in K_{0}}(\bigcap\limits_{j\in J_{i}}\diamond V_{j})\in\mathcal{F}$. For the sake of convenience, let $K_{0}=\{1,2,\cdot\cdot\cdot,k_{0}\}$. It is not difficult to verify that $$\bigcup\limits_{i\in K_{0}}(\bigcap\limits_{j\in J_{i}}\diamond V_{j})=\bigcap\limits_{V\in\mathcal{S}_{1}}\diamond V,$$ where $$
\mathcal{S}_{1}=\{\bigcup\limits_{1\leq k\leq k_{0}}V_{i_{k}}\mid (i_{1},i_{2},...,i_{k_{0}})\in J_{1}\times J_{2}\times\cdot\cdot\cdot\times J_{k_{0}}\}.$$
 Let $n=|\mathcal{S}_{1}|$ and $\mathcal{S}_{1}=\{U_{1},U_{2},\cdot\cdot\cdot,U_{n}\}$. Then $$\bigcup\limits_{i\in K_{0}}(\bigcap\limits_{j\in J_{i}}\diamond V_{j})=\bigcap\limits_{U\in\mathcal{S}_{1}}\diamond V=\bigcap\limits_{1\leq k\leq n}\diamond U_{k},$$  and $\bigcap\limits_{1\leq k\leq n}\diamond U_{k}\in\mathcal{F}$.
Define a mapping $\beta_{n}:\Sigma(\prod\limits^{n}(\mathcal{O}(X)))\longrightarrow\Sigma(\mathcal{O}(P_{H}(X)))$ as follows:
$$\beta_{n}(W_{1}, W_{2},\cdot\cdot\cdot,W_{n})=\bigcap\limits^{n}\limits_{i=1}\diamond W_{i}.$$
Since $\beta_{n}$ is Scott continuous for each component $W_{i}$, $\beta_{n}$ is continuous. Since $X$ is core-compact, $\Sigma(\prod\limits^{n}\mathcal{O}(X))=\prod\limits^{n}\Sigma(\mathcal{O}(X))$. Consequently, $\beta_{n}$ is also continuous from $\prod\limits^{n}\Sigma(\mathcal{O}(X))$ to $\Sigma(\mathcal{O}(P_{H}(X)))$. This implies that $\beta_{n}^{-1}(\mathcal{F})\in\mathcal{O}(\prod\limits^{n}\Sigma(\mathcal{O}(X)))$.  So there are Scott open sets $\{\mathcal{A}_{k}\mid 1\leq k\leq n\}\subseteq\sigma(\mathcal{O}(X))$ such that
$$(U_{1},U_{2},\cdot\cdot\cdot, U_{n})\in\mathcal{A}_{1}\times\mathcal{A}_{2}\times\cdot\cdot\cdot\times\mathcal{A}_{n}\subseteq\beta_{n}^{-1}(\mathcal{F}).$$ By Proposition 3.13, $X$ is co-consonant and for each $1\leq k\leq n$, there is a finite subset  $F_{k}=\{u_{k1},u_{k2},\cdot\cdot\cdot,u_{km_{k}}\}\subseteq U_{k}$ such that $$U_{k}\in\{V\in\mathcal{O}(X)\mid F_{k}\subseteq V\}=\bigcap\{\lozenge \dn u_{kj}\mid 1\leq j\leq m_{k}\}\subseteq\mathcal{A}_{k}.$$ Consider the finite family of closed sets of $P_{H}(X)$ $$\mathcal{U}=\{\Box(\bigcup\limits^{n}\limits_{k=1}\dn u_{km})\mid 1\leq m\leq m_{k}\}.$$

Claim 1: $\mathfrak{F}\in\bigcap\{\lozenge V\mid V\in\mathcal{U}\}$.

Let $V=\Box(\bigcup\limits^{n}\limits_{k=1}\dn u_{km})\in\mathcal{U}$, where for all $1\leq k\leq n$, $u_{km}\in F_{k}$ and $1\leq m\leq m_{k}$.
Since $U_{k}\in\bigcap\{\lozenge \dn u_{kj}\mid 1\leq j\leq m_{k}\}$ for all $1\leq k\leq n$ , then $u_{km}\in U_{k}$. Whence, $\bigcup\limits^{n}\limits_{k=1}\dn u_{km}\in(\bigcap\limits^{n}\limits_{k=1} \diamond U_{k})\cap V$. So $\bigcap\limits^{n}\limits_{k=1} \diamond U_{k}\in\bigcap\{\lozenge V\mid V\in\mathcal{U}\}$ and thus $\mathfrak{F}\in\bigcap\{\lozenge V\mid V\in\mathcal{U}\}$.

Claim 2: $\bigcap\{\lozenge V\mid V\in\mathcal{U}\}\subseteq\mathcal{F}$.

Let $\mathfrak{A}=\bigcup\limits_{r\in R}(\bigcap\limits_{s\in S_{r}}\diamond G_{s})\in\bigcap\{\lozenge V\mid V\in\mathcal{U}\}$, where $S_{r}$ is finite for each $r\in R$ and $\{G_{s}\mid s\in S_{r},r\in R\}\subseteq\mathcal{O}(X)$. Since $\bigcap\{\lozenge V\mid V\in\mathcal{U}\}$ is Scott open in $\mathcal{O}(P_{H}(X))$, there are finitely many open sets $W_{1}, W_{2},\cdot\cdot\cdot,W_{l}$ of $X$ such that $\bigcap\limits^{l}\limits_{i=1}\diamond W_{i}\subseteq \mathfrak{A}$ and $$\bigcap\limits^{l}\limits_{i=1}\diamond W_{i}\in\bigcap\{\lozenge V\mid V\in\mathcal{U}\}.$$ Assume that there is $l^{\ast}\in\{1,2,...,l\}$ such that $$W_{l^{\ast}}\not\in\bigcap\{\lozenge \dn u_{kj}\mid 1\leq j\leq m_{k}\}$$ for each $1\leq k\leq n$. Then there is $1\leq l_{k}\leq m_{k}$ satisfying $u_{kl_{k}}\not\in W_{l^{\ast}}$ for each $1\leq k\leq n$. Let $V^{\ast}=\Box(\bigcup\limits^{n}\limits_{k=1}\dn u_{kl_{k}})\in\mathcal{U}$. Then $\bigcap\limits^{l}\limits_{i=1}\diamond W_{i}\in\lozenge V^{\ast}$. So there is a $C\in\Gamma(X)$ such that $C\subseteq\bigcup\limits^{n}\limits_{k=1}\dn u_{kl_{k}}$ and $C\cap W_{l^{\ast}}\neq\emptyset$. This contradicts $u_{kl_{k}}\not\in W_{l^{\ast}}$, for each $1\leq k\leq n$. Then we have that for each $1\leq i\leq l$, there is a $1\leq k\leq n$ such that $$W_{i}\in\bigcap\{\lozenge \dn u_{kj}\mid 1\leq j\leq m_{k}\}.$$ For each $1\leq k\leq n$, let $$E_{k}=U_{k}\cap\{W_{i}\mid 1\leq i\leq l,W_{i}\in\bigcap\{\lozenge \dn u_{kj}\mid 1\leq j\leq m_{k}\}\}.$$ It is not difficult to check that $\bigcap\limits^{n}\limits_{k=1}\diamond E_{k}\subseteq\bigcap\limits^{l}\limits_{i=1}\diamond W_{i}$. Furthermore, $$E_{k}\in \bigcap\{\lozenge \dn u_{kj}\mid 1\leq j\leq m_{k}\}\subseteq\mathcal{A}_{k},$$ for each $1\leq k\leq n$. Then we have $$(E_{1},E_{2},\cdot\cdot\cdot,E_{n})\in\mathcal{A}_{1}\times\mathcal{A}_{2}\times\cdot\cdot\cdot\times\mathcal{A}_{n}\subseteq\beta_{n}^{-1}(\mathcal{F}).$$  Whence, $\bigcap\limits^{n}\limits_{k=1}\diamond E_{k}=\beta_{n}(E_{1},E_{2},...,E_{n})\in\mathcal{F}$. Since $\mathcal{F}$ is Scott open, it follows from $\bigcap\limits^{n}\limits_{k=1}\diamond E_{k}\subseteq\bigcap\limits^{l}\limits_{i=1}\diamond W_{i}\subseteq\mathfrak{A}$ that $\mathfrak{A}\in\mathcal{F}$. So the Claim 2 is proved.

Therefore, $P_{H}(X)$ is co-consonant.\end{proof}

\begin{example}{\rm  Let $P=\{a_{n}\mid n\in\mathbb{N}\}\cup\{b_{n}\mid n\in\mathbb{N}\}\cup\{\infty,b\}$. The order $\leq$ on $P$ is given by

(1) $a_{0}\leq a_{1}\leq a_{2}\leq\cdot\cdot\cdot\leq a_{n}\leq\cdot\cdot\cdot\leq\infty$, for all $n\in\mathbf{N}$;

(2) $a_{n}\leq b_{n}$, for all $n\in\mathbf{N}$;

(3) $b\leq\infty$ and $b\leq b_{n}$ , for all $n\in\mathbf{N}$.

\begin{center}
\centering
\includegraphics[totalheight=2.4in]{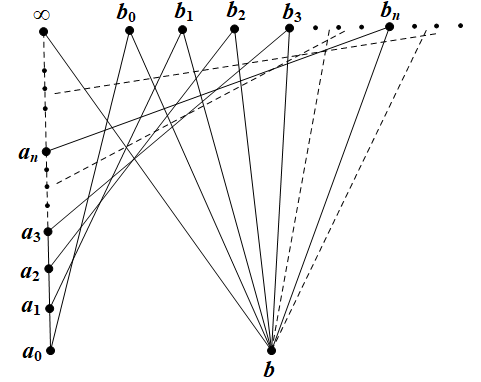}\\Figure 1
\end{center}

Then $P$ is a dcpo but not a continuous dcpo. Furthermore, $(P,\upsilon(P))$ is locally hypercompact and compact. Hence, by Proposition 3.10, $(P,\upsilon(P))$ is $\upsilon$-compact and co-consonant.

}
\end{example}

\quad In the following, we present an example which is compact but not $\upsilon$-compact. The following Lemma is useful and it has been proved in \cite{GOY03}.

\begin{lemma} (see \cite{GOY03}) Let $P$ be a poset. Then the following statements are equivalent:

{\rm (1)} $(P,\omega(P))$ is compact;

{\rm (2)} Every nonempty subset $A\subseteq P$  has an upper bound in $P$ if every finite
subset of $A$ has an upper bound in $P$.

\end{lemma}

\begin{theorem} Let $P$ be a poset. Then $(P,\upsilon(P))$ is $\upsilon$-compact iff $P=\up F$ for some finite subset $F$ of $P$.

\end{theorem}

\begin{proof} Sufficiency: Let $F_{1}={\rm min(F)}$. Then $P=\up F_{1}$ and ${\rm min(P)}=F_{1}$ is finite.  Let $\mathcal{U}\subseteq\upsilon(P)$ be an irreducible open cover of $P$ with respect to the upper topology of $\upsilon(P)$. Assume that $P\not\in\mathcal{U}$. Then for every $U\in\mathcal{U}$, $U\subseteq P\setminus\dn x=P\setminus\{x\}$ for some $x\in F_{1}$. Equivalently, $$\mathcal{U}\subseteq\bigcup\limits_{x\in F_{1}}\dn_{\upsilon(P)}P\setminus\dn x.$$
By the irreduciblity of $\mathcal{U}$, $\mathcal{U}\subseteq\dn_{\upsilon(P)}P\setminus\dn x_{0}$ for some $x_{0}\in F_{1}$. This implies $\bigcup\mathcal{U}\subseteq P\setminus\dn x_{0}$, a contradiction. Thus $P\in\mathcal{U}$ and then $(P,\upsilon(P))$ is $\upsilon$-compact.

Necessity: Since $(P,\upsilon(P))$ is $\upsilon$-compact, $(P,\upsilon(P))$ is compact. Then by Lemma 2.2, $P=\up {\rm min(P)}$ and ${\rm min(P)}$ is compact in $(P,\upsilon(P))$. Assume that ${\rm min(P)}$ is infinite. Consider $$\mathcal{S}=\{P\setminus\dn a\mid a\in {\rm min(P)}\}.$$ Clearly, $\mathcal{S}\subseteq\upsilon(P)$ and $\bigcup\mathcal{S}=P$. If $P\not\in cl_{\upsilon(\upsilon(P))}(\mathcal{S})$, then there exist $\{\mathcal{F}_{i}\mid i\in I\}\subseteq \mathbf{Fin}(\upsilon(P))$ such that $$cl_{\upsilon(\upsilon(P))}(\mathcal{S})=\bigcap\limits_{i\in I}\dn_{\upsilon(P)}\mathcal{F}_{i}.$$ It follows that $\mathcal{S}\subseteq\dn_{\upsilon(P)}\mathcal{F}_{i}$, for all $i\in I$.
Note that every element of $\mathcal{S}$ is maximal in $\upsilon(P)\setminus\{P\}$. Since ${\rm min(P)}$ is infinite and $\mathcal{F}_{i}$ is finite, $P\in\mathcal{F}_{i}$ for all $i\in I$. Whence, $$cl_{\upsilon(\upsilon(P))}(\mathcal{S})=\bigcap\limits_{i\in I}\dn_{\upsilon(P)}\mathcal{F}_{i}=\dn_{\upsilon(P)} P=\upsilon(P).$$ This means that $\mathcal{S}$ is irreducible in the upper topology of $\upsilon(P)$. As
$(P,\upsilon(P))$ is $\upsilon$-compact, $P\in\mathcal{S}$. It is a contradiction. Thus, ${\rm min(P)}$ is finite.
\end{proof}

\begin{example}{\rm  Let $P=\mathbb{N}\cup\{\infty\}$, where $\mathbb{N}$ is the natural numbers set. The order $\leq$ on $P$ is given by
$$\forall\ x\in P, x\leq\infty.$$
Let $Y=(P,\upsilon(P))$. Then $Y$ is compact and locally compact sober. Thus $\Sigma\mathcal{O}(Y)$ is sober and by Lemma 3.5, $Y$ is $SI$-compact. Clearly, we have $$\upsilon(P)=\{\emptyset,P\}\cup\{P\setminus F\mid F\ \mbox{is a nonempty finite subset of}\ \mathbb{N}\}.$$ So every directed subset of $\mathcal{O}(Y)$ contains a largest element. This means that the Scott topology agrees with the Alexandroff topology on $\mathcal{O}(Y)$. However, the upper topology on $\mathcal{O}(Y)$ is strictly less than the Alexandroff topology. Hence, $Y$ is not co-consonant. Note that ${\rm min(P)}=\mathbb{N}$ is infinite. Hence by Theorem 3.18, $Y$ is not $\upsilon$-compact.}
\end{example}

\quad Unfortunately, $\upsilon$-compactness does not hold for all infinite $T_{1}$ spaces. See the following Proposition 3.20.

\begin{proposition} Let $X$ be a $T_{1}$ space. If $X$ is $\upsilon$-compact, then $X$ is finite.
\end{proposition}

\begin{proof}Assume that $X$ is infinite. Let $\mathcal{U}=\{X\setminus\{x\}\mid x\in X\}$. Then $\mathcal{U}$ covers $X$. If $cl_{\upsilon(\mathcal{O}(X))}(\mathcal{U})=\bigcap\limits_{i\in I}\dn \mathcal{F}_{i}$ for some $\{ \mathcal{F}_{i}\mid i\in I\}\subseteq \mathbf{Fin}(\mathcal{O}(X))$, then $\mathcal{U}\subseteq\dn \mathcal{F}_{i}$. Since $X$ is infinite and every member of $\mathcal{U}$ is maximal in $\mathcal{O}(X)\setminus\{X\}$, $X\in\bigcap\limits_{i\in I}\mathcal{F}_{i}$. Hence, we have $$cl_{\upsilon(\mathcal{O}(X))}(\mathcal{U})=\bigcap\limits_{i\in I}\dn \mathcal{F}_{i}=\mathcal{O}(X).$$ This means that $\mathcal{U}$ is irreducible in the upper topology of $\mathcal{O}(X)$. This contradicts the $\upsilon$-compactness of $X$. The finiteness of $X$ is obtained.\end{proof}

\section{A new method for getting Scott non sober frames}

\quad In the following, we consider the $SI$-compactness of the lower topology, Scott topology, and Lawson topology of posets. Using the above conclusions, we obtain some new frames which are non sober for the Scott topology.

\begin{lemma} (see \cite{GG03}) Let $L$ be a complete lattice. If $\Sigma (L\times L)=\Sigma L\times\Sigma L$, then $\Sigma L$ is sober.
\end{lemma}

By Lemma 3.5 and Lemma 4.1, we have have the following Corollary 4.2 immediately.

\begin{corollary} Let $P$ be a poset. If $(P,\upsilon(P))$ is compact and $$\Sigma (\upsilon(P)\times \upsilon(P))=\Sigma \upsilon(P)\times\Sigma \upsilon(P),$$ then $(P,\upsilon(P))$ is $SI$-compact.

\end{corollary}

\quad The following Theorem 4.3 is is useful for proving the compactness and non $SI$-compactness of some spaces. Specifically, an infinite countable example will be given using the lower topology.

\begin{theorem} Let $L$ be a dcpo. If ${\rm Max(L)}$ is irreducible in $\Sigma L$, then $(L,\omega(L))$ is $SI$-compact iff ${\rm Max(L)}$ is finite.
\end{theorem}
\begin{proof}Sufficiency: Straightforward using Theorem 3.18.

Necessity: Assume that ${\rm Max(L)}$ is infinite. Consider $\mathcal{S}=\{L\setminus\up x\mid x\in {\rm Max(L)}\}$. Clearly, $\mathcal{S}\subseteq\omega(L)$ and $\bigcup \mathcal{S}= L$. Define a map $f:L\longrightarrow \omega(L)$ as follows:
$$\forall\ x\in L, f(x)=L\setminus\up x.$$
Let $D\subseteq L$ be a directed subset. Then we have $$\bigvee_{\omega(L)} f(D)=\bigcup\limits_{d\in D}L\setminus\up d=L\setminus\bigcap\limits_{d\in D}\up d=L\setminus\up \bigvee D=f(\bigvee D).$$
Thus $f$ is Scott continuous.
So $f({\rm Max(L)})=\mathcal{S}$ is irreducible in $\Sigma\omega(L)$. As $(L,\omega(L))$ is $SI$-compact, $L\in\mathcal{S}$, a contradiction. Thus ${\rm Max(L)}$ is finite.\end{proof}

\begin{corollary}Let $L$ be a dcpo with $\omega\leq|{\rm Max(L)}|$. If ${\rm Max(L)}$ is irreducible in $\Sigma L$ and $(L,\omega(L))$ is compact, then $\Sigma\omega(L)$ is non-sober.
\end{corollary}

\begin{proof}By Lemma 3.5 and Theorem 4.3, $\Sigma\omega(L)$ is non-sober.\end{proof}

\quad In the following, we present two new examples which are compact but not $SI$-compact by using the lower topology.

\begin{example}\setlength{\baselineskip}{1.3\baselineskip}
{\rm (see \cite{BZ35}) Let $P_{0}=(\mathbb{N}\times\mathbb{N})\cup\{e_{k}\mid k\in\mathbb{N}\}\cup\{(n,f)\mid n\in\mathbb{N}, f\in\mathbb{N}^{\mathbb{N}}\}\cup\{\Omega\}$ where $\mathbb{N}^{\mathbb{N}}$ denotes the set of all maps form $\mathbb{N}$ to $\mathbb{N}$. For convenience, we use $m_{k}$ ($n_{f}$) to denote the element $(m,k)$ ($(n,f)$) coming form $\mathbb{N}\times\mathbb{N}$ ($\{(n,f)\mid n\in\mathbb{N}, f\in\mathbb{N}^{\mathbb{N}}\}$). The order $\leq$ on $P_{0}$ is defined as follows:

(1) $e_{0}\leq e_{1}\leq e_{2}\leq\cdot\cdot\cdot \leq e_{k}\leq e_{k+1}\leq\cdot\cdot\cdot\leq \Omega$, for all $k\in\mathbb{N}$;

(2) $0_{k}\leq 1_{k}\leq 2_{k}\leq\cdot\cdot\cdot\leq m_{k}\leq (m+1)_{k}\leq\cdot\cdot\cdot\leq e_{k}$, for all $k,m\in\mathbb{N}$;

(3) $0_{f}\leq 1_{f}\leq 2_{f}\leq\cdot\cdot\cdot\leq m_{f}\leq (m+1)_{f}\leq\cdot\cdot\cdot\leq \Omega$, for all $m\in\mathbb{N},f\in\mathbb{N}^{\mathbb{N}}$.

The poset $P_{0}$ is also called a gadget. See figure 2.

\begin{center}
\centering
\includegraphics[totalheight=2.4in]{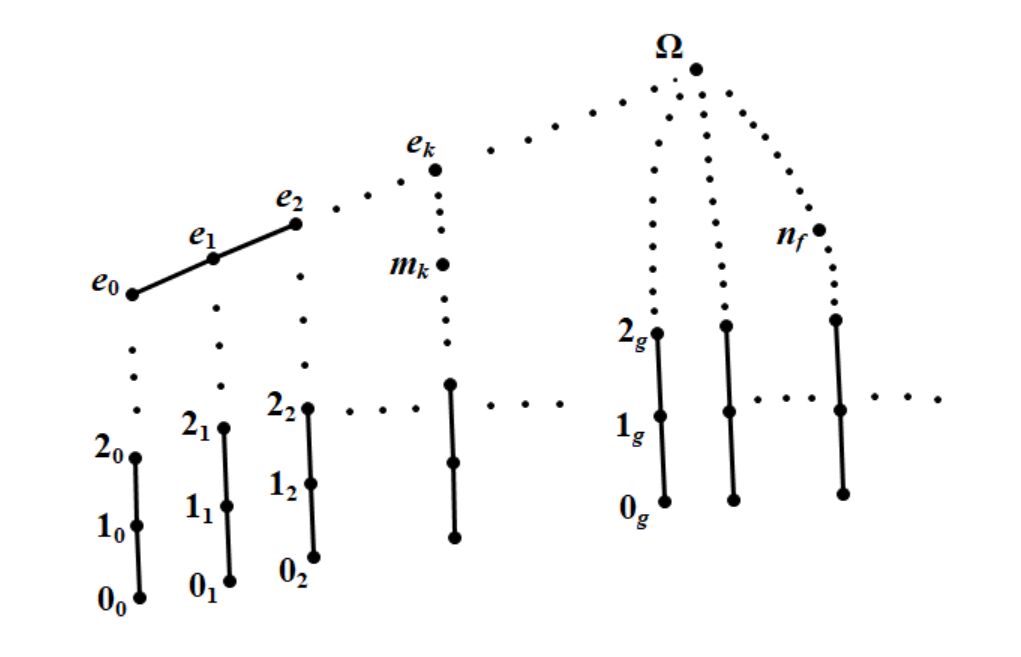}\\Figure 2 The gadget $P_{0}$
\end{center}

 The dcpo $P_{1}$ is obtained by putting $\mathbf{c}$ many copies of gadget $P_{0}$ side by side, where $\mathbf{c}$ is the cardinality of the set $\mathbb{R}$ of all real numbers. Let $i:\mathbb{R}^{+}\times\mathbb{R}^{+}\times\mathbb{N}^{\mathbb{N}}\times\mathbb{N}\rightarrow\mathbb{R}^{+}$ be an injection. Those copies will be indexed by positive real numbers $\alpha,\beta$. Next, we assemble the gadgets $\alpha$ and $\beta$ $(\alpha\neq\beta)$. The elements coming form gadget $\alpha$ and those of the gadget $\beta$ are incomparable except for the following new relations. For every $m_{k}$ coming form gadget $\alpha$ and every element $n_{f}$ in gadget $\beta$ with $m=f(k), n=k$, there is the top element $\Omega$ in the gadget $\gamma$ such that $m_{k},n_{f}\leq\Omega$ where the gadget $\gamma$ is the image of an injective function $i$ at $(\alpha,\beta,f,k)$. Please see the following Figure 3.

\begin{center}
\centering
\includegraphics[totalheight=2.4in]{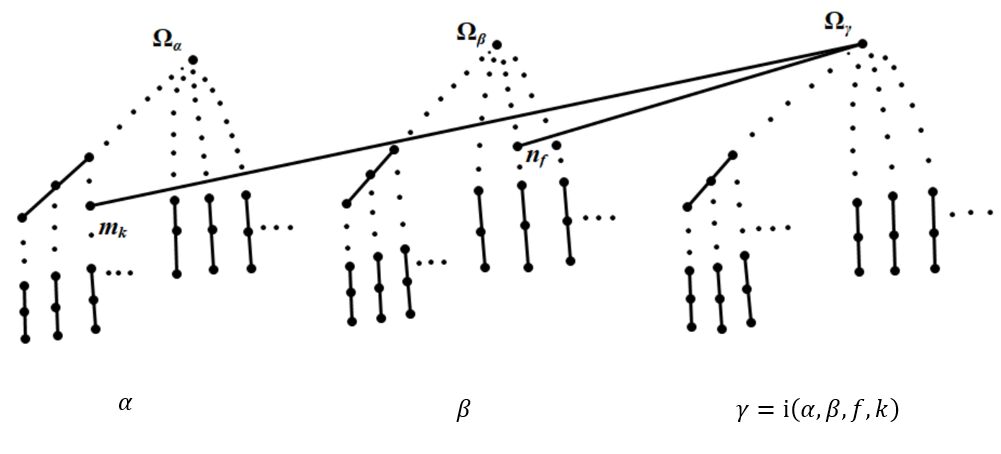}\\Figure 3 The gadget $\alpha$, $\beta$ and $\gamma=i(\alpha,\beta,f,k)$ with extra order
\end{center}
We take the dcpo $P_{\ast}$ constructed by assembling $\mathbf{c}$ many copies of gadget $P_{1}$ with the extra relations shown in figure 3.  The Dedekind-MacNeille completion of $P_{\ast}$ is just the Isbell complete lattice (see \cite{IJ82}).

Claim 1: ${\rm Max(P_{\ast})}=\{\Omega_{\alpha}\mid \alpha \in\mathbb{R}^{+}\}$ is infinite.

Claim 2: $P_{\ast}$ is irreducible in $\Sigma P_{\ast}$.

For the sake of completeness, we give a brief proof of this claim (or can see \cite{IJ82}).

If $U,V\in\sigma(P_{\ast})$ with $U\cap P_{\ast}\neq\emptyset$, and $V\cap P_{\ast}\neq\emptyset$, then there are reals $\alpha$ and $\beta$ such that $\Omega_{\alpha}\in U$ and $\Omega_{\beta}\in V$.
Take $D=\{e_{k}\mid k\in\mathbb{N}\}$ in gadget $\alpha$. Then $D$ is directed and $\bigvee D=\Omega_{\alpha}\in U$. As $U\in\sigma(P)$, there is a $k_{0}\in\mathbb{N}$ such that $e_{k}\in U$ for every $k_{0}\leq k$. For every  $k_{0}\leq k$, let $D_{k}=\{m_{k}\mid m\in\mathbb{N}\}$ in gadget $\alpha$. Similarly, $\bigvee D_{k}=e_{k}\in U$ and $D_{k}$ is directed. Then there is a $t_{k}\in \mathbb{N}$ such that $m_{t_{k}}\in U$. Define a map $f_{0}:\mathbb{N}\longrightarrow\mathbb{N}$ as follows:
$$\forall \ k\in\mathbb{N}, \ f_{0}(k)=\left\{
             \begin{array}{ll}
              t_{k}, &\ \ \  k_{0}\leq k, \\
               0, &\ \ \ 0\leq k< k_{0}.
             \end{array}
           \right.$$
Consider $D_{\ast}=\{n_{f_{0}}\mid n\in\mathbb{N}\}$ in gadget $\beta$. Then $\bigvee D^{\ast}=\Omega_{\beta}\in V$. Thus there is a $\hat{n}\in \mathbb{N}$ satisfying $n_{f_{0}}\in V$ for all $\hat{n}\leq n$. Choose a $n_{1}=k_{1}=max\{k_{0},\hat{n}\}$ and $m_{1}=f_{0}(n_{1})=t_{k_{1}}$. Using the extra relations, there is a $\Omega_{\gamma}$ in gadget $\gamma$ such that $(m_{1})_{k_{1}},(n_{1})_{f_{0}}\leq\Omega_{\gamma}$. Note that
$(m_{1})_{k_{1}}\in U,(n_{1})_{f_{0}}\in V$ and $U,V\in\sigma(P_{\ast})$. This implies $\Omega_{\gamma}\in U\cap V$. Thus claim 2 holds.

Claim 3: $(P_{\ast},\omega(P_{\ast}))$ is compact.

Suppose $A$ is a nonempty subset such that every finite subset of $A$ has an upper bound in $P_{\ast}$. If $A\cap {\rm max}(P_{\ast})\neq\emptyset$, then $A\cap {\rm max}(P_{\ast})$ is a singleton, namely $\Omega_{\alpha}$. Furthermore, we can see $A\subseteq\dn \Omega_{\alpha}$. Equivalently, $\Omega_{\alpha}$ is an upper bound of $A$ in $P_{\ast}$. If $A\cap {\rm max}(P_{\ast})=\emptyset$, then $A$ is contained in at most two different gadgets $\mu$ and $\nu$. Otherwise, $A$ meets three different gadgets $\alpha_{1},\alpha_{2}$ and $\alpha_{3}$ at least, for $i=1,2,3$, let $\{x_{i}\}$ be an element of $A$ which belongs to the gadget $\alpha_{i}$. However, by the order of $P^{\ast}$, $\{x_{1},x_{2},x_{3}\}$ has no upper bound, impossible. So we just consider the case that $A$ meets gadget $\mu$,$\nu$ with $\mu\neq\nu$ and $A\cap {\rm max}(P_{\ast})=\emptyset$. There are 5 cases in total.

case 1: $A\cap\{e_{k}\mid k\in\mathbb{N}\}\neq\emptyset$ in gadget $\mu$ or gadget $\nu$.

Suppose $e_{m}\in A$ in gadget $\mu$. Take an $x\in A$ in gadget $\nu$ with $x\neq\Omega_{\nu}$. Clearly, $\{e_{m},x\}$ has no upper bounds in $P_{\ast}$, impossible.

case 2: $A\cap(\mathbb{N}\times\mathbb{N})\neq\emptyset$ in gadget $\mu$ and gadget $\nu$.

Choose $a,b\in A\cap(\mathbb{N}\times\mathbb{N})$ in gadget $\mu$ and gadget $\nu$ respectively. Clearly, $\{a,b\}$ has no upper bounds in $P_{\ast}$, impossible.

case 3: $A\cap\{(n,f)\mid n\in\mathbb{N}, f\in\mathbb{N}^{\mathbb{N}}\}\neq\emptyset$ in gadget $\mu$ and gadget $\nu$.

It is also impossible.

case 4: $A\cap(\mathbb{N}\times\mathbb{N})\neq\emptyset$ in gadget $\mu$ and $A\cap\{(n,f)\mid n\in\mathbb{N}, f\in\mathbb{N}^{\mathbb{N}}\}\neq\emptyset$ in gadget $\nu$.

subcase (1): $|\{k\in\mathbb{N}\mid \exists m\in\mathbb{N}, m_{k}\in A \ \mbox{in gadget}\ \mu\}|\geq 2$.

Take $m_{k},\hat{m}_{\hat{k}}\in A$ in gadget $\mu$ with $k\neq\hat{k}$ and a $s_{g}\in A$ in gadget $\nu$ with $g\in\mathbb{N}^{\mathbb{N}}$. Then $s_{g}$ cannot be equal to $\Omega_{\nu}$. Obviously, $\{m_{k},\hat{m}_{\hat{k}},s_{g}\}$ has no upper bounds in $P_{\ast}$, impossible.

subcase (2): $|\{h\in\mathbb{N}^{\mathbb{N}}\mid \exists t\in\mathbb{N}, t_{h}\in A \ \mbox{in gadget}\ \nu\}|\geq 2$.

Take $p_{h},\hat{p}_{\hat{h}}\in A$ in gadget $\nu$ with $h\neq\hat{h}$ and a $q_{n}\in A$ in gadget $\mu$ with $q_{n}\neq\Omega_{\mu}$. $\{p_{h},\hat{p}_{\hat{h}},q_{n}\}$ has no upper bounds in $P_{\ast}$, impossible.

subcase (3): $|\{k\in\mathbb{N}\mid \exists m\in\mathbb{N}, m_{k}\in A \ \mbox{in gadget}\ \mu\}|=1$ and $|\{h\in\mathbb{N}^{\mathbb{N}}\mid \exists t\in\mathbb{N}, t_{h}\in A \ \mbox{in gadget}\ \nu\}|=1$.

Suppose $\{k\in\mathbb{N}\mid \exists m\in\mathbb{N}, m_{k}\in A \ \mbox{in gadget}\ \mu\}=\{k_{\ast}\}$ and $\{h\in\mathbb{N}^{\mathbb{N}}\mid \exists t\in\mathbb{N}, t_{h}\in A \ \mbox{in gadget}\ \nu\}=\{f_{\ast}\}$.

Let $E=\{m\in\mathbb{N}\mid m_{k_{\ast}}\in A\ \mbox{in gadget}\ \mu\}$ and $F=\{n\in\mathbb{N}\mid n_{f_{\ast}}\in A\ \mbox{in gadget}\ \nu\}$.

Note that if $x\in\mathbb{N}\times\{k_{\ast}\}$ and $y\in\{s_{f_{\ast}}\mid s\in\mathbb{N}\}$, then $x$ and $y$ has an upper bound in $P_{\ast}$ iff $x=(f(k_{\ast}))_{k_{\ast}}$ and $y=(k_{\ast})_{f_{\ast}}$.
Thus, $E$ and $F$ are singletons. Whence, $A$ has 2 elements and $A$ has an upper bound $\Omega_{\varepsilon}$ in gadget $\varepsilon=i(\mu,\nu,k_{\ast},f_{\ast})$.

case 5: $A\cap(\mathbb{N}\times\mathbb{N})\neq\emptyset$ in gadget $\nu$ and $A\cap\{(n,f)\mid n\in\mathbb{N}, f\in\mathbb{N}^{\mathbb{N}}\}\neq\emptyset$ in gadget $\mu$.

It is similar to case 4.

All in all, $A$ has an upper bound in $P_{\ast}$. Then by Lemma 3.17, $(P_{\ast},\omega(P_{\ast}))$ is compact. Using the three claims above, Theorem 4.3 and Corollary 4.4, $(P_{\ast},\omega(P_{\ast}))$  is not $SI$-compact and $\Sigma \omega(P_{\ast})$ is non-sober.

}
\end{example}

\quad The other example which is similar to the topological space $(P^{\ast},\omega(P^{\ast}))$ is the lower topology on the dcpo $\widehat{P}$ constructed by Miao, Xi, Li and Zhao in \cite{AB75}. Furthermore, the Dedekind-MacNeille completion of $\widehat{P}$ is a countable complete lattice which is non sober for the Scott topology (see \cite{AB75}). First we need to recall to definition of the dcpo $\hat P$.

\begin{example}{\rm (see \cite{AB75})\setlength{\baselineskip}{1.3\baselineskip}
 Let $\mathbb{N}^{<\omega}$ be the set of all nonempty finite sequences over $\mathbb{N}$. The
order $\leq $ on $\mathbb{N}^{<\omega}$ is the usual prefix order.

For $s=x_{1}x_{2}\cdot\cdot\cdot x_{n}\in\mathbb{N}^{<\omega}$, the length of $s$ is denoted by $len(s)$ and hence $len(s)=n$.

Let $\{p_{1},p_{2},\cdot\cdot\cdot,p_{n},\cdot\cdot\cdot\}$ be the set of all prime numbers of $\mathbb{N}$.

(i) $\mathbb{N}$ can be expressed as the union of countably disjoint infinite subsets of $\mathbb{N}$.

Let $E_{1}=\{p_{n}\mid 1\leq n\}\cup\{1\}$ and for every $2\leq k$
$$E_{k}=\{p_{r_{1}}\times p_{r_{2}}\times\cdot\cdot\cdot\times p_{r_{k}}\mid 1\leq r_{1},r_{2},\cdot\cdot\cdot, r_{k}\in\mathbb{N}\}.$$
For example, $$E_{2}=\{2\times2,2\times3,2\times5,\cdot\cdot\cdot\},$$ $$E_{3}=\{2\times2\times2,2\times2\times3,2\times2\times5,\cdot\cdot\cdot\}.$$ Then we can see $\bigcup\limits_{n\in\mathbb{N}} E_{n}=\mathbb{N}$ and $E_{m}\cap E_{n}=\emptyset$ for all $m\neq n$.

(ii) There is a bijection $h:\mathbb{N}\times \mathbb{N}\longrightarrow\mathbb{N}$ such that for all $m,n\in\mathbb{N}$, $m<h(m,n)$ and $n<h(m,n)$.

For every $n\in\mathbb{N}$, $x_{1}x_{2}\cdot\cdot\cdot x_{n_{0}}$ is the binary notation of $n$, where $n_{0}\in\mathbb{N}$ and $x_{k}\in\{0,1\}$ for all $1\leq k\leq n_{0}$. The function $h$ is defined by interleaving the binary notations of the two parameters. For example $x_{1}=22$ (or written $x_{1}$ as 10110), $y_{1}=55$ (or written $y_{1}$ as 110111), $h(x_{1},y_{1})=101100111110$, in other words $h(x_{1},y_{1})=2878$.

(ii) For every $(m,n)\in\mathbb{N}\times \mathbb{N}$, let $i(m,n)=E_{h(m,n)}$. Then the map $i:\mathbb{N}\times\mathbb{N}\longrightarrow\mathcal{P}(\mathbb{N})$ is an injection. In addition, every $i(m,n)$ is infinite and $i(m,n)\cap i(m_{1},n_{1})=\emptyset$ for $(m,n)\neq(m_{1},n_{1})$.

(iii) There is a  monotone bijection $f:\mathbb{N}^{<\omega}\longrightarrow \mathbb{N}$.

Define a monotone bijection $f:\mathbb{N}^{<\omega}\longrightarrow \mathbb{N}$ as follows:
$$\forall\ x=a_{1}a_{2}\cdot\cdot\cdot a_{n}\in\mathbb{N}^{<\omega}, f(x)=p_{1}^{a_{1}}p_{2}^{a_{2}}\cdot\cdot\cdot p_{n}^{a_{n}}.$$
(iv) For every $m,n\in\mathbb{N}$, let $h_{m,n}:\mathbb{N}\longrightarrow i(m,n)$ be a monotone injection. Then $f_{m,n}=h_{m,n}\circ f$ is still a monotone injection.

Let $L=\mathbb{N}\cup\mathbb{N}^{<\omega}\cup\{\top\}$. For all $x,y\in L$, $x\leq y$ in $L$ if and only if one of the following cases hold

(1) $x\in L, y=\top$;

(2) $x,y\in \mathbb{N}$ and $x\leq y$ in the poset $\mathbb{N}$;

(3) $x,y\in \mathbb{N}^{<\omega}$ and $x\leq y$ in the poset $\mathbb{N}^{<\omega}$;

so that both $\mathbb{N}^{<\omega}$ and $\mathbb{N}$ are sub-posets of $L$ and $\top$ is its top element.

Let $\widehat{P} = \mathbb{N }\times L$ and $L_{n} = \{(n, x)\in \widehat{P}\mid x\in L\}$. Define the four irreflexive relations $<_{1}, <_{2}, <_{3}$ and $<_{4}$ on $\widehat{P}$ as follows:

(1) $(n, x) <_{1} (m, y)$ if $n = m$ and $x < y$ in $L$;

(2) $(n, x) <_{2 }(m, y)$ if $y = \top, x \in \mathbb{N}^{<\omega}$ and there exists a unique $k\in \mathbb{N}$ with $k> n$ such that $m = f_{n,k}(x)$ and therefore $m \in i(n, k)$.

(3) $(n, x) <_{3} (m, y)$ if $y = \top, x \in \mathbb{N}$ and there exists a unique $d \in \mathbb{N}$ with $d < n$ such that $m = f_{d,n}(x)$ and therefore $m\in i(d, n)$.

(4) $(n, x) <_{4} (m, y)$ if $y = \top, x \in\mathbb{N} $ and there exists $a, b \in \mathbb{N}, s \in \mathbb{N}^{<\omega}$ with $a < b$ such that $f_{a,b}(s) =
n$ and $f_{a,b}(s.x) = m$.

\begin{center}
\centering
\includegraphics[totalheight=1.7in]{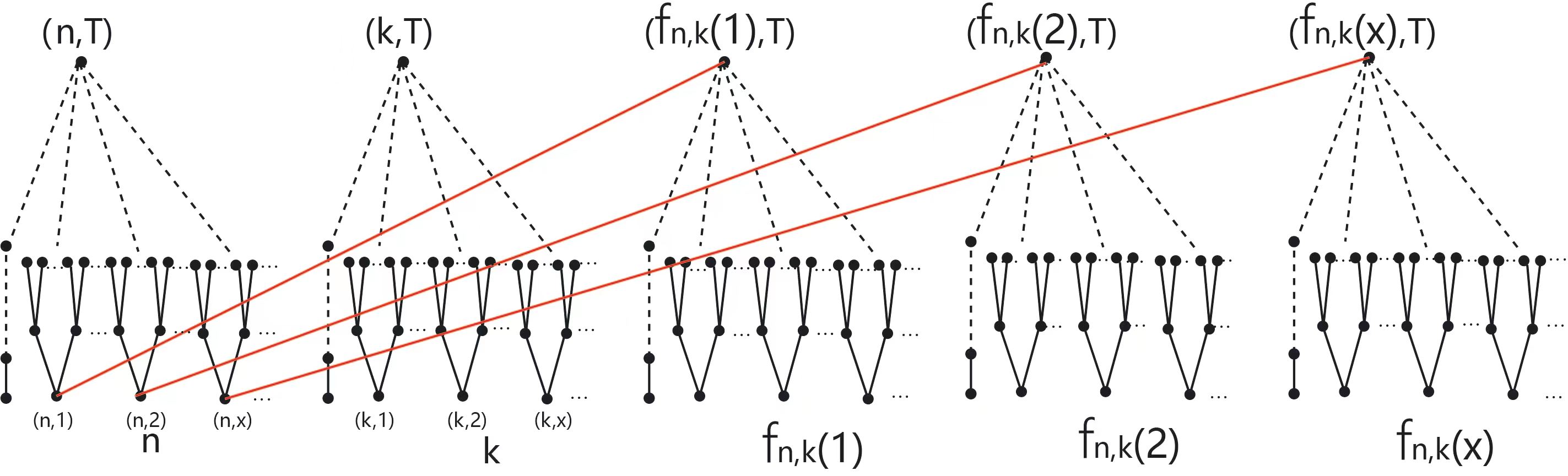}\\Figure 4 The strict order $<_{2}$.
\end{center}

\begin{center}
\centering
\includegraphics[totalheight=1.7in]{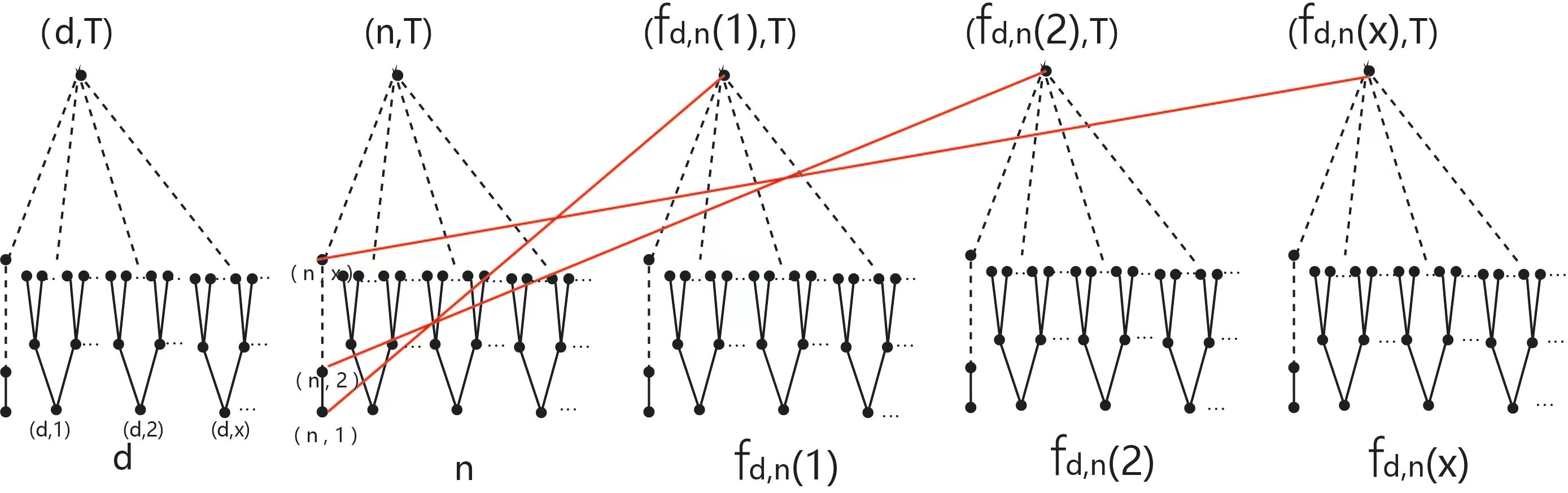}\\Figure 5 The strict order $<_{3}$.
\end{center}

\begin{center}
\centering
\includegraphics[totalheight=1.7in]{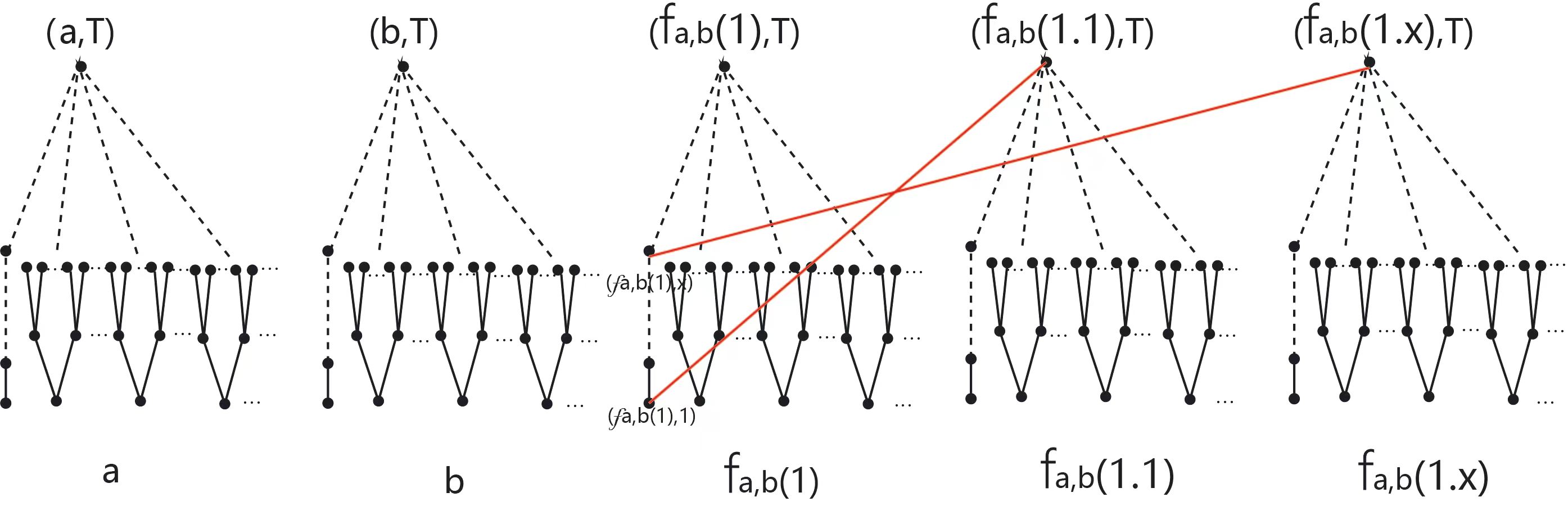}\\Figure 6 The strict order $<_{4}$.
\end{center}

\begin{center}
\centering
\includegraphics[totalheight=1.7in]{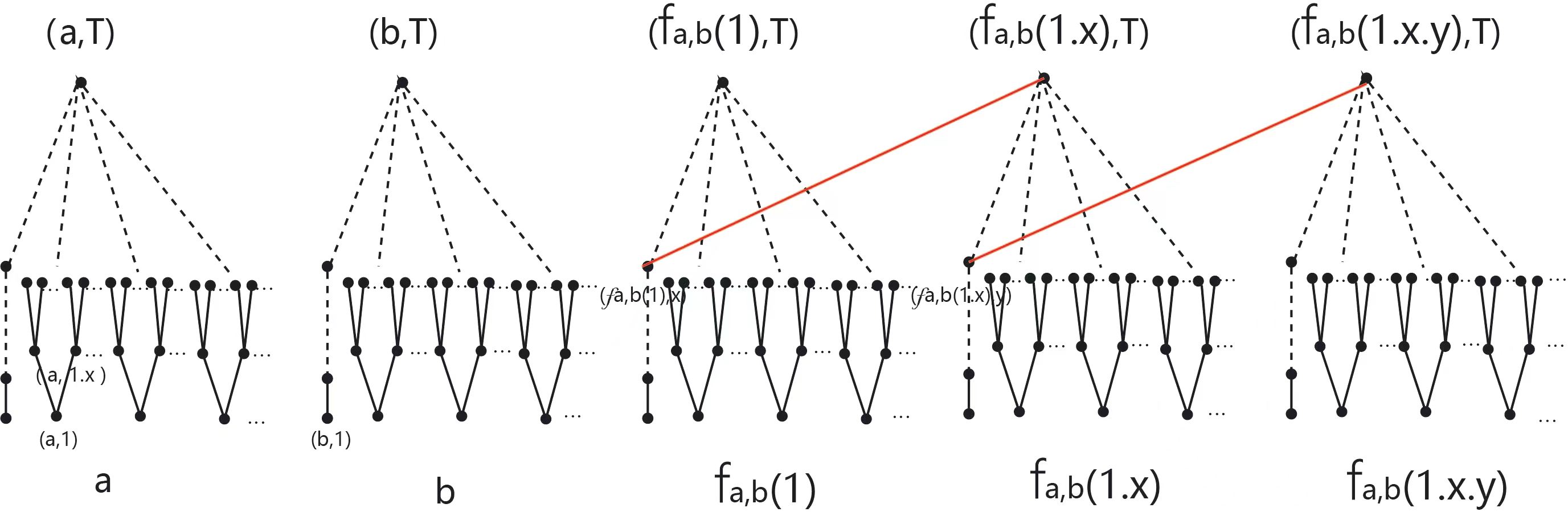}\\Figure 7 The strict order $<_{4}$.
\end{center}
The order $\leq$ on $\widehat{P}$ is $\leq_{1}\cup(\leq_{1}\circ<{2} )\cup(\leq_{1}\circ<_{3})\cup(\leq_{1}\circ<_{4})$. Then $(\widehat{P},\leq)$ is a dcpo. We use $\prec_{i}$ to denote the strict order obtained  form $\leq_{1}\circ<_{i}$, $i=2,3,4$.

Observe the following facts, which are useful in the sequel

(1) For all $a,b\in\mathbb{N},s\in\mathbb{N}^{<\omega}$ if $a<b$, then $a<f_{a,b}(s)$ and $b<f_{a,b}(s)$;

(2) If $x_{1}x_{2}\cdot\cdot\cdot x_{n}< y_{1}y_{2}\cdot\cdot\cdot y_{n}$ in $\mathbb{N}^{<\omega}$, then $f_{a,b}(x_{1}x_{2}\cdot\cdot\cdot x_{n})<f_{a,b}(y_{1}y_{2}\cdot\cdot\cdot y_{n})$ for all $a,b\in \mathbb{N}$ with $a<b$;

(3) For all $n\in\mathbb{N}$, the sets $$\{m\mid \exists t \ \mbox{such that}\ \ (m,t)\in\dn(n,\top)\cap(\mathbb{N}\times\mathbb{N})\}$$ and $$\{m\mid \exists s\ \mbox{such that}\ \ (m,s)\in\dn(n,\top)\cap(\mathbb{N}\times\mathbb{N}^{<\omega})\}$$ have at most two different elements;

(4) For every $n\in\mathbb{N},s\in\mathbb{N}^{<\omega}$, $(n,s)\prec_{2}(p,\top)$ iff there are $n<k$ and $s\leq s_{1}$ in $\mathbb{N}^{<\omega}$, such that  $p=f_{n,k}(s_{1})$;

(5) For every $n,m\in\mathbb{N}$, $(n,m)\prec_{3}(q,\top)$ iff there are $d<n$ and $m\leq m_{1}$, such that $q=f_{d,n}(m_{1})$;

(6) For every $n,m\in\mathbb{N}$ with $n=f_{a,b}(s)$ for some $a,b\in\mathbb{N}$ with $a<b$ and $s\in\mathbb{N}^{<\omega}$, $(n,m)\prec_{4}(r,\top)$ iff there are a $m\leq m_{1}\in\mathbb{N}$ such that $r=f_{a,b}(sm_{1})$;

(7) For every $a,b,c,d\in\mathbb{N}$ with $a<b$, $c<d$ and $s,s_{1}\in\mathbb{N}^{<\omega}$, $f_{a,b}(s)=f_{c,d}(s_{1})$ iff $a=c,b=d,s=s_{1}$. Furthermore, $\mathbb{N}=\{f_{a,b}(s)\mid a,b\in\mathbb{N}, s\in\mathbb{N}^{<\omega}\}$;

(8) $\prec_{2},\prec_{3},\prec_{4}$ are disjoint as sets of pairs.}
\end{example}

\begin{lemma} (See \cite{AB75}) $\widehat{P}$ is irreducible in $\Sigma\widehat{P}$.

\end{lemma}

\begin{lemma} Let $A\subseteq\widehat{P}$ be a subset such that every finite subset of $A$ has an upper bound. Then $A$ meet at most 3 different sub-posets $L_{n}$.
\end{lemma}
\begin{proof}If $A\cap {\rm max}(\widehat{P})\neq\emptyset$, then we can deduce that $A\cap {\rm max}(\widehat{P})$ is a singleton $(n_{0},\top)$ and $A\subseteq\dn (n_{0},\top)$. If $A\cap {\rm max}(\widehat{P})=\emptyset$, we consider $M=\{n\in\mathbb{N}\mid \exists x\in L\setminus\{\top\}, (n,x)\in A\}$. Clearly, $A$ has an upper bound if $|M|=1$ or $A$ is finite.

Next, we prove that $|M|\leq 3$.

Suppose $4\leq|M|$ and take $(\overline{n}_{i},\overline{t}_{i})\in A$, $i=1,2,3,4$ with $\{\overline{n}_{1},\overline{n}_{2},\overline{n}_{3},\overline{n}_{4}\}\subseteq M$. Then $\{(\overline{n}_{1},\overline{t}_{1}),(\overline{n}_{2},\overline{t}_{2}),(\overline{n}_{3},\overline{t}_{3}),(\overline{n}_{4},\overline{t}_{4})\}$
has an upper bound $a^{u}$. By the order of $\widehat{P}$, we can see $a^{u}\in {\rm max(\widehat{P})}$. Therefore $a^{u}=(f_{a^{\diamond},b^{\diamond}}(s^{\diamond}),\top)$ for some $s\in\mathbb{N}^{<\omega}$ and $a^{\diamond},b^{\diamond}\in\mathbb{N}$. For convenience, let $s^{\diamond}=a'_{1}a'_{2}\cdot\cdot\cdot a'_{n^{\diamond}}$. If $len(s^{\diamond})=1$, then
$\dn a^{u}=\dn(f_{a^{\diamond},b^{\diamond}}(a'_{1}),\top)=L_{f_{a^{\diamond},b^{\diamond}}(a'_{1})}\cup\{(a,a'_{1})\}\cup\{(b,a)\mid 1\leq a\leq a'_{1}\}$. Now, we can see the first projection of $\dn a^{u}$ has 3 natural numbers, a contradiction. If $1<len(s^{\diamond})$, then $\dn a^{u}=\dn(f_{a^{\diamond},b^{\diamond}}(a'_{1}a'_{2}\cdot\cdot\cdot a'_{n^{\diamond}}),\top)=L_{f_{a^{\diamond},b^{\diamond}}(a'_{1}a'_{2}\cdot\cdot\cdot a'_{n^{\diamond}})}\cup\{(a,a'_{1}a'_{2}\cdot\cdot\cdot a'_{m^{\diamond}})\mid 1\leq m^{\diamond}\leq n^{\diamond}\}\cup\{(f_{a^{\diamond},b^{\diamond}}(a'_{1}a'_{2}\cdot\cdot\cdot a'_{n^{\diamond}-1}),a)\mid 1\leq a\leq a'_{n^{\diamond}}\}$. Similarly, the first projection of $\dn a^{u}$ has 3 natural numbers, a contradiction. So we can conclude that $|M|\leq3$.\end{proof}

\begin{lemma} {\rm Let $A\subseteq\widehat{P}$ be subset such that

(1) every finite subset of $A$ has an upper bound.

(2) $A$ just meet $L_{n_{0}}$ and  $L_{n_{1}}$ nonempty.

Then $A$ has an upper bound.}

\end{lemma}

\begin{proof}We just consider the case $A\cap {\rm max}(\widehat{P})=\emptyset$ and $A$ is infinite.

We assume that $A\cap L_{n_{0}}$ is infinite and $A\cap L_{n_{1}}\neq\emptyset$. There are four cases (1)-(4).

(1) $A\cap (\{n_{0}\}\times \mathbb{N})$ is infinite and $(n_{1},m_{1})\in A$ for some $m_{1}\in\mathbb{N}$.

Then $\{(n_{0},m_{0}),(n_{1},m_{1})\}$ has an upper bound in $\widehat{P}$, for some $m_{0}\in\mathbb{N}$ such that $(n_{0},m_{0})\in A$. By order relation $\prec_{2},\prec_{3},\prec_{4}$, we have the following 4 cases:

$(n_{0},m_{0})\prec_{3}(n_{1},\top)$;

$(n_{0},m_{0})\prec_{4}(n_{1},\top)$;

$(n_{1},m_{1})\prec_{3}(n_{0},\top)$;

$(n_{1},m_{1})\prec_{4}(n_{0},\top)$.

Since  $A\cap (\{n_{0}\}\times \mathbb{N})$ is infinite, the element $m_{0}$ has infinitely many possibilities. This means that we need only consider the cases $(n_{1},m_{1})\prec_{3}(n_{0},\top)$ or $(n_{1},m_{1})\prec_{4}(n_{0},\top)$.

Subcase 1: $(n_{1},m_{1})\prec_{3}(n_{0},\top)$.

Then $n_{0}=f_{\widetilde{d},n_{1}}(\overline{m}_{1})$ for some $\widetilde{d}<n_{1}$ and $m_{1}\leq\overline{m}_{1}$.

Next, we claim that $A\subseteq \dn (n_{0},\top)=\dn (f_{\widetilde{d},n_{1}}(\overline{m}_{1}),\top)$.

If on the contrary, $A\not\subseteq\dn (f_{\widetilde{d},n_{1}}(m_{1}),\top)$. Take a $(n_{2},m_{2})\in A\setminus\dn (f_{\widetilde{d},n_{1}}(m_{1}),\top)$. Note that
$$\dn (f_{\widetilde{d},n_{1}}(\overline{m}_{1}),\top)=L_{f_{\widetilde{d},n_{1}}(\overline{m}_{1})}\cup \{(n_{1},m)\mid 1\leq m\leq \overline{m}_{1}\}\cup\{(\widetilde{d},\overline{m}_{1})\},$$
where $(\widetilde{d},\overline{m}_{1})\in \{\widetilde{d}\}\times \mathbb{N}^{<\omega}$. There are only 2 possibilities ((1.1) and (1.2)).

(1.1) $n_{2}=n_{1}$ and $\overline{m}_{1}<m_{2}$. Then $\{(n_{1},m_{2}),(n_{0},m')\}$ has an upper bound $\overline{u}$ for every $(n_{0},m')\in A$. We further know that $\overline{u}\in\{(n_{0},\top),(n_{1},\top)\}$ for $n_{0}\neq n_{1}$. Suppose $\overline{u}=(n_{0},\top)=(f_{\widetilde{d},n_{1}}(\overline{m}_{1}),\top)$. Using strict order $\prec_{3}$, $(n_{1},m_{2})\prec_{3} (n_{0},\top)$. Then $n_{0}=f_{d^{0}n_{1}}(\overline{m}_{2})$ for some $m_{2}\leq \overline{m}_{2}$. Whence, $f_{d^{0}n_{1}}(\overline{m}_{2})=f_{\widetilde{d},n_{1}}(\overline{m}_{1})$ and then $\overline{m}_{1}=\overline{m}_{2}$. It ia a contradiction, impossible. Suppose $\overline{u}=(n_{1},\top)$. Then we know $(n_{0},m')\prec_{3}(n_{1},\top)$ or $(n_{0},m')\prec_{4}(n_{1},\top)$. If $(n_{0},m')\prec_{3}(n_{1},\top)$, $n_{1}=f_{d_{0}n_{0}}(\overline{m}')$ for some $m'<\overline{m}'$. So we know $n_{0}<f_{d_{0}n_{0}}(\overline{m}')=n_{1}$ and $n_{1}<f_{\widetilde{d},n_{1}}(\overline{m}_{1})=n_{0}$, impossible. If $(n_{0},m')\prec_{4}(n_{1},\top)$, combining $f_{\widetilde{d},n_{1}}(\overline{m}_{1})=n_{0}$, we have $n_{1}=f_{\widetilde{d},n_{1}}(m_{1}\overline{m}'')$ for some $m'<\overline{m}''$ . This is still impossible since $n_{1}=f_{\widetilde{d},n_{1}}(m_{1}m')>n_{1}$.

(1.2) $n_{2}=n_{1}$ and $m_{2}\in\mathbb{N}^{<\omega}$. then $\{(n_{1},m_{2}),(n_{0},\tilde{m}),(n_{1},m_{1})\}$ has an upper bound $\widetilde{u}$ for every $(n_{0},\tilde{m})\in A$. Using strict order $\prec_{2}$, $\widetilde{u}=(n_{1},\top)$ or $\widetilde{u}=(f_{n_{1}\widetilde{k}}(\overline{m}_{2}),\top)$ for some $n_{1}<\widetilde{k}$ and $m_{2}< \overline{m}_{2}$ in $\mathbb{N}^{<\omega}$. Suppose $\widetilde{u}=(f_{n_{1}\widetilde{k}}(\overline{m}_{2}),\top)$ with $n_{1}<\widetilde{k}$. If $\widetilde{u}=(f_{n_{1}\widetilde{k}}(\overline{m}_{2}),\top)\neq(n_{0},\top)$. Then $\widetilde{u}$ is not an upper bound of $\{(n_{0},\tilde{m}),(n_{1},m_{1})\}$ because $f_{n_{1}\widetilde{k}}(\overline{m}_{2}),n_{0},n_{1}$ are different, impossible. Thus, $\widetilde{u}=(f_{n_{1}\widetilde{k}}(m_{2}),\top)=(n_{0},\top)$. Equivalently, $f_{n_{1}\widetilde{k}}(\overline{m}_{2})=n_{0}=f_{\widetilde{d},n_{1}}(\overline{m}_{1})$ and then $\widetilde{k}=\widetilde{d}=n_{1}$, impossible. Suppose $\widetilde{u}=(n_{1},\top)$. Then $(n_{0},\tilde{m})\prec_{3}\widetilde{u}$ or $(n_{0},\tilde{m})\prec_{4}\widetilde{u}$. If $(n_{0},\tilde{m})\prec_{3}\widetilde{u}=(n_{1},\top)$, $n_{1}=f_{d_{1}n_{0}}(\tilde{m}')$ for some $d_{1}<n_{0}$ and $\tilde{m}<\tilde{m}'$.
Now $n_{0}<f_{d_{1}n_{0}}(\tilde{m}')=n_{1}$ and $n_{1}<f_{\widetilde{d},n_{1}}(\overline{m}_{1})=n_{0}$, impossible. If $(n_{0},\tilde{m})\prec_{4}\widetilde{u}=(n_{1},\top)$, combining $f_{\widetilde{d},n_{1}}(\overline{m}_{1})=n_{0}$, $n_{1}=f_{\widetilde{d},n_{1}}(\overline{m}_{1}\tilde{m}'')>n_{1}$ for some $\tilde{m}<\tilde{m}''$, impossible.

In other words, $A\subseteq \dn (n_{0},\top)=\dn (f_{\widetilde{d},n_{1}}(\overline{m}_{1}),\top)$ in subcase 1.1.

Subcase 2: $(n_{1},m_{1})\prec_{4}(n_{0},\top)$.

Then $n_{1}=f_{\widehat{a},\widehat{b}}(\widetilde{s})$ and $n_{0}=f_{\widehat{a},\widehat{b}}(\widetilde{s}\widetilde{m}'_{1})$ for some $\widehat{a},\widehat{b},\widetilde{s}\in \mathbb{N}$, $\widehat{a}<\widehat{b}$ and $m_{1}<\widetilde{m}'_{1}$.

In this subcase 2, we verify that $A\subseteq \dn(n_{0},\top)=\dn (f_{\widehat{a},\widehat{b}}(\widetilde{s}\widetilde{m}'_{1}),\top)$. Note that
$$\dn (f_{\widehat{a},\widehat{b}}(\widetilde{s}\widetilde{m}'_{1}),\top)=L_{f_{\widehat{a},\widehat{b}}(\widetilde{s}\widetilde{m}'_{1})}\cup\{(f_{\widehat{a},\widehat{b}}(\widetilde{s}),m)\mid 1\leq m\leq \widetilde{m}'_{1}\} \cup \{(\widehat{a},\widetilde{s}),(\widehat{a},\widetilde{s}\widetilde{m}'_{1})\},$$ where $\{(\widehat{a},\widetilde{s}),(\widehat{a},\widetilde{s}\widetilde{m}'_{1})\}\subseteq\{\widehat{a}\}\times \mathbb{N}^{<\omega}$.

Assume that $A\not\subseteq\dn (f_{\widehat{a},\widehat{b}}(\widetilde{s}\widetilde{m}'_{1}),\top)$. Take a $(n_{3},m_{3})\in A\setminus\dn (f_{\widehat{a},\widehat{b}}(\widetilde{s}\widetilde{m}'_{1}),\top)$. There are 2 possibilities ((2.1) and (2.2)).

(2.1) $n_{3}=n_{1}$ and $\widetilde{m}'_{1}<m_{3}$. Then $\{(n_{0},m'''),(n_{1},m_{3})\}$ has an upper bound $v'$ for all $(n_{0},m''')\in A\cap (\{n_{0}\}\times \mathbb{N})$. Since $A\cap (\{n_{0}\}\times \mathbb{N})$ is infinite, we assume that $max\{\widetilde{s},\widetilde{m}'_{1}\}+1<m'''$. Since $n_{0}\neq n_{1}$ and $\widetilde{m}'_{1}<m_{3}$, $v'=(n_{1},\top)$ or $v'=(n_{0},\top)$. Suppose $v'=(n_{1},\top)$. Then  $(n_{0},m''')\prec_{3}(n_{1},\top)=(f_{\widehat{a},\widehat{b}}(\widetilde{s}),\top)$. By strict order $\prec_{3}$, $n_{0}=\widehat{b}$ and $m'''\leq \widetilde{s}$, a contradiction. Suppose $v'=(n_{0},\top)$. Then $(n_{1},m_{3})\prec_{4}(n_{0},\top)$. Whence, $m_{3}\leq \widetilde{m}'_{1}$, impossible.

(2.2) $n_{3}=n_{1}$ and $m_{3}\in\mathbb{N}^{<\omega}$. Then $\{(n_{0},m''),(n_{1},m_{1}),(n_{1},m_{3})\}$ has an upper bound $\widehat{v}$ for some $(n_{0},m'')\in A\cap (\{n_{0}\}\times \mathbb{N})$ with $\widetilde{s}<m''$ . Using strict order $\prec_{2}$, we have $\widehat{v}=(n_{1},\top)$ or $\widehat{v}=(f_{n_{1},k'}(\overline{m}_{3}),\top)$ for some $n_{3}< k'$ and $m_{3}<\overline{m}_{3}$ in $\mathbb{N}^{<\omega}$.

If $\widehat{v}=(n_{1},\top)$, then $(n_{0},m'')\prec_{3}(f_{\widehat{a},\widehat{b}}(\widetilde{s}),\top)$. This means that $n_{0}=\widehat{b}$ and $m''\leq \widetilde{s}$, a contradiction.

If $\widehat{v}=(f_{n_{1},k'}(\overline{m}_{3}),\top)$ for some $n_{1}< k'$ and $m_{3}<\overline{m}_{3}$ in $\mathbb{N}^{<\omega}$, considering $(n_{0},m''),(n_{1},m_{1})\leq \widehat{v}$ and $n_{0}\neq n_{1}$, then $\widehat{v}=(n_{0},\top)$ or $\widehat{v}=(n_{1},\top)$. Suppose $\widehat{v}=(n_{1},\top)=(f_{\widehat{a},\widehat{b}}(\widetilde{s}),\top)$. According to $(n_{0},m'')<\widehat{v}=(f_{\widehat{a},\widehat{b}}(\widetilde{s}),\top)$, $(n_{0},m'')\prec_{3}(f_{\widehat{a},\widehat{b}}(\widetilde{s}),\top)$. This implies that $n_{0}=\widehat{b}=f_{\widehat{a},\widehat{b}}(\widetilde{s}\widetilde{m}'_{1})>\widehat{b}$, impossible. Suppose $\widehat{v}=(n_{0},\top)=(f_{\widehat{a},\widehat{b}}(\widetilde{s}\widetilde{m}'_{1}),\top)$. Now, we know $f_{\widehat{a},\widehat{b}}(\widetilde{s}\widetilde{m}'_{1})=f_{n_{1},k'}(\overline{m}_{3})$. Thus $n_{1}=\widehat{a}$ and $n_{1}=f_{\widehat{a},\widehat{b}}(\widetilde{s})>\widehat{a}$, impossible.
Thus, $A\subseteq \dn(n_{0},\top)=\dn (f_{\widehat{a},\widehat{b}}(\widetilde{s}\widetilde{m}'_{1}),\top)$ in subcase 2.

Now, we can conclude that $A$ has an upper bound in case (1).

(2) $A\cap (\{n_{0}\}\times \mathbb{N})$ is infinite and $(n_{1},y_{1}y_{2}\cdot\cdot\cdot y_{m})\in A$ for some $y_{1}y_{2}\cdot\cdot\cdot y_{m}\in\mathbb{N}^{<\omega}$.

The details can see Appendix.

(3) $A\cap (\{n_{0}\}\times \mathbb{N}^{<\omega})$ is infinite and $(n_{1},\hat{m}_{1})\in A$ for some $\hat{m}_{1}\in\mathbb{N}$.

The details can see Appendix.

(4) $A\cap (\{n_{0}\}\times \mathbb{N}^{<\omega})$ is infinite and $(n_{1},e_{1}e_{2}\cdot\cdot\cdot e_{k_{1}^{\prime}})\in A$ for some $e_{1}e_{2}\cdot\cdot\cdot e_{k_{1}^{\prime}}\in\mathbb{N}^{<\omega}$.

The details can see Appendix.

Therefore, $A$ has an upper bound.\end{proof}

\begin{lemma} {\rm Let $A\subseteq\widehat{P}$ be subset such that

(1) every finite subset of $A$ has an upper bound.

(2) $A$ just meet $L_{n_{0}}$, $L_{n_{1}}$ and  $L_{n_{2}}$ nonempty.

Then $A$ has an upper bound.}

\end{lemma}

\begin{proof}We need only to consider the case that $A$ is infinite. Assume that $A\cap L_{n_{0}}$ is infinite and $A\cap L_{n_{1}}\neq\emptyset$, $A\cap L_{n_{2}}\neq\emptyset$. There are only 8 possibilities ((1)-(8)).

(1) $A\cap(\{n_{0}\}\times\mathbb{N})$ is infinite and $A\cap(\{n_{1}\}\times\mathbb{N})\neq\emptyset$, $A\cap(\{n_{2}\}\times\mathbb{N})\neq\emptyset$.

Choose a $(n_{i},m_{i})\in A\cap(\{n_{i}\}\times\mathbb{N})$, $i=1,2,3$.  By strict orders $\prec_{2},\prec_{3} $ and $\prec_{4}$,\\ $\{(n_{1},m_{1}),(n_{2},m_{2}),(n_{3},m_{3})\}$ has no upper bounds. So case (1) is impossible.

(2) $A\cap(\{n_{0}\}\times\mathbb{N})$ is infinite and $A\cap(\{n_{1}\}\times\mathbb{N}^{<\omega})\neq\emptyset$, $A\cap(\{n_{2}\}\times\mathbb{N})\neq\emptyset$.

Choose a $(n_{0},m_{0})\in A\cap(\{n_{0}\}\times\mathbb{N})$, a $(n_{1},\xi_{1}\xi_{2}\cdot\cdot\cdot\xi_{m})\in A\cap(\{n_{1}\}\times\mathbb{N}^{<\omega})$ and a $(n_{2},m_{2})\in A\cap(\{n_{2}\}\times\mathbb{N})$.

Then $\{(n_{0},m_{0}),(n_{1},\xi_{1}\xi_{2}\cdot\cdot\cdot\xi_{m}),(n_{2},m_{2})\}$ has an upper bound $l$. As $n_{0}\neq n_{2}$, $l=(n_{0},\top)$ or $l=(n_{2},\top)$.

(2.1) $l=(n_{2},\top)$.

Since $A\cap(\{n_{0}\}\times\mathbb{N})$ is infinite, the element $m_{0}$ has infinitely many possibilities. Hence $l=(n_{2},\top)$ is impossible.

(2.2) $l=(n_{0},\top)$

Then $(n_{1},\xi_{1}\xi_{2}\cdot\cdot\cdot\xi_{m})\prec_{2}l=(n_{0},\top)$ and $(n_{2},m_{2})\leq l=(n_{0},\top)$. By $(n_{1},\xi_{1}\xi_{2}\cdot\cdot\cdot\xi_{m})\prec_{2}l=(n_{0},\top)$, $n_{0}=f_{n_{1}r_{1}}(\alpha_{1}\alpha_{2}\cdot\cdot\cdot\alpha_{t})$ for some $n_{1}<r_{1}$ and $\xi_{1}\xi_{2}\cdot\cdot\cdot\xi_{m}\leq\alpha_{1}\alpha_{2}\cdot\cdot\cdot\alpha_{t}$ in $\mathbb{N}^{<\omega}$. There are only two cases (2.2.1) and (2.2.2).

(2.2.1) $t=1$.

In this case, $\xi_{1}\xi_{2}\cdot\cdot\cdot\xi_{m}=\alpha_{1}\alpha_{2}\cdot\cdot\cdot\alpha_{t}=\xi_{1}$ and $(n_{2},m_{2})\prec_{3}\leq l=(n_{0},\top)=(f_{n_{1}r_{1}}(\xi_{1}),\top)$. Then $n_{2}=r_{1}>n_{1}$ and $m_{2}\leq\xi_{1}$. In the following, we assert that $$A\subseteq\dn(n_{0},\top)=\dn(f_{n_{1}r_{1}}(\xi_{1}),\top).$$
Assume that $A\not\subseteq\dn(f_{n_{1}r_{1}}(\xi_{1}),\top)$. Take a $(n_{3},m_{3})\in A\setminus\dn(f_{n_{1}r_{1}}(\xi_{1}),\top)$. Note that
$$\dn(f_{n_{1}r_{1}}(\xi_{1}),\top)=L_{f_{n_{1}r_{6}}(\xi_{1})}\cup\{(n_{1},\xi_{1})\}\cup\{(r_{1},k)\mid 1\leq k\leq\xi_{1}\},$$
where $(n_{1},\xi_{1})\in\{n_{1}\}\times\mathbb{N}^{<\omega}$. So we just consider the following 4 cases (2.2.1-1)-(2.2.1-4).

(2.2.1-1) $n_{3}=n_{1}$, $m_{3}\in\mathbb{N}$.

Then $\{(n_{0},m_{0}),(n_{1},\xi_{1}),(n_{2},m_{2}),(n_{1},m_{3})\}$ has an upper bound $j_{2}$. This is impossible since $n_{0},n_{1},n_{2}$ are different.

(2.2.1-2) $n_{3}=n_{1}$, $m_{3}\in\mathbb{N}^{<\omega}\setminus\{\xi_{1}\}$.

Let $m_{3}=\beta_{1}\beta_{2}\cdot\cdot\cdot\beta_{t_{1}}$. Then $\{(n_{0},m_{0}),(n_{1},\xi_{1}),(n_{2},m_{2}),(n_{1},\beta_{1}\beta_{2}\cdot\cdot\cdot\beta_{t_{1}})\}$ has an upper bound $j'_{2}$.
Similarly, $j'_{2}=(n_{0},\top)$ or $j'_{2}=(n_{2},\top)$.

If $j'_{2}=(n_{0},\top)$, $(n_{1},\beta_{1}\beta_{2}\cdot\cdot\cdot\beta_{t_{1}})\prec_{2}(n_{0},\top)$. This implies $n_{0}=f_{n_{1}r_{2}}(\beta'_{1}\beta'_{2}\cdot\cdot\cdot\beta'_{t'_{1}})$ for some $n_{1}<r_{2}$ and $\beta_{1}\beta_{2}\cdot\cdot\cdot\beta_{t_{1}}\leq\beta'_{1}\beta'_{2}\cdot\cdot\cdot\beta'_{t'_{1}}$. Now, we have $n_{0}=f_{n_{1}r_{2}}(\beta'_{1}\beta'_{2}\cdot\cdot\cdot\beta'_{t'_{1}})=f_{n_{1}r_{1}}(\xi_{1})$ and then $\beta'_{1}\beta'_{2}\cdot\cdot\cdot\beta'_{t'_{1}}=\xi_{1}$. Since $\beta_{1}\beta_{2}\cdot\cdot\cdot\beta_{t_{1}}\leq\beta'_{1}\beta'_{2}\cdot\cdot\cdot\beta'_{t'_{1}}=\xi_{1}$, we have $\beta_{1}\beta_{2}\cdot\cdot\cdot\beta_{t_{1}}=\xi_{1}$, a contradiction.

If $j'_{2}=(n_{2},\top)$, then $(n_{0},m_{0})\prec_{3}(n_{2},\top)$ or $(n_{0},m_{0})\prec_{4}(n_{2},\top)$. Hence, by the strict orders $\prec_{3}$ and $\prec_{4}$, $n_{0}<n_{2}$. However, $n_{2}=r_{1}<f_{n_{1}r_{1}}(\xi_{1})=n_{0}$, a contradiction.

(2.2.1-3) $n_{3}=n_{2}=r_{1}$, $m_{3}\in\mathbb{N}^{<\omega}$.

Let $m_{3}=\beta'_{1}\beta'_{2}\cdot\cdot\cdot\beta'_{t'_{1}}$. Then $\{(n_{0},m_{0}),(n_{1},\xi_{1}),(n_{2},m_{2}),(n_{2},\beta'_{1}\beta'_{2}\cdot\cdot\cdot\beta'_{t'_{1}})\}$ has an upper bound $j_{3}$. As $n_{1}\neq n_{2}$ and $\xi_{1},\beta'_{1}\beta'_{2}\cdot\cdot\cdot\beta'_{t'_{1}}\in\mathbb{N}^{<\omega}$, $j_{3}=(n_{1},\top)$ or $j_{3}=(n_{2},\top)$.

If $j_{3}=(n_{1},\top)$, $(n_{0},m_{0})\prec_{3}(n_{1},\top)$ or $(n_{0},m_{0})\prec_{4}(n_{1},\top)$. Using the strict orders $\prec_{3}$ and $\prec_{4}$, $n_{0}<n_{1}$. However, $n_{1}<f_{n_{1}r_{6}}(\xi_{1})=n_{0}$, a contradiction.

If $j_{3}=(n_{2},\top)$, $(n_{0},m_{0})\prec_{3}(n_{2},\top)$ or $(n_{0},m_{0})\prec_{4}(n_{2},\top)$. Similarly, $n_{0}<n_{2}$ but $n_{2}=r_{1}<f_{n_{1}r_{6}}(\xi_{1})=n_{0}$, a contradiction.

(2.2.1-4) $n_{3}=n_{2}=r_{1}$, $m_{3}\in\mathbb{N}$ with $\xi_{1}<m_{3}$.

Then $\{(n_{0},m_{0}),(n_{1},\xi_{1}),(n_{2},m_{3})\}$ has an upper bound $j'_{3}$. Similarly, $j'_{3}=(n_{0},\top)$ or $j'_{3}=(n_{2},\top)$.

If $j'_{3}=(n_{2},\top)$, $(n_{0},m_{0})\prec_{3}(n_{2},\top)$ or $(n_{0},m_{0})\prec_{4}(n_{2},\top)$. So we have $n_{0}<n_{2}$ but $n_{2}=r_{1}<f_{n_{1}r_{1}}(\xi_{1})=n_{0}$, a contradiction.

If $j'_{3}=(n_{0},\top)$, then  $(n_{2},m_{3})\prec_{3}(n_{0},\top)=(f_{n_{1}r_{1}}(\xi_{1}),\top)$. Thus, $m_{3}\leq\xi_{1}$, a contradiction.

(2.2.2) $1<t$.

In this case $(n_{2},m_{2})\prec_{4}l=(n_{0},\top)=(f_{n_{1}r_{1}}(\alpha_{1}\alpha_{2}\cdot\cdot\cdot\alpha_{t}),\top)$. So we have that $n_{2}=f_{n_{1}r_{1}}(\alpha_{1}\alpha_{2}\cdot\cdot\cdot\alpha_{t-1})$ and $m_{2}\leq\alpha_{t}$. In the following, we assert that $$A\subseteq\dn(n_{0},\top)=\dn(f_{n_{1}r_{1}}(\alpha_{1}\alpha_{2}\cdot\cdot\cdot\alpha_{t}),\top).$$
Assume that $A\not\subseteq\dn(f_{n_{1}r_{1}}(\alpha_{1}\alpha_{2}\cdot\cdot\cdot\alpha_{t}),\top)$. Choose a $(n_{4},m_{4})\in A\setminus\dn(f_{n_{1}r_{1}}(\alpha_{1}\alpha_{2}\cdot\cdot\cdot\alpha_{t}),\top)$. Note that
\begin{center}$\begin{array}{lll}
\dn(f_{n_{1}r_{1}}(\alpha_{1}\alpha_{2}\cdot\cdot\cdot\alpha_{t}),\top)&=&L_{f_{n_{1}r_{1}}(\alpha_{1}\alpha_{2}\cdot\cdot\cdot\alpha_{t})}\cup\{(n_{1},\alpha_{1}\alpha_{2}\cdot\cdot\cdot\alpha_{t'})\mid 1\leq t'\leq t\} \\
&&\cup\{(f_{n_{1}r_{1}}(\alpha_{1}\alpha_{2}\cdot\cdot\cdot\alpha_{t-1}),k)\mid 1\leq k\leq\alpha_{t}\}.\\
\end{array}$\end{center}
So we just consider the following 4 cases (2.2.2-1)-(2.2.2-4).

(2.2.2-1) $n_{4}=n_{1}$, $m_{4}\in\mathbb{N}$.

Then
$$\{(n_{0},m_{0}),(n_{1},\xi_{1}\xi_{2}\cdot\cdot\cdot\xi_{m}),(n_{2},m_{2}),(n_{1},m_{4})\}$$ has an upper bound $j_{4}$. It is impossible since $n_{0},n_{1},n_{2}$ are different.

(2.2.2-2) $n_{4}=n_{1}$, $m_{4}\in\mathbb{N}^{<\omega}\setminus\{(n_{1},\alpha_{1}\alpha_{2}\cdot\cdot\cdot\alpha_{t'})\mid 1\leq t'\leq t\}$.

Let $m_{4}=\varphi_{1}\varphi_{2}\cdot\cdot\cdot\varphi_{t_{2}}$. Then $$\{(n_{0},m_{0}),(n_{1},\xi_{1}\xi_{2}\cdot\cdot\cdot\xi_{m}),(n_{2},m_{2}),(n_{1},\varphi_{1}\varphi_{2}\cdot\cdot\cdot\varphi_{t_{2}})\}$$ has an upper bound $j'_{4}$. Clearly, $j'_{4}=(n_{0},\top)$ or $j'_{4}=(n_{2},\top)$.

If $j'_{4}=(n_{2},\top)$, then $(n_{0},m_{0})\prec_{3}j'_{4}=(n_{2},\top)$ or $(n_{0},m_{0})\prec_{4}j'_{4}=(n_{2},\top)$. So we have $n_{0}<n_{2}$. However, $$n_{0}=f_{n_{1}r_{1}}(\alpha_{1}\alpha_{2}\cdot\cdot\cdot\alpha_{t})>f_{n_{1}r_{1}}(\alpha_{1}\alpha_{2}\cdot\cdot\cdot\alpha_{t-1})=n_{2},$$ a contradiction.

If $j'_{4}=(n_{0},\top)$, then $(n_{1},\xi_{1}\xi_{2}\cdot\cdot\cdot\xi_{m})\prec_{2}j'_{4}=(n_{0},\top)$ , $(n_{1},\varphi_{1}\varphi_{2}\cdot\cdot\cdot\varphi_{t_{2}})\prec_{2}j'_{4}=(n_{0},\top)$ and $(n_{2},m_{2})\leq(n_{0},\top)$.
Then $n_{0}=f_{n_{1}r_{3}}(\psi_{1}\psi_{2}\cdot\cdot\cdot\psi_{t_{3}})$ for some $n_{1}<r_{3}$ and $$\xi_{1}\xi_{2}\cdot\cdot\cdot\xi_{m}\leq\psi_{1}\psi_{2}\cdot\cdot\cdot\psi_{t_{3}}, \varphi_{1}\varphi_{2}\cdot\cdot\cdot\varphi_{t_{2}}\leq\psi_{1}\psi_{2}\cdot\cdot\cdot\psi_{t_{3}}. $$ Now, we have $n_{0}=f_{n_{1}r_{3}}(\psi_{1}\psi_{2}\cdot\cdot\cdot\psi_{t_{3}})=f_{n_{1}r_{6}}(\alpha_{1}\alpha_{2}\cdot\cdot\cdot\alpha_{t})$. Whence, $\psi_{1}\psi_{2}\cdot\cdot\cdot\psi_{t_{3}}=\alpha_{1}\alpha_{2}\cdot\cdot\cdot\alpha_{t}$ and then $\varphi_{1}\varphi_{2}\cdot\cdot\cdot\varphi_{t_{2}}\leq\alpha_{1}\alpha_{2}\cdot\cdot\cdot\alpha_{t}$, a contradiction.

(2.2.2-3) $n_{4}=n_{2}$, $m_{4}\in\mathbb{N}$ with $\alpha_{t}<m_{4}$.

Then $\{(n_{0},m_{0}),(n_{1},\xi_{1}\xi_{2}\cdot\cdot\cdot\xi_{m}),(n_{2},m_{4})\}$ has an upper bound $j_{5}$. Obviously, $j_{5}=(n_{0},\top)$ or $j_{5}=(n_{2},\top)$.

If $j_{5}=(n_{2},\top)$, then $(n_{0},m_{0})\prec_{3}j'_{4}=(n_{2},\top)$ or $(n_{0},m_{0})\prec_{4}j'_{4}=(n_{2},\top)$. Whence, $n_{0}<n_{2}$ but $$n_{0}=f_{n_{1}r_{1}}(\alpha_{1}\alpha_{2}\cdot\cdot\cdot\alpha_{t})>f_{n_{1}r_{1}}(\alpha_{1}\alpha_{2}\cdot\cdot\cdot\alpha_{t-1})=n_{2},$$ a contradiction.

If $j_{5}=(n_{0},\top)$, then $(n_{2},m_{4})\prec_{4}(n_{0},\top)=(f_{n_{1}r_{6}}(\alpha_{1}\alpha_{2}\cdot\cdot\cdot\alpha_{t}),\top)$. Then $m_{4}\leq\alpha_{t}$, impossible.

(2.2.2-4) $n_{4}=n_{2}$, $m_{4}\in\mathbb{N}^{<\omega}$.

Then $\{(n_{0},m_{0}),(n_{1},\xi_{1}\xi_{2}\cdot\cdot\cdot\xi_{m}),(n_{2},m_{2}),(n_{2},m_{4})\}$ has an upper bound $j'_{5}$. Since $n_{0}\neq n_{2}$ and $m_{0}, m_{2}\in\mathbb{N}$, $j'_{5}=(n_{0},\top)$ or $j'_{5}=(n_{2},\top)$. Simultaneously, as  $n_{1}\neq n_{2}$ and $\xi_{1}\xi_{2}\cdot\cdot\cdot\xi_{m}, m_{4}\in\mathbb{N}^{<\omega}$, $j'_{5}=(n_{1},\top)$ or $j'_{5}=(n_{2},\top)$. Consequently, $j'_{5}=(n_{2},\top)$ and $(n_{0},m_{0})\prec_{3}j'_{5}=(n_{2},\top)$ or $(n_{0},m_{0})\prec_{4}j'_{5}=(n_{2},\top)$. Similarly, we have $n_{0}<n_{2}$ but it contradicts  $n_{0}=f_{n_{1}r_{1}}(\alpha_{1}\alpha_{2}\cdot\cdot\cdot\alpha_{t})>f_{n_{1}r_{1}}(\alpha_{1}\alpha_{2}\cdot\cdot\cdot\alpha_{t-1})=n_{2}$, impossible.

So we can conclude that $A$ has an upper bound in case (2).

(3) $A\cap(\{n_{0}\}\times\mathbb{N})$ is infinite and $A\cap(\{n_{1}\}\times\mathbb{N})\neq\emptyset$, $A\cap(\{n_{2}\}\times\mathbb{N}^{<\omega})\neq\emptyset$.

It is similar to the case (2) and hence $A$ still has an upper bound in case (3).

(4) $A\cap(\{n_{0}\}\times\mathbb{N})$ is infinite and $A\cap(\{n_{1}\}\times\mathbb{N}^{<\omega})\neq\emptyset$, $A\cap(\{n_{2}\}\times\mathbb{N}^{<\omega})\neq\emptyset$.

Choose a $(n_{1},\mu_{1}\mu_{2}\cdot\cdot\cdot\mu_{k})\in(\{n_{1}\}\times\mathbb{N}^{<\omega})$ and $(n_{2},\nu_{1}\nu_{2}\cdot\cdot\cdot\nu_{k'})\in(\{n_{2}\}\times\mathbb{N}^{<\omega})$. Since $A\cap(\{n_{0}\}\times\mathbb{N})$ is infinite, there is a $(n_{0},\overline{m}_{0})\in A$ such that $$max\{\mu_{1}+2,\mu_{2}+2,\cdot\cdot\cdot,\mu_{k}+2,\nu_{1}+2,\nu_{2}+2,\cdot\cdot\cdot,\nu_{k'}+2\}<\overline{m}_{0}.$$
Now, $$\{(n_{0},\overline{m}_{0}),(n_{1},\mu_{1}\mu_{2}\cdot\cdot\cdot\mu_{k}),(n_{2},\nu_{1}\nu_{2}\cdot\cdot\cdot\nu_{k'})\}$$ has an upper bound $u$. Clearly, $u=(n_{1},\top)$ or $u=(n_{2},\top)$. we just consider the following two possibilities (4.1) and (4.2).

(4.1) $u=(n_{1},\top)$.

In this case, $(n_{2},\nu_{1}\nu_{2}\cdot\cdot\cdot\nu_{k'})\prec_{2}(n_{1},\top)$ and $(n_{0},\overline{m}_{0})\leq(n_{1},\top)$. By $(n_{2},\nu_{1}\nu_{2}\cdot\cdot\cdot\nu_{k'})\prec_{2}(n_{1},\top)$, $n_{1}=f_{n_{2}s_{1}}(\sigma_{1}\sigma_{2}\cdot\cdot\cdot\sigma_{k_{1}})$ for some $n_{2}<s_{1}$ and $\nu_{1}\nu_{2}\cdot\cdot\cdot\nu_{k'}\leq\sigma_{1}\sigma_{2}\cdot\cdot\cdot\sigma_{k_{1}}$.

If $k_{1}=1$, then $\nu_{1}\nu_{2}\cdot\cdot\cdot\nu_{k'}=\sigma_{1}\sigma_{2}\cdot\cdot\cdot\sigma_{k_{1}}=\nu_{1}$ and $(n_{0},\overline{m}_{0})\prec_{3}(f_{n_{2}s_{1}}(\nu_{1}),\top)$. Whence $\overline{m}_{0}<\nu_{1}$, impossible.

If $1<k_{1}$, then $(n_{0},\overline{m}_{0})\prec_{4}(f_{n_{2}s_{1}}(\sigma_{1}\sigma_{2}\cdot\cdot\cdot\sigma_{k_{1}}),\top)$. Hence, $n_{0}=f_{n_{2}s_{1}}(\sigma_{1}\sigma_{2}\cdot\cdot\cdot\sigma_{k_{1}-1})$ and $\overline{m}_{0}\leq\sigma_{k_{1}}$. Since $A\cap(\{n_{0}\}\times\mathbb{N})$ is infinite, there is a $(n_{0},\widetilde{m}_{0})\in A$ such that $$max\{\overline{m}_{0}+2,\sigma_{1}+2,\sigma_{2}+2,\cdot\cdot\cdot,\sigma_{k_{1}}+2\}<\widetilde{m}_{0}.$$
Then $\{(n_{0},\widetilde{m}_{0}),(n_{1},\mu_{1}\mu_{2}\cdot\cdot\cdot\mu_{k}),(n_{2},\nu_{1}\nu_{2}\cdot\cdot\cdot\nu_{k'})\}$ has an upper bound $u'$. Similarly, $u'=(n_{1},\top)$ or $u'=(n_{2},\top)$. Assume that $u'=(n_{2},\top)$. Then $(n_{1},\mu_{1}\mu_{2}\cdot\cdot\cdot\mu_{k})\prec_{2}u'=(n_{2},\top)$ and hence $n_{1}< n_{2}$ by strict order $\prec_{2}$. However, $n_{1}=f_{n_{2}s_{1}}(\sigma_{1}\sigma_{2}\cdot\cdot\cdot\sigma_{k_{1}})>n_{2}$, impossible. So we have $u'=(n_{1},\top)$. As $1<k_{1}$, $(n_{0},\widetilde{m}_{0})\prec_{4}(n_{1},\top)=(f_{n_{2}s_{1}}(\sigma_{1}\sigma_{2}\cdot\cdot\cdot\sigma_{k_{1}}),\top)$. This means that $\widetilde{m}_{0}\leq\sigma_{k_{1}}$, a contradiction.

(4.2) $u=(n_{2},\top)$.

It is similar to (4.1) and hence (4.2) is still impossible.

In a conclusion, the case (4) is always impossible.

(5) $A\cap(\{n_{0}\}\times\mathbb{N}^{<\omega})$ is infinite and $A\cap(\{n_{1}\}\times\mathbb{N}^{<\omega})\neq\emptyset$, $A\cap(\{n_{2}\}\times\mathbb{N}^{<\omega})\neq\emptyset$.

This case is always impossible since $n_{0},n_{1},n_{2}$ are different.

(6) $A\cap(\{n_{0}\}\times\mathbb{N}^{<\omega})$ is infinite and $A\cap(\{n_{1}\}\times\mathbb{N})\neq\emptyset$, $A\cap(\{n_{2}\}\times\mathbb{N}^{<\omega})\neq\emptyset$.

The details can see Appendix.

(7) $A\cap(\{n_{0}\}\times\mathbb{N}^{<\omega})$ is infinite and $A\cap(\{n_{1}\}\times\mathbb{N}^{<\omega})\neq\emptyset$, $A\cap(\{n_{2}\}\times\mathbb{N})\neq\emptyset$.

It is similar to case (6) and hence $A$ still has an upper bound in case (7).

(8) $A\cap(\{n_{0}\}\times\mathbb{N}^{<\omega})$ is infinite and $A\cap(\{n_{1}\}\times\mathbb{N})\neq\emptyset$, $A\cap(\{n_{2}\}\times\mathbb{N})\neq\emptyset$.

The details can see Appendix.

According to case (1)-(8), $A$ has an upper bound in $\widehat{P}$.\end{proof}

\begin{proposition} $\omega(\widehat{P})$ is an uncountable frame which is not sober for the Scott topology.
\end{proposition}
\setlength{\baselineskip}{1.1\baselineskip}

\begin{proof}Clearly, ${\rm max(\widehat{P})}=\{(n,\top)\mid n\in\mathbb{N}\}$ and hence ${\rm max(\widehat{P})}$ is infinite. By Lemma 4.7, ${\rm max}(\widehat{P})$ is irreducible in $\Sigma\widehat{P}$. By Lemma 3.17, Lemma 4.8, Lemma 4.9 and Lemma 4.10, $(\widehat{P},\omega(\widehat{P}))$ is compact. By Theorem 4.3, and Corollary 4.4, we have $(\widehat{P},\omega(\widehat{P}))$ is not $SI$-compact and  $\Sigma \omega(\widehat{P})$ is non sober.

Claim: $\omega(\widehat{P})$ is a uncountable frame.

Note that $\omega(\widehat{P})=\{\widehat{P}\setminus\bigcap\limits_{i\in I}\up F_{i}\mid (F_{i})_{i\in I}\ \mbox{is a family of finite subsets of}\ \widehat{P}\}.$ Since $\widehat{P}$ is countable, the set of all finite subsets of $\widehat{P}$ is still countable. So we have

 $\omega(\widehat{P})=\{\widehat{P}\setminus\bigcap\limits_{n\in\mathbb{N}}\up F_{n}\mid (F_{n})_{n\in\mathbb{N}}\ \mbox{is a family of finite subsets of}\ \widehat{P}\}.$
 In the following, we check that $\mathcal{A}=\{ \bigcap\limits_{n\in\mathbb{N}}\up F_{n}\mid (F_{n})_{n\in\mathbb{N}}\ \mbox{is a finite subsets of}\ \widehat{P}\}$ is uncountable.

In the following, we just consider the case $F_{n}\subseteq\{1\}\times\mathbb{N}^{<\omega}$ for all $n\in\mathbb{N}$.

Let $s=(s_{n})_{n\in \mathbb{N}},t=(t_{n})_{n\in \mathbb{N}}\in\{1\}\times \mathbb{N}^{<\omega}$ be sequences such that $s$ is increasing and

(1) $s_{n}\leq t_{n}$ in $\mathbb{N}^{<\omega}$ and $t_{n}\|s_{k}$, for all $ n,k\in\mathbb{N}$ with $n<k$;

(2) every two different elements of $\{t_{n}\mid n\in\mathbb{N}\}$ are incomparable.

\begin{center}
\centering
\includegraphics[totalheight=2.9in]{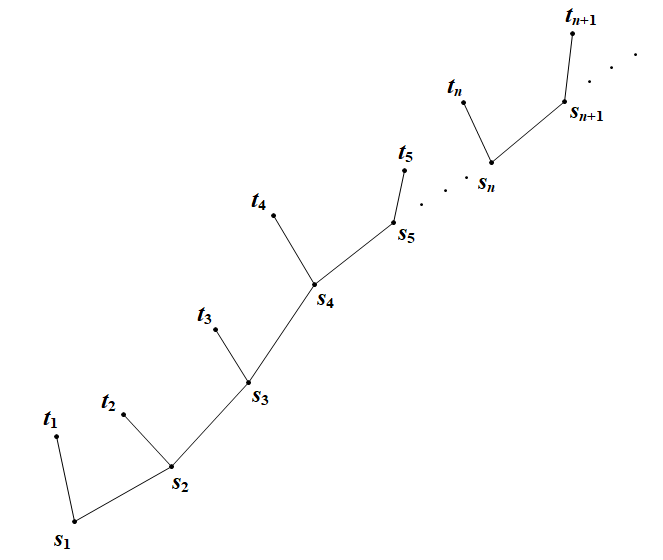}\\Figure 8 The elements $s_{n}$ and $t_{n}$.
\end{center}

Let $F_{n}=\{t_{1},t_{2},\cdot\cdot\cdot ,t_{n},s_{n+1}\}$ for all $n\in\mathbb{N}$. Then we have $\bigcap\limits_{n\in\mathbb{N}}\up F_{n}=\bigcup\limits_{n\in\mathbb{N}}\up t_{n}$.

Since the second projection of this set $H$ is uncountable, $\mathcal{A}$ is uncountable, where $H$ is the set of all pairs $(s,t)$ satisfying the conditions above.
Hence, $\omega(\widehat{P})$ is not countable.\end{proof}

\quad In the following, we present a countable compact $T_{1}$ space which is non $SI$-compact.

\begin{lemma} Let $(X,\tau)$ be a space which is not $SI$-compact. If $\widehat{\tau}$ is a topology on $X$ such that $\tau\subseteq\widehat{\tau}$, then $(X,\widehat{\tau})$ is not $SI$-compact.
\end{lemma}

\begin{proof}Consider the inclusion map $i:\tau\longrightarrow\widehat{\tau}$. Clearly, $i$ is Scott continuous. Since $(X,\tau)$ is a space which is not $SI$-compact, there is an open cover $\mathcal{U}\subseteq \tau\setminus\{X\}$ of $X$ which is irreducible in $\Sigma \tau$. Hence, $i(\mathcal{U})=\mathcal{U}$ is still irreducible in $\Sigma(\hat{\tau})$. Thus, $(X,\widehat{\tau})$ is not $SI$-compact.\end{proof}

\begin{corollary} There is a countable $T_{1}$ space $X$ which is compact but not $SI$-compact.

\end{corollary}

\begin{proof}Let $X=(\widehat{P},\omega(\widehat{P}))$. Consider the topology $\widehat{\tau}$ which has $$\mathcal{A}=\tau\cup\{X\setminus F\mid F\ \mbox{is a finite subset of}\ X\}$$ as a subbase, where $\tau=\omega(\widehat{P})$. Obviously, $(X,\widehat{\tau})$ is $T_{1}$. Let $\mathcal{U}\subseteq\mathcal{A}$ be an cover of $X$. We consider the following two cases:

(1) There is a finite subset $F\subseteq X$ such that $X\setminus F\in\mathcal{U}$.

For each $x\in F$, take a $U_{x}\in\mathcal{U}$. We see that $\mathcal{V}\subseteq\mathcal{U}$ is a finite subcover of $X$, where $\mathcal{V}=\{X\setminus F\}\cup\{U_{x}\mid x\in F\}$.

(2) $\mathcal{U}\cap\{X\setminus F\mid F\ \mbox{is a finite subset of}\ X\}=\emptyset$.

In this case, we have $\mathcal{U}\subseteq\omega(\widehat{P})$. Since $(X,\omega(\widehat{P}))$ is  a compact space, there is a finite subset $\mathcal{W}\subseteq\mathcal{U}$ such that  $\mathcal{W}$ covers $X$.

Therefore, $(X,\widehat{\tau})$ is compact. By Lemma 4.12, $(X,\widehat{\tau})$ is not $SI$-compact.\end{proof}

\quad We have seen that compactness and $SI$-compactness are not equivalent in upper topology, lower topology of posets. The following Proposition 4.14 states that they are equivalent for the Scott topology of posets.

\begin{proposition}Let $L$ be a poset. The following statements are equivalent.

{\rm(1)} $\Sigma L$ is compact;

{\rm(2)} $L=\up F$ for some nonempty finite subset $F\subseteq L$;

{\rm(3)} $\Sigma L$ is $SI$-compact.

{\rm(4)} $\Sigma L$ is $\upsilon$-compact.

\end{proposition}

\begin{proof}By Corollary 3.14, it is enough to prove that (2) implies (4).

By (2), we can further require that $F=min(L)$. Assume that $$F=\{x_{0},x_{1},\cdot\cdot\cdot,x_{n}\}.$$

Claim 1: $L\setminus \{x_{i}\}\in\sigma(L)$ for all $0\leq i\leq n$.

Obviously, $L\setminus \{x_{i}\}$ is an upper set since $x_{i}$ is a minimal element. Let $D$ be a directed set such that $\bigvee D$ exists and $\bigvee D\in L\setminus \{x_{i}\}$. If $D\cap L\setminus \{x_{i}\}=\emptyset$, then $D=\{x_{i}\}$, a contradiction. So we have $D\cap L\setminus \{x_{i}\}\neq\emptyset$. This shows that $L\setminus \{x_{i}\}\in\sigma(L)$.

Let $\mathcal{U}\subseteq\sigma(L)$ be an irreducible open cover of $L$ in the upper topology of $\sigma(L)$. Assume that $L\not\in\mathcal{U}$. Then we have
$$\mathcal{U}\subseteq\bigcup\limits_{0\leq i\leq n}\dn_{\sigma(L)} L\setminus\{x_{i}\} .$$
By the irreducibility of $\mathcal{U}$, $\mathcal{U}\subseteq \dn_{\sigma(L)} L\setminus\{x_{i_{0}}\}$ for some $0\leq i_{0}\leq n$. This means that $\bigcup\mathcal{U}=L\setminus\{x_{i_{0}}\}$, a contradiction. Hence $X\in\mathcal{U}$ and thus $\Sigma L$ is $\upsilon$-compact.
\end{proof}

\section{Weak $\upsilon$-compactness}
\quad By the proof of Proposition 3.20, for a infinite $T_{1}$ space $X$, if an open cover $\mathcal{U}$ without $X$ is irreducible in $(\mathcal{O}(X),\upsilon(\mathcal{O}(X))$, then $\mathcal{U}$ does not contain infinitely many open sets coming from $\mathscr{F}(X)$. So we modify the definition of $\upsilon$-compactness.

\begin{definition}A topological space $X$ is called weakly $\upsilon$-compact if for every irreducible open cover $\mathcal{U}$ of $X$ in the upper topology of $\mathcal{O}(X)$ with $\mathcal{U}\cap\mathscr{F}(X)$ being finite, $X\in\mathcal{U}$.
\end{definition}

\begin{proposition} Every weakly $\upsilon$-compact space is compact.
\end{proposition}

\begin{proof}Let $\mathcal{U}\subseteq\mathcal{O}(X)$ be a directed open cover of $X$. Then $\mathcal{U}$ is irreducible in the upper topology of $\mathcal{O}(X)$. If $\mathcal{U}\cap\mathscr{F}(X)\neq\emptyset$, then by the directedness of $\mathcal{U}$,  it is not difficult to show that $X\in\mathcal{U}$. If on the contrary, $\mathcal{U}\cap\mathscr{F}(X)=\emptyset$, then by the weak $\upsilon$-compactness of $X$, $X\in\mathcal{U}$. Therefore, $X$ is compact.\end{proof}

\begin{example}{\rm (1) Equip the natural numbers set $\mathbf{N}$  with the co-finite topology $\tau_{cof}$. Then $(\mathbf{N},\tau_{cof})$ is weakly $\upsilon$-compact but not $\upsilon$-compact. Similarly, the topological space $Y$ given in Example 3.19 is weakly $\upsilon$-compact.

(2) Let $X$ be an uncountable infinite set and $a\in X$. The topology $\tau$ on $X$ is given by $$\tau=\{X\setminus F\mid F\ \mbox{is a finite set of}\ X\}\cup\{V\subseteq X\mid a\not\in V\}.$$ Then $X$ is a compact Hausdorff space but not first-countable. By Proposition 3.20, $(X,\tau)$ is not $\upsilon$-compact. Let $\mathcal{U}\subseteq\tau$ be an irreducible cover of $X$ in the upper topology of $\mathcal{O}(X)$ with $\mathcal{U}\cap \mathscr{F}(X)$ being finite. Assume that $X\not\in\mathcal{U}$. Clearly,
$$\mathcal{U}\subseteq \dn(\mathcal{U}\cap \mathscr{F}(X))\cup\dn V_{a},$$
where $V_{a}=\bigcup\{V\in\mathcal{U}\mid a\not\in V\}$. As $\mathcal{U}\cap \mathscr{F}(X)$ is finite,$\dn(\mathcal{U}\cap \mathscr{F}(X))\cup\dn V_{a}$ is closed in the upper topology of $\mathcal{O}(X)$. Then we have $cl_{\upsilon(\tau)}(\mathcal{U})\subseteq\dn(\mathcal{U}\cap \mathscr{F}(X))\cup\dn V_{a}$. It follows from the sobriety of $(\mathcal{O}(X),\upsilon(\mathcal{O}(X)))$ that $cl_{\upsilon(\tau)}(\mathcal{U})=\dn_{\mathcal{O}(X)}U_{0}$ for some $U_{0}\in\tau$. Obviously, $U_{0}=X$ since $\mathcal{U}$ covers $X$. It is a contradiction. Consequently, $X\in\mathcal{U}$ and then $(X,\tau)$ is weakly $\upsilon$-compact.

(3) The unit interval $[0,1]$ of the real line $\mathbb{E}$ is a compact metric space but not weakly $\upsilon$-compact. For every natural number $1\leq m$, take $$\mathcal{U}_{m}=\{[0,\frac{m}{m+1})\cup(\frac{m}{m+1},\frac{m+1}{m+2}-\frac{1}{n})\cup(\frac{m+1}{m+2},1]\mid n\in\mathbb{N}, 1+(m+1)(m+2)\leq n\},$$
and let $\mathcal{U}=\bigcup\limits_{1\leq m<\infty}\mathcal{U}_{m}$. Note that $\bigcup\mathcal{U}_{m}=[0,1]\setminus\{\frac{m}{m+1},\frac{m+1}{m+2}\}$ and hence $\mathcal{U}$ covers $[0,1]$. In addition, $\mathcal{U}\cap \mathscr{F}([0,1])=\emptyset$.
If $cl_{\upsilon(\mathcal{O}([0,1]))}(\mathcal{U})\neq\mathcal{O}([0,1])$, then there exists a nonempty family $\{\mathcal{F}_{i}\mid i\in I\}\subseteq\mathbf{Fin}(\mathcal{O}([0,1]))$ such that $cl_{\upsilon(\mathcal{O}([0,1]))}(\mathcal{U})=\bigcap\limits_{i\in I}\dn \mathcal{F}_{i}$. For every $i\in I$, $\mathcal{U}\subseteq\dn \mathcal{F}_{i}$. For convenience, suppose $\mathcal{F}_{i}=\{V_{i1},V_{i2},\cdot\cdot\cdot,V_{ik}\}\subseteq\mathcal{O}([0,1])\setminus\{[0,1]\}$. Thus, we have that for every $1\leq r\leq k$, there is a $a_{r}\in(0,1)$ such that $V_{i_{r}}\subseteq[0,1]\setminus\{a_{r}\}$. This means that $\mathcal{U}\subseteq\dn\{[0,1]\setminus\{a_{r}\}\mid 1\leq r\leq k \}$. We can find  $m_{0}\in\mathbb{N}$ such that $2\leq m_{0}$ and $max\{a_{1},a_{2},\cdot\cdot\cdot,a_{k}\}<\frac{m_{0}}{m_{0}+1}$. Consider $$W=[0,\frac{m_{0}}{m_{0}+1})\cup(\frac{m_{0}}{m_{0}+1},\frac{m_{0}+1}{m_{0}+2}-\frac{1}{(m_{0}+1)(m_{0}+2)+1})\cup(\frac{m_{0}+1}{m_{0}+2},1].$$ Then $W\in\mathcal{U}_{m_{0}}\subseteq\mathcal{U}$ but $W\not\subseteq[0,1]\setminus\{a_{r}\}$ for all $1\leq r\leq k$, a contradiction. So $cl_{\upsilon(\mathcal{O}([0,1]))}(\mathcal{U})=\mathcal{O}([0,1])$ and hence $\mathcal{U}$ is irreducible in the upper topology of $\mathcal{O}([0,1])$. Clearly, $[0,1]\not\in\mathcal{U}$ and then $[0,1]$ is not weakly $\upsilon$-compact.

(4) Let $X_{0}=\{0\}\cup\{\frac{1}{n}\mid n\in\mathbb{N}^{+}\}\subseteq\mathbb{R}$ be endowed with the subspace topology of $\mathbb{R}$. Then $X_{0}$ is a compact metric space. Let $\mathcal{V}\subseteq\mathcal{O}(X_{0})$ be an irreducible cover of $X$ with respect to the upper topology of $\mathcal{O}(X_{0})$ such that $\mathcal{V}\cap\mathscr{F}(X_{0})$ is finite. Then $\{V\in\mathcal{V}\mid 0\in V\}$ is a nonempty finite set. If $X_{0}\not\in \mathcal{V}$, then $$cl_{\upsilon(\mathcal{O}(X_{0}))}(\mathcal{U})\subseteq \dn(\{V\in\mathcal{V}\mid 0\in V\})\cup\dn\{\frac{1}{n}\mid n\in\mathbb{N}^{+}\}.$$ Similarly, $cl_{\upsilon(\mathcal{O}(X_{0}))}(\mathcal{U})=\mathcal{O}(X_{0})$, a contradiction. Therefore, $X_{0}$ is weakly $\upsilon$-compact.}

\end{example}

\quad Naturally the following question arises.

$\bullet$ Let $X\subseteq\mathbb{E}$ be a compact subspace. When is $X$  weakly $\upsilon$-compact?

\quad In the following, we will completely answer that question. Given a subset $M\subseteq \mathbb{R}^{m}$ and $x\in M$, $x$ is a limit point in $M$ if there is a strictly monotone sequence $(x_{n})_{n\in\mathbb{N}}\subseteq M\setminus\{x\}$ such that $\lim\limits_{n\longrightarrow\infty}x_{n}=x$. We use $lim(M)$ to denote the set of all limit-points in $M$.

\begin{theorem} Let $X\subseteq\mathbb{E}$ be a compact subspace of the real line. Then the following statements are equivalent.

(a) $X$ is weakly $\upsilon$-compact;

(b) $lim(X)$ is finite.
\end{theorem}

\begin{proof}${\rm (a)\Longrightarrow(b)}$ Assume that $lim(X)$ is infinite. Then there is a strictly increasing (or strictly decreasing) sequence $(k_{n})_{n\in\mathbb{N}}\subseteq lim(X)$. Without loss of generality, we suppose that $(k_{n})_{n\in\mathbb{N}}$ is increasing. By compactness, $X$ is a bounded closed subset of $\mathbb{R}$. Then the sequence $(k_{n})_{n\in\mathbb{N}}$ is  convergent in the real line $\mathbb{R}$ and let $k$ be its limit. It is evident that $k_{n}< k$ for all $n\in\mathbb{N}$ and $k\in X$ since $X$ is closed. Since $k_{n}$ is a limit point of $X$, there is a strictly monotone sequence $(f^{n}_{m})_{m\in\mathbb{N}}\subseteq X\setminus\{k_{n}\}$ such that $\lim\limits_{m\longrightarrow\infty}f^{n}_{m}=k_{n}$ for all $n\in\mathbb{N}$. Furthermore, upon possibly truncating the sequences $(f^{n}_{m})_{m\in\mathbb{N}}$, we can assume that for all $n\in\mathbb{N}$,

(1) $(f^{n}_{m})_{m\in\mathbb{N}}\subseteq(-\infty,k_{n})$ or $(f^{n}_{m})_{m\in\mathbb{N}}\subseteq(k_{n},+\infty)$;

(2) $(f^{n}_{m})_{m\in\mathbb{N}}$ is strictly increasing, whenever $(f^{n}_{m})_{m\in\mathbb{N}}\subseteq(-\infty,k_{n})$ and $(f^{n}_{m})_{m\in\mathbb{N}}$ is strictly decreasing, whenever $(f^{n}_{m})_{m\in\mathbb{N}}\subseteq(k_{n},+\infty)$;

(3) $(f^{n}_{m})_{m\in\mathbb{N}}\subseteq(-\infty,k_{n+1})$;

(4) $(f^{n}_{m})_{m\in\mathbb{N}}\subseteq\bigcap\limits_{m\in\mathbb{N}}(-\infty,f^{n+1}_{m})$.

Define a countable family $(\mathcal{U}_{n})_{n\in\mathbb{N}}$ consisting of open sets of $X$ as follows:

$\mathcal{U}_{n}=\{U^{n}_{m}\mid m\in\mathbb{N}\}$, where $\forall \ m\in\mathbb{N}$

 if $ (f^{n+1}_{m})_{m\in\mathbb{N}}\subseteq(-\infty,k_{n+1})$, $\ U^{n}_{m}= [(-\infty,k_{n})\cap X]\cup[(k_{n},f^{n+1}_{m})\cap X]\cup[(k_{n+1},+\infty)\cap X]$;

 if $(f^{n+1}_{m})_{m\in\mathbb{N}}\subseteq(k_{n+1},+\infty)$, $\ U^{n}_{m}=[(-\infty,k_{n+1})\cap X]\cup[(f^{n+1}_{m},k_{n+2})\cap X]\cup[(k_{n+2},+\infty)\cap X]$. Let $\mathcal{U}=\bigcup\limits_{m\in\mathbb{N}}\mathcal{U}_{m}$. Obviously, $\mathcal{U}\cap\mathscr{F}(X)=\emptyset$. Moreover, we can see that

 $$\bigcup\mathcal{U}_{n}=X\setminus\{k_{n},k_{n+1}\},\ \mbox{if}\ (f^{n+1}_{m})_{m\in\mathbb{N}}\subseteq(-\infty,k_{n+1})$$ and
 $$\bigcup\mathcal{U}_{n}=X\setminus\{k_{n+1},k_{n+2}\},\ \mbox{if}\ (f^{n+1}_{m})_{m\in\mathbb{N}}\subseteq(k_{n+1},+\infty).$$ This implies that $\mathcal{U}$ is an open cover of $X$.

 Claim 1: $\mathcal{U}$ is irreducible in the upper topology of $\mathcal{O}(X)$.

It suffices to show that $cl_{\upsilon(\mathcal{O}(X))}(\mathcal{U})=\mathcal{O}(X)$. Assume that $cl_{\upsilon(\mathcal{O}(X))}(\mathcal{U})\neq\mathcal{O}(X)$. Then there exists a nonempty family $\{\mathcal{F}_{i}\mid i\in I\}\subseteq\mathbf{Fin}(\mathcal{O}(X))$ such that $cl_{\upsilon(\mathcal{O}(X))}=\bigcap\limits_{i\in I}\dn_{\mathcal{O}(X))}\mathcal{F}_{i}$. So we have $\mathcal{U}\subseteq\dn_{\mathcal{O}(X)}\mathcal{F}_{i}$ for all $i\in I$ and $X\not\in\bigcap\limits_{i\in I}\mathcal{F}_{i}$ since $cl_{\upsilon(\mathcal{O}(X))}(\mathcal{U})\neq\mathcal{O}(X)$. Note that $\{X\setminus\{x\}\mid x\in X\}$ is exactly the set of all maximal elements in $\mathcal{O}(X)$. This means that for every $i\in I$, there are $a_{1},a_{2},\cdot\cdot\cdot,a_{s}\in X$ such that $\mathcal{F}_{i}\subseteq\dn\{X\setminus\{a_{1}\},X\setminus\{a_{2}\},\cdot\cdot\cdot,X\setminus\{a_{s}\}\}$. In addition, $[k,+\infty)\cap X=\bigcap\limits_{n\in \mathbb{N}}(k_{n},+\infty)\cap X\subseteq U^{p}_{q}$ for all $p,q\in\mathbb{N}$. If $\{a_{1},a_{2},\cdot\cdot\cdot,a_{s}\}\cap[k, +\infty)\neq\emptyset$, then $\mathcal{F}_{i}\subseteq\dn\{X\setminus\{a\}\mid a\in A\}$, where $A=\{a_{t}\mid 1\leq t\leq s\}\cap(-\infty,k)$. Based on this consideration, we can require $\{a_{1},a_{2},\cdot\cdot\cdot,a_{s}\}\subseteq(-\infty, k)$. Since $(k_{n})_{n\in\mathbb{\mathbb{N}}}$ is strictly increasing and $\lim\limits_{n\longrightarrow\infty}k_{n}=k$, there is an $N\in\mathbb{N}$ such that $max\{a_{1},a_{2},\cdot\cdot\cdot,a_{s}\}<k_{N}$. Now we consider $U^{N}_{0}\in\mathcal{U}_{N}\subseteq\mathcal{U}$. Whence, $$U^{N}_{0}\in\dn\{X\setminus\{a_{1}\},X\setminus\{a_{2}\},\cdot\cdot\cdot,X\setminus\{a_{s}\}\}.$$ It follows from $max\{a_{1},a_{2},\cdot\cdot\cdot,a_{s}\}<k_{N}$ that $\{a_{1},a_{2},\cdot\cdot\cdot,a_{s}\}\subseteq(-\infty,k_{N})\cap X\subseteq U^{N}_{0}$. This is a contradiction. Thus, $cl_{\upsilon(\mathcal{O}(X))}(\mathcal{U})=\mathcal{O}(X)$ and Claim 1 holds.

By Claim 1 and $X\not\in\mathcal{U}$, we know that $X$ is not weakly $\upsilon$-compact, a contradiction. Therefore, $lim(X)$ is finite.

${\rm (b)\Longrightarrow(a)}$ Let $\mathcal{V}\subseteq\mathcal{O}(X)$ be an irreducible cover of $X$ with respect to the upper topology of $\mathcal{O}(X)$ and $\mathcal{V}\cap\mathscr{F}(X)$ is finite. If $X$ is finite, then it is evident that $X\in\mathcal{V}$. Otherwise, we suppose $lim(X)=\{x_{0},x_{1},\cdot\cdot\cdot,x_{m}\}$. Let $\mathcal{U}$ be the set of all elements $V\in\mathcal{V}$ with $\{x_{0},x_{1},\cdot\cdot\cdot,x_{m}\}\subseteq V$.

Claim 2:  $\mathcal{U}\subseteq\mathscr{F}(X)$.

Assume that there is a $W\in\mathcal{U}$ such that $W\not\in\mathscr{F}(X)$. Since $\{x_{0},x_{1},\cdot\cdot\cdot,x_{m}\}\subseteq W\in\mathcal{O}(X)$, there are finitely many open intervals $$\{(a_{k},b_{k})\mid a_{k}<b_{k}, k=0,1,\cdot\cdot\cdot,m\}$$ such that $x_{k}\in(a_{k},b_{k})\cap X\subseteq W$ for all $0\leq k\leq m$. It follows from $W\not\in\mathscr{F}(X)$ that there is a strictly increasing (or strictly decreasing) sequence $(c_{n})_{n\in\mathbb{N}}\subseteq X\setminus W$. Since $X$ is a bounded closed subset of $\mathbb{R}$, $(x_{n})_{n\in\mathbb{N}}$ is a convergent sequence in $\mathbb{R}$ and we let $\lim\limits_{n\longrightarrow\infty}c_{n}=c$. Then $c_{n}< c$ (or $c< c_{n}$) for all $n\in \mathbb{N}$ and $c\in X$ since $X$ is closed. So we have $c\in lim(X)$ and hence $c=x_{s}$ for some $1\leq s\leq m$. Hence, there is a natural number $N$ such that for all $N\leq m$, $c_{m}\in(a_{s},b_{s})$. This is a contradiction. So the Claim 2 holds.

By Claim 2, $\mathcal{U}$ is finite because $\mathcal{V}\cap\mathscr{F}(X)$ is finite.  If $X\not\in \mathcal{V}$, we can conclude that $$\mathcal{V}\subseteq \dn_{\mathcal{O}(X)}\{X\setminus\{x_{0}\},X\setminus\{x_{1}\},\cdot\cdot\cdot,X\setminus\{x_{n}\}\}\cup\dn_{\mathcal{O}(X)} \mathcal{U}.$$
Since $(\mathcal{O}(X),\upsilon(\mathcal{O}(X)))$ is sober and $\bigcup\mathcal{V}=X$, $cl_{\upsilon(\mathcal{O}(X))}(\mathcal{V})=\dn_{\mathcal{O}(X)} X=\mathcal{O}(X)$. Whence $$cl_{\upsilon(\mathcal{O}(X))}(\mathcal{V})=\mathcal{O}(X)\subseteq\dn_{\mathcal{O}(X)}\{X\setminus\{x_{0}\},X\setminus\{x_{1}\},\cdot\cdot\cdot,X\setminus\{x_{n}\}\}\cup\dn \mathcal{U}.$$ It is a contradiction. Therefore $X\in\mathcal{V}$ and $X$ is weakly $\upsilon$-compact.\end{proof}

\begin{remark}\ {\rm Theorem 5.4 is still valid for the higher dimension space $\mathbb{R}^{n}$. Specifically, a compact subset $M\subseteq \mathbb{R}^{n}$ is weakly $\upsilon$-compact iff $lim(M)$ is finite.}
\end{remark}

\begin{example} {\rm Let $X$ be the topological space given in Example 5.3(4). Then $X$ is weakly $\upsilon$-compact. But $X\times X\subseteq\mathbb{R}^{2}$ has infinitely limit points. Hence $X\times X$ is not weakly $\upsilon$-compact.}
\end{example}

\quad The Sorgenfrey line $\mathbb{R}_{r}$ is a topological space on $\mathbb{R}$ having a base $\{[x,y)\mid x<y\}$. Dually, we equip $\mathbb{R}$ with the topology having a base $\{(x,y]\mid x<y\}$ and the corresponding topological space is denoted by $\mathbb{R}_{l}$.

\begin{lemma}(see \cite{XSW87}) A subset $A$ of $\mathbb{R}$ is compact in Sorgenfrey line iff $A\subseteq\mathbb{R}$ is bounded and closed in real line $\mathbb{R}$ and $A$ has no accumulation points in $\mathbb{R}_{l}$.
\end{lemma}

\quad Given a subset $A\subseteq\mathbb{R}$, $x$ is a left(right) limit point of $A$ if there is a strictly increasing (decreasing) sequence $(x_{n})_{n\in\mathbb{N}}\subseteq A\setminus\{x\}$ such that $\lim\limits_{n\longrightarrow\infty}x_{n}=x$ in real line $\mathbb{R}$. The set of all left(right) limit points of $A$ is denoted by $l(A)$ ($r(A)$).

\begin{proposition}  Let $A\subseteq\mathbb{R}$. Then $x$ is an accumulation point of $A$ in $\mathbb{R}_{l}$ iff $x\in l(A)$. Dually, $x$ is an accumulation point of $A$ in $\mathbb{R}_{r}$ iff $x\in r(A)$.

\end{proposition}

\begin{proof}Suppose $x$ is an accumulation point of $A$ in $\mathbb{R}_{l}$. Then for every open neighbourhood $U$ of $x$ which is open in the subspace $A$ of $\mathbb{R}_{l}$, $U\cap (A\setminus\{x\})\neq\emptyset$.

Assuming that $U_{n}=(x-\frac{1}{n},x]\cap (x_{n-1}]\cap A$ and $x_{1},\cdot\cdot\cdot,x_{n}$ are defined, we set $U_{n+1}=(x-\frac{1}{n+1},x]\cap(x_{n},x]\cap A$ and take for  $x_{n+1}$ any element of $U_{n+1}\cap (A\setminus\{x\})$ which is nonempty by our assumption $U_{n}$ is open in the subspace $A$ of $\mathbb{R}_{l}$.

Then, we obtain a sequence $(x_{n})_{n\in\mathbb{N}}\subseteq A$ such that for all $n\in\mathbb{N}$, $x_{n}< x_{n+1}$ and $x-\frac{1}{n+1}<x_{n+1}<x$. Then $(x_{n})_{n\in\mathbb{N}}\subseteq A$ is strictly increasing and $\lim\limits_{n\longrightarrow\infty}x_{n}=x$. Equivalently, $x\in l(A)$.

Conversely, $x\in l(A)$ implies that there is a strictly increasing sequence $(a_{n})_{n\in\mathbb{N}}\subseteq A\setminus\{x\}$ such that $\lim\limits_{n\longrightarrow\infty}a_{n}=x$. If $a<x$, then there is a $N\in\mathbb{N}$ such that $|a_{N}-x|<x-a$. This implies $a_{N}\in(a,x]\cap A\setminus\{x\}$. Hence, $x$ is an accumulation point of $A$. The proof of the second statement is similar.\end{proof}

\begin{theorem} Let $X\subseteq\mathbb{R}_{r}$ be a subspace of the Sorgenfrey line. Then the following statements are equivalent.

(a) $X$ is weakly $\upsilon$-compact;

(b) $X\subseteq\mathbb{R}$ is a bounded closed subset in the real line $\mathbb{R}$ such that $l(X)=\emptyset$ and $r(X)$ is finite.
\end{theorem}

\begin{proof}${\rm (a)\Longrightarrow(b)}$ By Proposition 5.2, Lemma 5.7 and Proposition 5.8, $X\subseteq\mathbb{R}$ is a bounded closed subset and $l(X)=\emptyset$. Assume that $r(X)$ is infinite. Then there is a strictly monotone sequence $(k_{n})_{n\in\mathbb{N}}\subseteq r(X)$. If $(k_{n})_{n\in\mathbb{N}}$ is strictly increasing, then the sequence $(k_{n})_{n\in\mathbb{N}}$ is convergent in the real line $\mathbb{R}$ and let $k$ be its limit. One checks that $k$ is a left limit point of $X$. This is impossible since $l(X)=\emptyset$. Consequently, the sequence $(k_{n})_{n\in\mathbb{N}}$ is strictly decreasing and $k=\lim\limits_{n\longrightarrow\infty}k_{n}\in r(X)$. Since $k_{n}\in r(X)$, there is a strictly decreasing sequence $(h^{n}_{m})_{m\in\mathbb{N}}\subseteq X\setminus\{k_{n}\}$ such that $k_{n}=\lim\limits_{m\longrightarrow\infty}h^{n}_{m}$ for all $n\in\mathbb{N}$. Actually, we can further require that for all $n\in\mathbb{N}$

(1) $(h^{n}_{m})_{m\in\mathbb{N}}\subseteq(k_{n},+\infty)$;

(2) $(h^{n}_{m})_{m\in\mathbb{N}}$ is strictly decreasing;

(3) $(h^{n+1}_{m})_{m\in\mathbb{N}}\subseteq(-\infty,k_{n})$.

Define a family $(\mathcal{S}_{n})_{ n\in\mathbb{N}}$ such that each $\mathcal{S}_{n}$ consists of countably many open sets of $X$ as follows:

$\mathcal{S}_{n}=\{S^{n}_{m}\mid m\in\mathbb{N}\}$, where for $m\in\mathbb{N}$,
\setlength{\baselineskip}{0.9\baselineskip}
$$S^{n}_{m}=((-\infty,k_{n+1})\cap X)\cup([h^{n+1}_{m},k_{n})\cap X)\cup([k_{n},+\infty)\cap X).$$
Let $\mathcal{S}=\bigcup\limits_{m\in\mathbb{N}}\mathcal{S}_{m}$. It is evident that $\mathcal{S}\cap\mathscr{F}(X)=\emptyset$. Clearly, $\bigcup\mathcal{S}_{n}=X\setminus\{k_{n+1}\}$ for every $n\in\mathbb{N}$ and thus $\mathcal{S}$ covers $X$.

Claim 1: $\mathcal{S}$ is irreducible in the upper topology of $\mathcal{O}(X)$.

Assume that $X\not\in cl_{\upsilon(\mathcal{O}(X))}(\mathcal{S})$. Suppose $cl_{\upsilon(\mathcal{O}(X))}(\mathcal{S})=\bigcap\limits_{i\in I}\dn_{\mathcal{O}(X))}\mathcal{F}_{i}$, where $\{\mathcal{F}_{i}\mid i\in I\}\subseteq\mathbf{Fin}(\mathcal{O}(X))$ is nonempty. Then we have $\mathcal{S}\subseteq\dn_{\mathcal{O}(X)}\mathcal{F}_{i}$ for all $i\in I$ and $X\not\in\bigcap\limits_{i\in I}\mathcal{F}_{i}$. Note that $\{X\setminus\{x\}\mid x\in X\}$ is exactly the set of all maximal elements in $\mathcal{O}(X)$. Then we have for every $i\in I$, $\mathcal{F}_{i}\subseteq\dn_{\mathcal{O}(X)}\{X\setminus\{s_{1}\},X\setminus\{s_{2}\},\cdot\cdot\cdot,X\setminus\{s_{p}\}\}$ for $s_{1},s_{2},\cdot\cdot\cdot,s_{p}\in X$. It follows that $\mathcal{S}\subseteq\dn_{\mathcal{O}(X)}\{X\setminus\{s_{1}\},X\setminus\{s_{2}\},\cdot\cdot\cdot,X\setminus\{s_{p}\}\}$. As $(-\infty,k]\subseteq S^{n}_{m}$ for all $m,n\in\mathbb{N}$, we can require $\{s_{1},s_{2},\cdot\cdot\cdot,s_{p}\}\subseteq(k,+\infty)$. Since $k=\lim\limits_{n\longrightarrow\infty}k_{n}$ and $(k_{n})_{n\in\mathbb{N}}$ is strictly decreasing, there exists a $N\in\mathbb{N}$ such that $k_{N}<min\{s_{1},s_{2},\cdot\cdot\cdot,s_{p}\}$. Consider $S^{N}_{0}\in \mathcal{S}_{N}\subseteq\mathcal{S}$. Then $S^{N}_{0}\subseteq\dn_{\mathcal{O}(X)}\{X\setminus\{s_{1}\},X\setminus\{s_{2}\},\cdot\cdot\cdot,X\setminus\{s_{p}\}\}$. Obviously, $\{s_{1},s_{2},\cdot\cdot\cdot,s_{p}\}\subseteq[k_{N},+\infty)\cap X\subseteq S^{N}_{0}$, impossible. Thus, $X\in cl_{\upsilon(\mathcal{O}(X))}(\mathcal{S})$ and equivalently, $cl_{\upsilon(\mathcal{O}(X))}(\mathcal{S})=\mathcal{O}(X)$. Claim 1 is proved.

By Claim 1, we can assert that $X$ is not weakly $\upsilon$-compact. This is again a contradiction. Hence, we can conclude that $r(X)$ is finite.

${\rm (b)\Longrightarrow(a)}$ Suppose $X\subseteq\mathbb{R}$ is a bounded closed subset in the real line $\mathbb{R}$ such that $l(X)=\emptyset$ and $r(X)$ is finite. Clearly, $X$ is finite whenever $r(X)=\emptyset$. Hence $X$ is weakly $\upsilon$-compact. So we suppose $r(X)=\{x_{0},x_{1},\cdot\cdot\cdot,x_{m}\}$.  Let $\mathcal{U}$ be an irreducible open cover of $X$ in the upper topology $\mathcal{O}(X)$ such that $\mathcal{U}\cap\mathscr{F}(X)$ is finite. We use $\mathcal{V}$ to denote the set of all elements $U\in\mathcal{U}$ satisfying $\{x_{0},x_{1},\cdot\cdot\cdot,x_{m}\}\subseteq U$.

Claim 2: $\mathcal{V}\subseteq\mathscr{F}(X)$.

Assume that there is a $V_{0}\in\mathcal{V}$ such that $V_{0}\not\in\mathscr{F}(X)$. Then $X\setminus V_{0}$ is infinite. So there is a strictly monotone sequence $(a_{n})_{n\in\mathbb{N}}\subseteq X\setminus V_{0}$. As $X\subseteq\mathbb{\mathbb{R}}$ is bounded closed in the real line $\mathbb{R}$, $\lim\limits_{n\longrightarrow\infty}a_{n}=a$ exists and $a\in X$ with $(a_{n})_{n\in\mathbb{N}}\subseteq X\setminus\{a\}$. If $(a_{n})_{n\in\mathbb{N}}$ is strictly increasing, then $a\in l(X)$, a contradiction. If $(a_{n})_{n\in\mathbb{N}}$ is strictly decreasing, $a\in r(X)$. This means that $a\in V_{0}$ and then there is a $[a,r)\cap X\subseteq V_{0}$ for some $a<r$. Then there exists a $N\in\mathbb{N}$ such that $a_{N}\in[a,r)\cap X$. This is a contradiction. So Claim 2 holds.

By Claim 2, $\mathcal{V}$ is finite. If $X\not\in\mathcal{U}$, we have
$$\mathcal{U}\subseteq\dn_{\mathcal{O}(X)}\{X\setminus\{x_{0}\},X\setminus\{x_{1}\},\cdot\cdot\cdot,X\setminus\{x_{m}\}\}\cup\dn_{\mathcal{O}(X)} \mathcal{V}.$$ Note that $cl_{\upsilon(\mathcal{O}(X))}\mathcal{U}=\mathcal{O}(X)$. Consequently,$$\mathcal{O}(X)\subseteq\dn_{\mathcal{O}(X)}\{X\setminus\{x_{0}\},X\setminus\{x_{1}\},\cdot\cdot\cdot,X\setminus\{x_{m}\}\}\cup\dn_{\mathcal{O}(X)} \mathcal{V}.$$ This is a contradiction. Therefore $X\in\mathcal{U}$ and then $X$ is weakly $\upsilon$-compact.\end{proof}

\section{The $\mathcal{T}$-compactness of Subspaces, product spaces and power spaces}

\quad In this section, we discuss the properties of $SI$-compactness, $\upsilon$-compactness and weak $\upsilon$-compactness. Especially, the heredity for closed subspaces of those kinds of compactness is obtained.

\begin{theorem} Every topological space $(X,\tau)$ is an open subspace of a $\upsilon$-compact space.
\end{theorem}

\begin{proof}Let $Y$ be a set with $X\subseteq Y$ and $Y\setminus X\neq\emptyset$. Equip $Y\setminus X$ with a topology $\sigma$ such that $(Y\setminus X,\sigma)$ is $\upsilon$-compact. Put $\rho=\tau\cup\{X\cup U\mid U\in\sigma\}$. It is easy to show that $\rho$ is a topology on $Y$ and $X$ is a dense open subspace of $(Y,\rho)$. Define a map $\psi:\rho\longrightarrow\sigma$ by
$$\forall \ W\in \rho,\ \ \psi(W)=W\setminus X.$$
Then for every $V\in\sigma$, $\psi^{-1}(\sigma\setminus \dn V)=\rho\setminus\dn(X\cup V)$. It follows that $\psi$ is continuous with respect to the upper topology. Let $\mathcal{U}\subseteq\rho$ be an irreducible open cover of $Y$ in the upper topology of $\mathcal{O}(Y)$. By the continuity of $\psi$, $\psi(\mathcal{U})$ is irreducible in the upper topology of $\mathcal{O}(Y\setminus X)$ and $\psi(\mathcal{U})$ covers $Y\setminus X$. Since  $(Y\setminus X,\sigma)$ is $\upsilon$-compact, there is a $U\in\mathcal{U}$ such that $\psi(U)=Y\setminus X$. So we can conclude that $U=Y$ and hence $(Y,\rho)$ is $\upsilon$-compact.\end{proof}

\begin{theorem} Every topological space $(X,\tau)$ is an open subspace of a weakly $\upsilon$-compact space.

\end{theorem}

\begin{proof}It is similar to the proof of Theorem 6.1.\end{proof}

\begin{theorem} Let $(X,\tau)$ be a weakly $\upsilon$-compact topological space and $Y\subseteq X$ a closed set. Then the subspace $Y$ is weakly $\upsilon$-compact.

\end{theorem}

\begin{proof}It is clear when $Y$ is finite, so assume this is not the case. Let $\mathcal{V}\subseteq\tau|_{Y}$ be an irreducible cover of $Y$ with respect to the upper topology of $\mathcal{O}(Y)$ and $\mathcal{V}\cap\mathscr{F}(Y)$ is finite. Define a map $g:\mathcal{O}(Y)\longrightarrow\mathcal{O}(X)$ by
$$\forall\ U\in\mathcal{O}(Y),\ g(U)=(X\setminus Y)\cup(\bigcup\{V\in\mathcal{O}(X)\mid V\cap Y=U\}).$$
Take a $W\in\mathcal{O}(X)$, $g^{-1}(\dn_{\mathcal{O}(X)}W)=\{G\in\mathcal{O}(Y)\mid g(G)\subseteq W\}$. Obviously, $g^{-1}(\dn_{\mathcal{O}(X)}W)=\emptyset$ if $X\setminus Y\not\subseteq W$. If $X\setminus Y\subseteq W$, we can verify  $g^{-1}(\dn_{\mathcal{O}(X)}W)=\dn_{\mathcal{O}(Y)}(W\cap Y)$. This means that $g$ is continuous with respect to the upper topology on $\mathcal{O}(X)$ and $\mathcal{O}(Y)$. Then $g(\mathcal{V})$ is still irreducible in the upper topology of $\mathcal{O}(X)$ and $g(\mathcal{V})$ coves $X$. In addition, as $Y$ and $\mathcal{V}\cap\mathscr{F}(Y)$ are infinite, $g(\mathcal{V})\cap\mathscr{F}(X)$ is finite. By the weak $\upsilon$-compactness of $X$, $X\in g(\mathcal{V})$. Whence, $X=g(V_{0})=\bigcup\mathcal{D}$ for some $V_{0}\in\mathcal{V}$, where $$\mathcal{D}=\{(X\setminus Y)\cup H\mid H\cap Y=V_{0}, H\in\mathcal{O}(X)\}.$$ By Theorem 6.1, $X$ is compact. Since $\mathcal{D}$ is directed and $X=\bigcup\mathcal{D}$, $X\in\mathcal{D}$. Equivalently, $Y=V_{0}\in\mathcal{V}$. Therefore, $Y$ is weakly $\upsilon$-compact.\end{proof}

\begin{theorem}Let $(X,\tau)$ be a $\upsilon$-compact topological space and $Y\subseteq X$ a closed set. Then the subspace $Y$ is  $\upsilon$-compact.
\end{theorem}
\begin{proof}It is similar to the proof of Theorem 6.3.\end{proof}

\begin{proposition} Let $(X,\tau)$ be a topological space and $Y\subseteq X$ such that $\{X\}=\{U\in\tau\mid Y\subseteq U\}$. The following statements hold:

{\rm(1)}  the subspace $Y$ is compact iff $X$ is compact;

{\rm(2)} if the subspace $Y$ is SI-compact, then $X$ is $SI$-compact.

\end{proposition}

\begin{proof}(1) Assume first that $Y$ is compact. Let $\mathcal{V}\subseteq\tau$ be a cover of $X$. Then $\mathcal{V}$ covers $Y$ and hence by the compactness of $Y$, there are finitely $V_{1},V_{2},\cdot\cdot\cdot, V_{n}\in\mathcal{V}$ such that $Y\subseteq\bigcup\limits_{1\leq s\leq n} V_{s}$. By assumption, $X=\bigcup\limits_{1\leq s\leq n} V_{s}$. So $X$ is compact.

Conversely, let $\mathcal{U}$ be an open cover of $Y$. By assumption, $X=\bigcup\mathcal{U}$. Since $X$ is compact, there are $U_{1},U_{2},...,U_{m}\in\mathcal{U}$ such that $X=\bigcup\limits_{1\leq t\leq m}U_{t}$. This means that $Y$ is compact.

(2) Let $\mathcal{U}=\{U_{i}\mid U_{i}\in\tau, i\in I\}$ be a Scott irreducible open cover of $X$. Then $\hat{\mathcal{U}}=\{U_{i}\cap Y\mid i\in I\}$ covers $Y$.  Define a map $f_{Y}:\Sigma(\tau)\longrightarrow\Sigma(\tau|_{Y})$ by
$$\forall\ V\in\tau, f_{Y}(V)=Y\cap V.$$
Clearly, $f_{Y}$ is a continuous map since $f_{Y}$ preserves suprema of directed sets. Then $\hat{\mathcal{U}}=f(\mathcal{U})$ is also irreducible in $\Sigma(\tau|_{Y})$. It follows from the SI-compactness of $Y$ that $Y\in\hat{\mathcal{U}}$. This means that $A\subseteq U_{i}$ for $i\in I$. By the assumption, $X=U_{i}\in\mathcal{U}$. Therefore, $X$ is $SI$-compact.\end{proof}

\begin{theorem} Every space $(X,\tau)$ is an open subspace of a SI-compact space.
\end{theorem}
\begin{proof}Let $Y$ be a set and $X\subseteq Y$ with $Y\setminus X\neq\emptyset$. Equip $Y\setminus X$ with a topology $\delta$ such that $(Y\setminus X,\delta)$ is SI-compact. Set $\omega=\tau\cup\{X\cup U\mid U\in\delta\}$. It is easy to show that $\omega$ is a topology on $Y$ and $X$ is a dense open subspace of $(Y,\omega)$. Define a map $\varphi:\omega\longrightarrow\delta$ by
$$ \forall \ W\in \omega,\ \ \varphi(W)=W\setminus X.$$
Then $\varphi$ is continuous with respect to the Scott topology. Let $\mathcal{U}\subseteq\omega$ be a Scott irreducible open cover of $Y$. By the continuity of $\varphi$, $\varphi(\mathcal{U})$ is Scott irreducible and $\varphi(\mathcal{U})$ covers $Y\setminus X$. Since  $(Y\setminus X,\delta)$ is SI-compact, there is a $U\in\mathcal{U}$ such that $\varphi(U)=Y\setminus X$. So we can conclude that $U=Y$ and hence $(Y,\omega)$ is SI-compact.\end{proof}

\begin{remark}{\rm By Theorem 6.6,  open subspaces of a $SI$-compact space are not $SI$-compact in general.}
\end{remark}

\begin{theorem} Let $(X,\tau)$ be a $SI$-compact topological space and $Y\subseteq X$ a closed set. Then the subspace $Y$ is $SI$-compact.
\end{theorem}
\begin{proof}Let $\mathcal{U}\subseteq \tau|_{Y}$ be a Scott irreducible open cover of $Y$. Define the map $g:\tau|_{Y}\longrightarrow \tau$ by $$\forall\ U\in\tau|_{Y}, g(U)=\bigcup\{V\cup (X\setminus Y)\mid V\in\tau, V\cap Y=U\}.$$

Claim: $g(\bigcup\limits_{i\in I}U_{i})=\bigcup\limits_{i\in I}g(U_{i})$ for $\{U_{i}\mid i\in I\}\subseteq\tau|_{Y}$.

 Actually, if $x\in g(\bigcup\limits_{i\in I}U_{i})$, then $x\in V\in\tau$ with $V\cap Y=\bigcup\limits_{i\in I}U_{i}$ or $x\in X\setminus Y$. As $X\setminus Y\subseteq\bigcup\limits_{i\in I}g(U_{i})$, we need only to consider the case where $x\in Y$. This means that $x\in U_{i_{0}}$ for some $i_{0}\in I$ and hence $x\in g(U_{i_{0}})\subseteq\bigcup\limits_{i\in I}g(U_{i})$. Conversely, if $y\in\bigcup\limits_{i\in I}g(U_{i})$, then $y\in g(U_{i_{1}})$ for some $i_{1}\in I$. Similarly, we suppose $y\in Y$. It follows that there is a $W\in\tau$ with $W\cap Y=U_{i_{1}}$ such that $y\in W$. As $\{U_{i}\mid i\in I\}\subseteq\tau|_{Y}$, we can choose for each $j\in I\setminus\{i_{1}\}$, a $W_{j}\in \tau$ satisfying $W_{j}\cap Y=U_{j}$. Let $W_{I}=W\cup(\bigcup\limits_{j\in I\setminus\{i_{1}\}}W_{j})$.
 One can check that $y\in W_{I}$ and $W_{I}\subseteq g(\bigcup\limits_{i\in I}U_{i})$.

 By this Claim, $g$ is continuous with respect to the Scott topology. So we know $g(\mathcal{U})$ is Scott irreducible in $\Sigma (\tau)$. Furthermore, $g(\mathcal{U})$ covers $X$. By the $SI$-compactness of $X$, $X\in g(\mathcal{U})$. Whence $X=g(U_{0})=\bigcup\mathcal{D}$, where $$\mathcal{D}=\{(X\setminus Y)\cup V\mid V\in\tau ,V\cap Y=U_{0}\}$$ for some $U_{0}\in\mathcal{U}$. Since $\mathcal{D}$ is directed, $X\in\mathcal{D}$ by using the $SI$-compactness of $X$ again. Thus $U_{0}=Y\in\mathcal{U}$ and then $Y$ is $SI$-compact.\end{proof}

\begin{definition} A topological space $X$ is called weakly $SI$-compact if there is a base $\mathcal{B}$ of the topology of $X$ such that a Scott irreducible open cover $\mathcal{U}$ of $X$ contains $X$ as soon as $\mathcal{U}$ is contained in $\mathcal{B}$.
\end{definition}

\begin{theorem} Let $X$ and $Y$ be compact spaces. If $\Sigma \mathcal{O}(X)$ and $\Sigma \mathcal{O}(Y)$ are sober, then $X\times Y$ is weakly $SI$-compact.

\end{theorem}

\begin{proof}Define a map $f:\Sigma\mathcal{O}(X\times Y)\longrightarrow\Sigma \mathcal{O}(X)\times\Sigma \mathcal{O}(Y)$ as follows:
$$\forall\ M\in\mathcal{O}(X\times Y), f(M)=(pr_{X}(M),pr_{Y}(M)),$$ where $pr_{X},pr_{Y}$ are projections.

Claim : $f$ is continuous.

Let $\mathcal{A}\in\sigma(\mathcal{O}(X))$ and $\mathcal{B}\in\sigma(\mathcal{O}(X))$. There is no doubt that $f^{-1}(\mathcal{A}\times\mathcal{B})$ is an upper set. If $\mathcal{D}$ is a directed subset of $\mathcal{O}(X\times Y)$ with $\bigcup\mathcal{D}\in f^{-1}(\mathcal{A}\times\mathcal{B})$, then $\bigcup pr_{X}(\mathcal{D})\in\mathcal{A}$ and $\bigcup pr_{Y}(\mathcal{D})\in\mathcal{B}$. Since $\mathcal{A}$ and $\mathcal{B}$ are Scott open, there are $D_{1},D_{2}\in\mathcal{D}$ such that $pr_{X}(D_{1})\in\mathcal{A}$ and $pr_{Y}(D_{2})\in\mathcal{B}$. Using the directedness of $\mathcal{D}$, take a $D_{3}\in\mathcal{D}$ with $D_{1}\cup D_{2}\subseteq D_{3}$. Then we know $D_{3}\in\mathcal{D}\cap f^{-1}(\mathcal{A}\times\mathcal{B})$.

Consider $\mathcal{S}=\{\ U\times V\mid U\in \mathcal{O}(X), V\in\mathcal{O}(Y)\}$. Then $\mathcal{S}$ is a base of $X\times Y$. Let $\mathcal{U}\subseteq\mathcal{S}$ be a Scott irreducible open cover of $X\times Y$. Assume that $X\times Y\not \in \mathcal{U}$. By Claim, $f(\mathcal{U})$ is irreducible in the product space $\Sigma \mathcal{O}(X)\times\Sigma \mathcal{O}(Y)$. As $\Sigma \mathcal{O}(X)$ and $\Sigma \mathcal{O}(Y)$ are sober, $\Sigma \mathcal{O}(X)\times\Sigma \mathcal{O}(Y)$ is also sober. So there are $M\in\mathcal{O}(X)$ and $N\in\mathcal{O}(Y)$ such that $$cl_{\Sigma \mathcal{O}(X)\times\Sigma \mathcal{O}(Y)}(f(\mathcal{U}))=\downarrow_{\mathcal{O}(X)\times{O}(Y)}(M,N)=\downarrow_{\mathcal{O}(X)}M\times\downarrow_{\mathcal{O}(Y)}N.$$
Then we have $f(\mathcal{U})\subseteq\downarrow_{\mathcal{O}(X)}M\times\downarrow_{\mathcal{O}(Y)}N$. This means that $\bigcup pr_{X}(\mathcal{U})\subseteq M$ and $\bigcup pr_{Y}(\mathcal{U})\subseteq N$. Equivalently, $X=pr_{X}(\bigcup\mathcal{U})=M$ and $Y=pr_{Y}(\bigcup\mathcal{U})=N$. Since $X$ and $Y$ are compact, $\mathcal{O}(X)\setminus\{X\}$ and $\mathcal{O}(Y)\setminus\{Y\}$ are Scott closed in $\Sigma \mathcal{O}(X)$ and $\Sigma \mathcal{O}(Y)$, respectively. Note that $$\mathcal{O}(X)\times\mathcal{O}(Y)\setminus\{(X,Y)\}=(\mathcal{O}(X)\setminus\{X\}\times\mathcal{O}(Y))\cup(\mathcal{O}(X)\times\mathcal{O}(Y)\setminus\{Y\}).$$ It follows that $\mathcal{O}(X)\times\mathcal{O}(Y)\setminus\{(X,Y)\}$ is closed in the product space $\Sigma \mathcal{O}(X)\times\Sigma \mathcal{O}(Y)$. Clearly, $f(\mathcal{U})\subseteq\mathcal{O}(X)\times\mathcal{O}(Y)\setminus\{(X,Y)\}$ and then $$cl_{\Sigma \mathcal{O}(X)\times\Sigma \mathcal{O}(Y)}(f(\mathcal{U}))=\mathcal{O}(X)\times\mathcal{O}(Y)\subseteq\mathcal{O}(X)\times\mathcal{O}(Y)\setminus\{(X,Y)\},$$ a contradiction. Thus, $X\times Y\in\mathcal{U}$ and $X\times Y$ is weakly $SI$-compact.\end{proof}

\quad Finally, we will consider the $SI$-compactness of powerspaces.

\begin{proposition} Let $X$ be a topological space. Then the following statements hold.

{\rm(1)} if $X$ is compact then $P_{S}(X)$ is SI-compact;

{\rm(2)} if $P_{S}(X)$ is compact, then $X$ is compact.
\end{proposition}

\begin{proof}(1) Let $\mathcal{U}\subseteq\mathcal{O}(P_{S}(X))$ be a Scott irreducible open cover of $\mathcal{Q}(X)$.  Since $X$ is compact, or equivalently $X\in\mathcal{Q}(X)$, $X\in \bigcup\limits_{i\in I}\Box U_{i}\in\mathcal{U}$. So $\mathcal{Q}(X)=\bigcup\limits_{i\in I}\Box U_{i}\in\mathcal{U}$ and hence $P_{S}(X)$ is SI-compact.

(2)Let $\mathcal{U}\subseteq\mathcal{O}(X)$ be an open cover of $X$. The collection of all sets having the form $\Box(\bigcup\mathcal{F})$ for a finite subset $\mathcal{F}\subseteq\mathcal{U}$ is denoted by $\mathcal{V}$. For every $K\in\mathcal{Q}(X)$, $K\subseteq\bigcup\mathcal{U}$ and hence $K\subseteq\bigcup\mathcal{F}_{0}$ for some finite subset $\mathcal{F}_{0}\subseteq\mathcal{U}$ by the compactness of $K$. This means that $\mathcal{V}$ covers $\mathcal{Q}(X)$. As $P_{S}(X)$ is compact, there are finite subsets $\mathcal{F}_{1},\mathcal{F}_{2},\cdot\cdot\cdot,\mathcal{F}_{n}\subseteq\mathcal{U}$ such that $$\mathcal{Q}(X)\subseteq(\Box(\bigcup\mathcal{F}_{1}))\cup(\Box(\bigcup\mathcal{F}_{2}))\cup\cdot\cdot\cdot\cup(\Box(\bigcup\mathcal{F}_{n})).$$ So we have $$X=\bigcup\mathcal{Q}(X)\subseteq(\bigcup\mathcal{F}_{1})\cup(\bigcup\mathcal{F}_{2}))\cup\cdot\cdot\cup(\bigcup\mathcal{F}_{n})).$$ Therefore, $X$ is compact.\end{proof}

\begin{corollary}Let $X$ be a topological space. Then the following statements are equivalent.

{\rm(1)} $X$ is compact;

{\rm(2)} $P_{S}(X)$ is compact;

{\rm(3)} $P_{S}(X)$ is SI-compact.

\end{corollary}

\quad Unfortunately, the SI-compactness of $P_{S}(X)$ does not imply the SI-compactness of $X$ (See example 6.13).

\begin{example} Let $X$ be a compact and non SI-compact space. By Proposition 6.12, $P_{S}(X)$ is SI-compact but $X$ is not.
\end{example}

\begin{proposition} Let $X$ be a topological space. Then the following statements are equivalent.

(a) $X$ is a SI-compact space;

(b) $P_{H}(X)$ is SI-compact.
\end{proposition}

\begin{proof}${\rm(a)\Longrightarrow(b)}$ Note that $\xi:X\longrightarrow P_{H}(X)$ is continuous, where $\xi(x)=cl(\{x\})$, for all$\ x\in X$. Thus $\xi^{-1}:\mathcal{O}(P_{H}(X))\longrightarrow\mathcal{O}(X)$ is continuous with respect to the Scott topology. Suppose $\mathcal{F}\subseteq\mathcal{O}(P_{H}(X))$ is a Scott irreducible open cover of $\Gamma_{0}(X)$. Then $\xi^{-1}(\mathcal{F})$ is also irreducible. Clearly, $\xi^{-1}(\mathcal{F})$ covers $X$ since $\{cl(\{x\})\mid x\in X\}\subseteq \bigcup\mathcal{F}$. It follows from the SI-compactness of $X$ that $X\in\xi^{-1}(\mathcal{F})$. Choose a $F\in\mathcal{F}$ with $X=\xi^{-1}(F)$. Let $F=\bigcup\limits_{i\in I}(\bigcap\limits_{j\in J_{i}}\diamond U_{j})$, where  $J_{i}$ is finite for each $i\in I$ and $\{U_{j}\mid j\in J_{i}, i\in I\}\subseteq\mathcal{O}(X)$. If $A\in\Gamma_{0}(X)$ and $a\in A$, then $a\in\xi^{-1}(\mathcal{F})$ and hence $cl(\{a\})\in F$. Clearly, $F$ is upper in $\Gamma_{0}(X)$. Then we have $A\in F$. This implies $\Gamma_{0}(X)=F$. Therefore, $P_{H}(X)$ is SI-compact.

${\rm(b)\Longrightarrow(a)}$ Define a map $\eta:\Sigma(\mathcal{O}(X))\longrightarrow\Sigma(\mathcal{O}(P_{H}(X)))$ by $\eta(U)=\diamond U$. Obviously, $\eta$ is continuous since $\eta(U)$ preserves arbitrary unions. Let $\mathcal{E}\subseteq\mathcal{O}(X)$ be a Scott irreducible open cover of $X$. Then $\eta(\mathcal{E})$ is also irreducible in $\Sigma(\mathcal{O}(P_{H}(X)))$. If $B\in\Gamma_{0}(X)$ and $b\in B$, then $b\in E$ for some $E\in\mathcal{E}$. So we have $B\in\diamond E$ and hence $\Gamma_{0}(X)=\bigcup\eta(\mathcal{E})$. As $P_{H}(X)$ is SI-compact, $\Gamma_{0}(X)\in\eta(\mathcal{E})$. Take a $E_{0}\in\mathcal{E}$ such that $\Gamma_{0}(X)=\diamond E_{0}$. Especially, $\{cl(\{x\})\mid x\in X\}\subseteq\diamond E_{0}$. It is easy to deduce that $X=E_{0}$. Thus $X$ is SI-compact.\end{proof}

\section{Conclusions }
In this paper, we mainly discussed the $SI$-compactness and $\upsilon$-compactness of topological spaces. Specifically, we gave a new method for obtaining a Scott non sober frame by using the non $SI$-compact space and we further got a countable $T_{1}$ space whose open set lattice is Scott non sober. However, we still did not know the answers of the following questions.

\begin{problem} {\rm Is every $\sigma_{\mathcal{C}}$-compact space $SI$-compact?}
\end{problem}

\begin{problem} {\rm Is $(\widehat{P},\omega(\widehat{P}))$ sober for the dcpo $\widehat{P}$ presented in Example 4.8?}
\end{problem}

\begin{problem} {\rm Let $X$ be a countable Hausdorff space. Is $\Sigma\mathcal{O}(X)$ sober? In particular, for the rational numbers subspace $\mathbb{Q}$ of real line $\mathbb{R}$, is  $\Sigma \mathcal{O}(\mathbb{Q})$ sober?}
\end{problem}

\begin{problem} {\rm Let $X$ and $Y$ be a pair of compact spaces such that $\Sigma \mathcal{O}(X)$ and $\Sigma \mathcal{O}(Y)$ are sober. Is $X\times Y$ always $SI$-compact?}
\end{problem}
\vspace{0.3cm}

{\bf Conflict of Interest Statement:} On behalf of all authors, the corresponding author states that there is no conflict of interest.

\vspace{0.2cm}
{\bf Date Availability Statement:} No new data are created.

\newpage

\section{Appendix}

{\bf 1. The proof of case (2) in Lemma 4.9}

For convenience, we let $A\cap (\{n_{0}\}\times \mathbb{N})=\{(n_{0},a_{k})\mid k\in\mathbb{N}\}$, where $a_{1}< a_{2}<\cdot\cdot\cdot <a_{k}<\cdot\cdot\cdot$ in $\mathbb{N}$.
Using strict order $\prec_{2}$, we know
\setlength{\baselineskip}{0.7\baselineskip}

 \begin{center}$\begin{array}{lll}
\up(n_{1},y_{1}y_{2}\cdot\cdot\cdot y_{m})\cap max(\widehat{P})&=&\{(n_{1},\top)\}\cup\{(f_{n_{1},r}(y'_{1}y'_{2}\cdot\cdot\cdot y'_{k'_{1}}),\top)\mid n_{1}<r,\\
&& y_{1}y_{2}\cdot\cdot\cdot y_{m}\leq y'_{1}y'_{2}\cdot\cdot\cdot y'_{k'_{1}},k'_{1},r\in\mathbb{N}\}.\\
\end{array}$\end{center}

 Then $\{(n_{0},a_{k_{0}}),(n_{1},y_{1}y_{2}\cdot\cdot\cdot y_{m})\}$ has an upper bound $z$ in $\widehat{P}$. Since $\{a_{k_{0}}\}$ has infinitely many possibilities, $z=(n_{0},\top)$ and $(n_{1},y_{1}y_{2}\cdot\cdot\cdot y_{m})\prec_{2} (n_{0},\top)$.

Now, we have $z\in\{(f_{n_{1},r}(y_{1}y_{2}\cdot\cdot\cdot y_{k'_{1}}),\top)\mid n_{1}<r, y_{1}y_{2}\cdot\cdot\cdot y_{m}\leq y'_{1}y'_{2}\cdot\cdot\cdot y'_{k'_{1}},k'_{1},r\in\mathbb{N}\}$. Let $z=(f_{n_{1},r_{0}}(y'_{1}y'_{2}\cdot\cdot\cdot y'_{k'_{1}}),\top)$ for some $n_{1}<r_{0}$ and $y_{1}y_{2}\cdot\cdot\cdot y_{m}\leq y'_{1}y'_{2}\cdot\cdot\cdot y'_{k'_{1}}$ in $\mathbb{N}^{<\omega}$.

Then $z=(f_{n_{1},r_{0}}(y'_{1}y'_{2}\cdot\cdot\cdot y'_{k'_{1}}),\top)=(n_{0},\top)$.

Next, we claim that $A\subseteq\dn(f_{n_{1},r_{0}}(y'_{1}y'_{2}\cdot\cdot\cdot y'_{k'_{1}}),\top)$.

Assume that $A\not\subseteq\dn(f_{n_{1},r_{0}}(y'_{1}y'_{2}\cdot\cdot\cdot y'_{k'_{1}}),\top)$. Take a $(n_{1},m''_{3})\in A\setminus\dn(f_{n_{1},r_{0}}(y'_{1}y'_{2}\cdot\cdot\cdot y'_{k'_{1}}),\top)$. Note that
\begin{center}$\begin{array}{lll}
\dn(f_{n_{1},r_{0}}(y'_{1}y'_{2}\cdot\cdot\cdot y'_{k_{1}'}),\top)&=&L_{f_{n_{1},r_{0}}(y'_{1}y'_{2}\cdot\cdot\cdot y'_{k_{1}'})}\cup\{(n_{1},y'_{1}y'_{2}\cdot\cdot\cdot y'_{r}))\mid 1\leq r\leq k_{1}'\},\\
&& \cup\{(f_{n_{1},r_{0}}(y'_{1}y'_{2}\cdot\cdot\cdot y'_{k_{1}'-1}),a)\mid 1\leq a\leq y'_{k_{1}'-1}\}.\\
\end{array}$\end{center}
We just consider the following two possibilities (2.1) and (2.2).

(2.1) $m''_{3}\in\mathbb{N}$. $\{(n_{0},a_{k_{0}}),(n_{1},m''_{3}),(n_{1},y_{1}y_{2}\cdot\cdot\cdot y_{m})\}$ has an upper bound $\widetilde{z}'$. As $n_{0}\neq n_{1}$, $\widetilde{z}'=(n_{0},\top)$ or $\widetilde{z}'=(n_{1},\top)$. As $(n_{1},m''_{3})\in A\setminus\dn(f_{n_{1},r_{0}}(y'_{1}y'_{2}\cdot\cdot\cdot y'_{k_{1}'}),\top)=\dn(n_{0},\top)$, $\widetilde{z}'\neq(n_{0},\top)$. Now, $(n_{0},a_{k_{0}})\prec_{3}\widetilde{z}'=(n_{1},\top)$ or $(n_{0},a_{k_{0}})\prec_{4}\widetilde{z}'=(n_{1},\top)$.

If $(n_{0},a_{k_{0}})\prec_{3}\widetilde{z}'=(n_{1},\top)$, then $n_{1}=f_{\overline{r}_{1}'''n_{0}}(\widetilde{a}''''_{k_{0}})$ for some $\overline{r_{1}}'''<n_{0}$ and  $a_{k_{0}}\leq\widetilde{a}''''_{k_{0}}$. Whence, $n_{1}=f_{\overline{r_{1}}'''n_{0}}(\widetilde{a}''''_{k_{0}})>n_{0}$ and $n_{0}=f_{n_{1},r_{0}}(y'_{1}y'_{2}\cdot\cdot\cdot y'_{k_{1}'})>n_{1}$, impossible.

If $(n_{0},a_{k_{0}})\prec_{4}\widetilde{z}'=(n_{1},\top)$, combining $n_{0}=f_{n_{1},r_{0}}(y'_{1}y'_{2}\cdot\cdot\cdot y'_{k_{1}'})$, then we know $n_{1}=f_{n_{1},r_{0}}(y'_{1}y'_{2}\cdot\cdot\cdot y'_{k_{1}'}\overline{a}'''_{k_{0}})$ for some $a_{k_{0}}\leq \overline{a}'''_{k_{0}}$. Whence, $n_{1}=f_{n_{1},r_{0}}(y'_{1}y'_{2}\cdot\cdot\cdot y'_{k_{1}'}\overline{a}'''_{k_{0}})>n_{1}$, impossible.

(2.2) If $m''_{3}\in\mathbb{N}^{<\omega}\setminus\{(n_{1},y'_{1}y'_{2}\cdot\cdot\cdot y'_{r}))\mid 1\leq r\leq k_{1}'\}$, then $\{(n_{0},a'_{\widehat{k}_{0}}),(n_{1},m''_{3}),(n_{1},y_{1}y_{2}\cdot\cdot\cdot y_{m})\}$ has an upper bound $\widetilde{z}''$ for some $max\{a_{k_{0}},y_{1},y_{2},\cdot\cdot\cdot,y_{m},\phi_{1},\phi_{2},\cdot\cdot\cdot,\phi_{\overline{k}'''}\}<a'_{\widehat{k}_{0}}$, where $m''_{3}=\phi_{1}\phi_{2}\cdot\cdot\cdot\phi_{\overline{k}'''}$. Obviously, $\widetilde{z}''\in\up(n_{1},y_{1}y_{2}\cdot\cdot\cdot y_{m})\cap max(\widehat{P})$.

If $\widetilde{z}''=(n_{1},\top)$, then  $(n_{0},a_{\widehat{k}_{0}})\prec_{3}\widetilde{z}''=(n_{1},\top)$ or $(n_{0},a_{\widehat{k}_{0}})\prec_{4}\widetilde{z}''=(n_{1},\top)$. Then we know $n_{0}<n_{1}$, it is impossible for $f_{n_{1},r_{0}}(y'_{1}y'_{2}\cdot\cdot\cdot y'_{k'_{1}})=n_{0}>n_{1}$.

If $(n_{1},y_{1}y_{2}\cdot\cdot\cdot y_{m})\prec_{2}\widetilde{z}''$, then $\widetilde{z}''=(f_{n_{1}\widetilde{r}''_{0}}(\widetilde{s}''),\top)$ for some $n_{1}<\widetilde{r}''_{0}$ and $y_{1}y_{2}\cdot\cdot\cdot y_{m}\leq\widetilde{s}''$ in $\mathbb{N}^{<\omega}$. Since $(n_{1},m''_{3})\in A\setminus\dn(f_{n_{1},r_{0}}(y'_{1}y'_{2}\cdot\cdot\cdot y'_{k_{1}'}),\top)$, $\widetilde{z}''\neq (n_{0},\top)$. This implies $(n_{0},a'_{\widehat{k}_{0}})\prec_{3}\widetilde{z}''$ or $(n_{0},a'_{\widehat{k}_{0}})\prec_{4}\widetilde{z}''$. In addition, according to $(n_{1},m''_{3})\prec_{2}\widetilde{z}''$, we have $m''_{3}\leq\widetilde{s}''$ in $\mathbb{N}^{<\omega}$. Suppose $(n_{0},a'_{\widehat{k}_{0}})\prec_{3}\widetilde{z}''$. Then $len (\widetilde{s}'')=1$, $n_{0}=\widetilde{r}''_{0}$ and $a'_{\widehat{k}_{0}}\leq\widetilde{s}''$. As $len (\widetilde{s}'')=1$ and $y_{1}y_{2}\cdot\cdot\cdot y_{m}\leq\widetilde{s}''$, we know $y_{1}=\widetilde{s}''$ and then $a'_{\widehat{k}_{0}}\leq y_{1}$, a contradiction. Suppose $(n_{0},a'_{\widehat{k}_{0}})\prec_{4}\widetilde{z}''$. Since $n_{0}=f_{n_{1},r_{0}}(y'_{1}y'_{2}\cdot\cdot\cdot y'_{k_{1}'})$, we have $f_{n_{1}\widetilde{r}''_{0}}(\widetilde{s}'')=f_{n_{1},r_{0}}(y'_{1}y'_{2}\cdot\cdot\cdot y'_{k_{1}}\overline{y}''')$ for some $a'_{\widehat{k}_{0}}\leq \overline{y}'''$.
So we have $\widetilde{s}''=y'_{1}y'_{2}\cdot\cdot\cdot y'_{k_{1}'}\overline{y}'''\geq m''_{3}=\phi_{1}\phi_{2}\cdot\cdot\cdot\phi_{\overline{k}'''}$. Since $\overline{y}'''\geq a'_{\widehat{k}_{0}}>\phi_{\overline{k}'''}$, $y'_{1}y'_{2}\cdot\cdot\cdot y'_{k_{1}'}\geq \phi_{1}\phi_{2}\cdot\cdot\cdot\phi_{\overline{k}'''}$. This means that $m'_{3}\in\{(n_{1},y'_{1}y'_{2}\cdot\cdot\cdot y'_{r}))\mid 1\leq r\leq k_{1}'\} $, a contradiction.

In a conclusion, $A$ has an upper bound in case (2).

{\bf 2. The proof of case(3) in Lemma 4.9}




Take a $(n_{0},w_{1}w_{2}\cdot\cdot\cdot w_{\widetilde{m}})\in A\cap (\{n_{0}\}\times \mathbb{N}^{<\omega})$, $\{(n_{0},w_{1}w_{2}\cdot\cdot\cdot w_{\widetilde{m}}),(n_{1},\hat{m}_{1})\}$ has an upper bound $l'$. Since $n_{0}\neq n_{1}$, $l'\in\up(n_{0},w_{1}w_{2}\cdot\cdot\cdot w_{\widetilde{m}})\cap max(\widehat{P})$. Note that

\begin{center}$\begin{array}{lll}
\up(n_{0},w_{1}w_{2}\cdot\cdot\cdot w_{\widetilde{m}})\cap max(\widehat{P})&=&\{(n_{0},\top)\}\cup\{(f_{n_{0}q}(w'_{1}w'_{2}\cdot\cdot\cdot w'_{\widetilde{m}'}),\top)\mid n_{0}<q,\\
&&  w_{1}w_{2}\cdot\cdot\cdot w_{\widetilde{m}}\leq w'_{1}w'_{2}\cdot\cdot\cdot w'_{\widetilde{m}'}\}.\\
\end{array}$\end{center}

Since $A\cap (\{n_{0}\}\times \mathbb{N}^{<\omega})$ is infinite, $l'=(n_{0},\top)$. Then we can see $(n_{1},\hat{m}_{1})<_{3}l'=(n_{0},\top)$ or $(n_{1},\hat{m}_{1})<_{4}l'=(n_{0},\top)$.

(3.1) $(n_{1},\hat{m}_{1})<_{3}l'=(n_{0},\top)$.

Then $n_{0}=f_{pn_{1}}(\hat{m}'_{1})$ for some $p<n_{1}$ and $\hat{m}'_{1}\leq\hat{m}_{1}$.

Claim: $A\subseteq \dn (n_{0},\top)=\dn(f_{pn_{1}}(\hat{m}'_{1}),\top)$.

Assume that $A\not\subseteq \dn(n_{0},\top)=\dn(f_{pn_{1}}(\hat{m}'_{1}),\top)$. Then there is a $(n_{1},m_{5})\in A\setminus\dn(f_{pn_{1}}(\hat{m}_{1}),\top)$. Note that $$\dn(f_{pn_{1}}(\hat{m}_{1}),\top)=L_{f_{pn_{1}}(\hat{m}_{1})}\cup\{(p,\hat{m}_{1})\}\cup\{(n_{1},j)\mid 1\leq j\leq \hat{m}'_{1}\},$$ where $(p,\hat{m}_{1})\in\{p\}\times\mathbb{N}^{<\omega}$. There are only 2 possibilities (3.1.1-3.1.2).

(3.1.1) $m_{5}\in\mathbb{N}$ with $m_{5}>\hat{m}'_{1}$.

 Then $\{(n_{0},w_{1}w_{2}\cdot\cdot\cdot w_{\widetilde{m}}),(n_{1},m_{5})\}$ has an upper bound $\widetilde{l}$. Similarly, $\widetilde{l}\in\up(n_{0},w_{1}w_{2}\cdot\cdot\cdot w_{\widetilde{m}})\cap max(\widehat{P})$. Thus,  $\widetilde{l}\in\{(f_{n_{0}q}(w'_{1}w'_{2}\cdot\cdot\cdot w'_{\widetilde{m}'}),\top)\mid n_{0}<q, w_{1}w_{2}\cdot\cdot\cdot w_{\widetilde{m}}\leq w'_{1}w'_{2}\cdot\cdot\cdot w'_{\widetilde{m}'}\}$ or $\widetilde{l}=(n_{0},\top)$.

If $\widetilde{l}=(n_{0},\top)$, then $(n_{1},m_{5})\prec_{3}\widetilde{l}=n_{0}=f_{pn_{1}}(\hat{m}'_{1})$. Whence, $m_{5}\leq\hat{m}'_{1}$, impossible.

If $\widetilde{l}\in\{(f_{n_{0}q}(w'_{1}w'_{2}\cdot\cdot\cdot w'_{\widetilde{m}'}),\top)\mid n_{0}<q, w_{1}w_{2}\cdot\cdot\cdot w_{\widetilde{m}}\leq w'_{1}w'_{2}\cdot\cdot\cdot w'_{\widetilde{m}'}\}$, or equivalently, $\widetilde{l}=(f_{n_{0}\widetilde{q}}(w'_{1}w'_{2}\cdot\cdot\cdot w'_{\widetilde{m}''}),\top)$ for some $n_{0}<\widetilde{q}$ and $w_{1}w_{2}\cdot\cdot\cdot w_{\widetilde{m}}\leq w'_{1}w'_{2}\cdot\cdot\cdot w'_{\widetilde{m}''}$ , then $(n_{1},m_{5})\prec_{3}\widetilde{l}$ or $(n_{1},m_{5})\prec_{4}\widetilde{l}$.

If $\widetilde{m}''=1$, then $(n_{1},m_{5})<_{3}\widetilde{l}=(f_{n_{0}\widetilde{q}}(w'_{1}),\top)$. Thus, $n_{1}=\widetilde{q}>n_{0}$. It contradicts with $n_{0}=f_{pn_{1}}(\hat{m}_{1})>n_{1}$, impossible.

If $1<\widetilde{m}$, then $(n_{1},m_{5})\prec_{4}\widetilde{l}=(f_{n_{0}\widetilde{q}}(w'_{1}w'_{2}\cdot\cdot\cdot w'_{\widetilde{m}''}),\top)$. Hence,  $n_{1}=f_{n_{0}\widetilde{q}}(w'_{1}w'_{2}\cdot\cdot\cdot w'_{\widetilde{m}''-1})$ and $m_{5}\leq w'_{\widetilde{m}''}$. Now, we have $n_{1}=f_{n_{0}\widetilde{q}}(w'_{1}w'_{2}\cdot\cdot\cdot w'_{\widetilde{m}''-1})>n_{0}$ and $n_{0}=f_{pn_{1}}(\hat{m}_{1})>n_{1}$, impossible.

(3.1.2) $m_{5}\in \mathbb{N}^{<\omega}$. For convenience, we let $m_{5}=\theta_{1}\theta_{2}\cdot\cdot\cdot\theta_{q_{1}}\in\mathbb{N}^{<\omega}$. Then we have $\{(n_{0},w_{1}w_{2}\cdot\cdot\cdot w_{\widetilde{m}}),(n_{1},\hat{m}_{1}),(n_{1},\theta_{1}\theta_{2}\cdot\cdot\cdot\theta_{q_{1}})\}$ has an upper bound $l_{0}$. As $n_{0}\neq n_{1}$, $l_{0}=(n_{0},\top)$ or $l_{0}=(n_{1},\top)$. As $(n_{1},m_{5})\in A\setminus\dn(f_{pn_{1}}(\hat{m}_{1}),\top)=\dn(n_{0},\top)$, $l_{0}\neq(n_{0},\top)$. Hence, we know $l_{0}=(n_{1},\top)$. Then $(n_{1},\theta_{1}\theta_{2}\cdot\cdot\cdot\theta_{q_{1}})\prec_{2}l_{0}=(n_{0},\top)$. As a result, $n_{0}=f_{n_{1}q^{\circ}}(\theta'_{1}\theta'_{2}\cdot\cdot\cdot\theta'_{q'_{1}})$ for some $n_{1}<q^{\circ}$ and $\theta_{1}\theta_{2}\cdot\cdot\cdot\theta_{q_{1}}\leq\theta'_{1}\theta'_{2}\cdot\cdot\cdot\theta'_{q'_{1}}$ in $\mathbb{N}^{<\omega}$. Now, we can see $n_{0}=f_{pn_{1}}(\hat{m}'_{1})=f_{n_{1}q^{\circ}}(\theta'_{1}\theta'_{2}\cdot\cdot\cdot\theta'_{q'_{1}})$ and hence $p=n_{1}<n_{1}$, a contradiction.

Consequently, $A\subseteq \dn (n_{0},\top)=\dn(f_{pn_{1}}(\hat{m}'_{1}),\top)$ whenever $(n_{1},\hat{m}_{1})\prec_{3}l'=(n_{0},\top)$.

(3.2) $(n_{1},\hat{m}_{1})\prec_{4}l'=(n_{0},\top)$.

Then $n_{1}=f_{c'd'}(o')$ and $n_{0}=f_{c'd'}(o'\hat{m}''_{1})$ for some $c',d',\hat{m}''_{1}\in\mathbb{N}$, $o'\in\mathbb{N}^{<\omega}$ and $\hat{m}_{1}\leq\hat{m}''_{1}$.

Claim: $A\subseteq\dn (n_{0},\top)=\dn(f_{c'd'}(o'\hat{m}''_{1}),\top)$.

Assume that $A\not\subseteq\dn(f_{c'd'}(o'\hat{m}_{1}),\top)$. Then we can take a $(n_{1},m_{6})\in A\setminus\dn(f_{c'd'}(o'\hat{m}_{1}),\top)$. Note that

\setlength{\baselineskip}{1.0\baselineskip}

\begin{center}$\begin{array}{lll}
\dn(f_{c'd'}(o'\hat{m}''_{1}),\top)=&=&L_{f_{c'd'}(o'\hat{m}_{1})}\cup\{(c',s'')\mid s''\leq o'\hat{m}''_{1}, s''\in\mathbb{N}^{<\omega}\}\\
&& \cup\{(n_{1},g)\mid 1\leq g\leq \hat{m}''_{1}\},\\
\end{array}$\end{center}

where $\{(c',o'),(c',o'\hat{m}''_{1})\}\subseteq\{c'\}\times\mathbb{N}^{<\omega}$.
So there are only 2 possibilities (3.2.1-3.2.2).

(3.2.1) $m_{6}\in\mathbb{N}$ with $\hat{m}''_{1}<m_{6}$.

Then $\{(n_{0},w_{1}w_{2}\cdot\cdot\cdot w_{\widetilde{m}}),(n_{1},m_{6})\}$ has an upper bound $l_{2}$. Clearly, we can see
$l_{2}\in\up(n_{0},w_{1}w_{2}\cdot\cdot\cdot w_{\widetilde{m}})\cap max(\widehat{P})$. Suppose $l_{2}=(n_{0},\top)=(f_{c'd'}(o'\hat{m}''_{1}),\top)$. Then $(n_{1},m_{6})\prec_{4}(f_{c'd'}(o'\hat{m}''_{1}),\top)$. Hence, $m_{6}\leq\hat{m}''_{1}$, a contradiction. Suppose $l_{2}=(f_{n_{0}q_{2}}(w''_{1}w''_{2}\cdot\cdot\cdot w''_{\widetilde{m}''}),\top)$ for some $n_{0}<q_{2}$ and $w_{1}w_{2}\cdot\cdot\cdot w_{\widetilde{m}}\leq w''_{1}w''_{2}\cdot\cdot\cdot w''_{\widetilde{m}''}$. If $\widetilde{m}''=1$, then
$(n_{1},m_{6})\prec_{3}(f_{n_{0}q_{2}}(w''_{1}),\top)$. So we have $n_{1}=q_{2}>n_{0}$. This contradicts $n_{1}=f_{c'd'}(o')<f_{c'd'}(o'\hat{m}''_{1})=n_{0}$. If $1<\widetilde{m}$, then $(n_{1},m_{6})\prec_{4}(f_{n_{0}q_{2}}(w''_{1}w''_{2}\cdot\cdot\cdot w''_{\widetilde{m}''}),\top)$. Therefore, $n_{1}=f_{n_{0}q_{2}}(w''_{1}w''_{2}\cdot\cdot\cdot w''_{\widetilde{m}''-1})>n_{0}$, impossible.

(3.2.2) $m_{6}\in\mathbb{N}^{<\omega}$. Then $\{(n_{0},w_{1}w_{2}\cdot\cdot\cdot w_{\widetilde{m}}),(n_{1},\hat{m}_{1}),(n_{1},m_{6})\}$ has an upper bound $l_{3}$. Since $n_{0}\neq n_{1}$, $l_{3}=(n_{0},\top)$ or $l_{3}=(n_{1},\top)$. Suppose $l_{3}=(n_{0},\top)=(f_{c'd'}(o'\hat{m}''_{1}),\top)$. Then $(n_{1},m_{6})\prec_{2}l_{3}=(f_{c'd'}(o'\hat{m}''_{1}),\top)$. Whence, $n_{1}=c'<f_{c'd'}(o')=n_{1}$, a contradiction. Suppose $l_{3}=(n_{1},\top)$. This implies $(n_{0},w_{1}w_{2}\cdot\cdot\cdot w_{\widetilde{m}})\prec_{2}(n_{1},\top)$ and $(n_{1},\hat{m}_{1})\leq(n_{1},\top)$. It follows from $(n_{0},w_{1}w_{2}\cdot\cdot\cdot w_{\widetilde{m}})\prec_{2}(n_{1},\top)=(f_{c'd'}(o'),\top)$ that there are $n_{0}<q_{3}$ and $w_{1}w_{2}\cdot\cdot\cdot w_{\widetilde{m}}\leq\overline{w}'''_{1}\overline{w}'''_{2}\cdot\cdot\cdot \overline{w}'''_{\widetilde{m}'''}$ in $\mathbb{N}^{<\omega}$ such that $f_{c'd'}(o')=f_{n_{0}q_{3}}(\overline{w}'''_{1}\overline{w}'''_{2}\cdot\cdot\cdot \overline{w}'''_{\widetilde{m}'''})$.  Consequently, $n_{0}=c'<f_{c'd'}(o'\hat{m}''_{1})=n_{0}$, impossible.

Consequently, $A\subseteq\dn (n_{0},\top)=\dn(f_{c'd'}(o'\hat{m}''_{1}),\top)$, whenever $(n_{1},\hat{m}_{1})\prec_{4}l'=(n_{0},\top)$.

In a conclusion, $A$ has an upper bound in case (3).

{\bf 3 The proof of case(4) in Lemma 4.9.}



Choose a $(n_{0},x_{1}x_{2}\cdot\cdot\cdot x_{\widetilde{r}})\in A\cap (\{n_{0}\}\times \mathbb{N}^{<\omega})$. Then $\{(n_{0},x_{1}x_{2}\cdot\cdot\cdot x_{\widetilde{r}}),(n_{1},e_{1}e_{2}\cdot\cdot\cdot e_{k_{1}^{\prime}})\}\in \mathbf{Fin}(A)$ has an upper bounds $h^{\ast}$. Since $n_{0}\neq n_{1}$, $h^{\ast}\in\up(n_{0},x_{1}x_{2}\cdot\cdot\cdot x_{\widetilde{r}})\cap max(\widehat{P})$.
Since $A\cap (\{n_{0}\}\times \mathbb{N}^{<\omega})$ is infinite, $h^{\ast}=(n_{0},\top)$ and $(n_{1},e_{1}e_{2}\cdot\cdot\cdot e_{k_{1}^{\prime}})<_{2}h^{\ast}=(n_{0},\top)$. So we have $n_{0}=f_{n_{1}t'_{1}}(e'_{1}e'_{2}\cdot\cdot\cdot e'_{\overline{k}_{1}^{\prime}})$ for some $n_{1}<t'_{1}$ and $e_{1}e_{2}\cdot\cdot\cdot e_{k_{1}^{\prime}}\leq e'_{1}e'_{2}\cdot\cdot\cdot e'_{\overline{k}_{1}^{\prime}}$.

Now, we claim that $A\subseteq\dn(n_{0},\top)=\dn(f_{n_{1}t'_{1}}(e'_{1}e'_{2}\cdot\cdot\cdot e'_{\overline{k}_{1}^{\prime}}),\top)$.

Assume that $A\not\subseteq\dn(f_{n_{1}t'_{1}}(e'_{1}e'_{2}\cdot\cdot\cdot e'_{\overline{k}_{1}^{\prime}}),\top)$. Take a $(n_{1},m_{1})\in A\setminus\dn(f_{n_{1}t'_{1}}(e'_{1}e'_{2}\cdot\cdot\cdot e'_{\overline{k}_{1}^{\prime}}),\top)$. Note that

 \begin{center}$\begin{array}{lll}
\dn(f_{n_{1}t'_{1}}(e'_{1}e'_{2}\cdot\cdot\cdot e'_{\overline{k}_{1}^{\prime}}),\top)&=&L_{f_{n_{1}t'_{1}}(e'_{1}e'_{2}\cdot\cdot\cdot e'_{\overline{k}_{1}^{\prime}})}\cup\{(n_{1},(e'_{1}e'_{2}\cdot\cdot\cdot e'_{s}))\mid 1\leq s\leq \overline{k}_{1}^{\prime}\}\\
&&\cup\{(f_{n_{1}t'_{1}}(e'_{1}e'_{2}\cdot\cdot\cdot e'_{\overline{k}_{1}^{\prime}-1}),e)\mid 1\leq e\leq e_{\overline{k}_{1}^{\prime}}\}.\\
\end{array}$\end{center}

Then $\{(n_{0},x_{1}x_{2}\cdot\cdot\cdot x_{\widetilde{r}}),(n_{1},e_{1}e_{2}\cdot\cdot\cdot e_{k_{1}^{\prime}}),(n_{1},m_{1})\}\in\mathbf{Fin}(A)$ has an upper bound $h'_{1}$. Obviously, $h'_{1}\in\up(n_{0},x_{1}x_{2}\cdot\cdot\cdot x_{\widetilde{r}})\cap max(\widehat{P})$. As $(n_{1},m_{1})\not\in\dn(f_{n_{1}t'_{1}}(e'_{1}e'_{2}\cdot\cdot\cdot e'_{\overline{k}_{1}^{\prime}}),\top)$, we have that $(n_{0},x_{1}x_{2}\cdot\cdot\cdot x_{\widetilde{r}})\prec_{2} h'_{1}$. Let $h'_{1}=(f_{n_{0}t'_{1}}(x'_{1}x'_{2}\cdot\cdot\cdot x'_{\widetilde{r}'}),\top)$ with $n_{0}<t'_{1}$ and $x_{1}x_{2}\cdot\cdot\cdot x_{\widetilde{r}}\leq x'_{1}x'_{2}\cdot\cdot\cdot x'_{\widetilde{r}'}$. Assume that $n_{1}=f_{n_{0}t'_{1}}(x'_{1}x'_{2}\cdot\cdot\cdot x'_{\widetilde{r}'})$. Then we have $n_{1}=f_{n_{0}t'_{1}}(x'_{1}x'_{2}\cdot\cdot\cdot x'_{\widetilde{r}'})>n_{0}=f_{n_{1}t'_{1}}(e'_{1}e'_{2}\cdot\cdot\cdot e'_{\overline{k}_{1}^{\prime}})>n_{1}$, impossible. So $n_{1}\neq f_{n_{0}t'_{1}}(x'_{1}x'_{2}\cdot\cdot\cdot x'_{\widetilde{r}'})$ and $(n_{1},e_{1}e_{2}\cdot\cdot\cdot e_{k_{1}^{\prime}})\prec_{2}(f_{n_{0}t'_{1}}(x'_{1}x'_{2}\cdot\cdot\cdot x'_{\widetilde{r}'}),\top)$. Consequently,$n_{0}=n_{1}$,impossible.

Therefore, $A\subseteq\dn(n_{0},\top)=\dn(f_{n_{1}t'_{1}}(e'_{1}e'_{2}\cdot\cdot\cdot e'_{\overline{k}_{1}^{\prime}}),\top)$ in subcase (4).

{\bf 4. The proof of case (6) in Lemma 4.10.}




Choose a $(n_{0},\delta_{1}\delta_{2}\cdot\cdot\cdot\delta_{k_{0}})\in A\cap(\{n_{0}\}\times\mathbb{N}^{<\omega})$, a $(n_{1},k_{1})\in A\cap(\{n_{1}\}\times\mathbb{N})$ and a $(n_{2},\lambda_{1}\lambda_{2}\cdot\cdot\cdot\lambda_{k_{2}})\in A\cap(\{n_{2}\}\times\mathbb{N}^{<\omega})$.
Then $\{(n_{0},\delta_{1}\delta_{2}\cdot\cdot\cdot\delta_{k_{0}}),(n_{1},k_{1}),(n_{2},\lambda_{1}\lambda_{2}\cdot\cdot\cdot\lambda_{k_{2}})\}$ has an upper bound $\overline{g}$. Since $A\cap (\{n_{0}\}\times \mathbb{N}^{<\omega})$ is infinite, $\overline{g}=(n_{0},\top)$, $(n_{1},k_{1})\prec_{3}\overline{g}$ and $(n_{2},\lambda_{1}\lambda_{2}\cdot\cdot\cdot\lambda_{k_{2}})\prec_{2}\overline{g}$ with $k_{2}=1$.

By $(n_{2},\lambda_{1})\prec_{2}\overline{g}=(n_{0},\top)$, we have $n_{0}=f_{n_{2}\overline{h}_{1}}(\lambda_{1})$ for some $n_{2}<\overline{h}$.
Since $(n_{1},k_{1})\prec_{3}\overline{g}=(n_{0},\top)=(f_{n_{2}\overline{h}_{1}}(\lambda_{1}),\top)$, $n_{1}=\overline{h}_{1}$ and $k_{1}\leq\lambda_{1}$.
In the following, we assert that $A\subseteq\dn(n_{0},\top)=\dn(f_{n_{2}\overline{h}_{1}}(\lambda_{1}),\top)$.

Assume that $A\not\subseteq\dn(f_{n_{2}\overline{h}_{1}}(\lambda_{1}),\top)$. Choose a $(n_{3},\overline{k}_{3})\in A\setminus\dn(f_{n_{2}\overline{h}_{1}}(\lambda_{1}),\top)$. Note that
$$\dn(f_{n_{2}\overline{h}_{1}}(\lambda_{1}),\top)=L_{f_{n_{2}\overline{h}_{1}}(\lambda_{1})}\cup\{(n_{2},\lambda_{1})\}\cup\{(\overline{h}_{1},i)\mid 1\leq i\leq \lambda_{1}\},$$ where $(n_{2},\lambda_{1})\in\{n_{1}\}\times\mathbb{N}^{<\omega}$.

(6.1) $n_{3}=n_{1}=\overline{h}_{1}$ and $\overline{k}_{3}\in\mathbb{N}$ with $\lambda_{1}<\overline{k}_{3}$.

Then $\{(n_{0},\delta_{1}\delta_{2}\cdot\cdot\cdot\delta_{k_{0}}),(n_{1},\overline{k}_{3}),(n_{2},\lambda_{1})\}$ has an upper bound $\overline{g}_{3}$. Then $\overline{g}_{3}=(n_{0},\top)$ and $\overline{g}_{3}=(n_{2},\top)$. As $(n_{1},\overline{k}_{3})\in A\setminus\dn(f_{n_{2}\overline{h}_{1}}(\lambda_{1}),\top)=\dn(n_{0},\top)$, $\overline{g}_{3}\neq(n_{0},\top)$. Then $\overline{g}_{3}=(n_{2},\top)$ and $(n_{0},\delta_{1}\delta_{2}\cdot\cdot\cdot\delta_{k_{0}})\prec_{2}\overline{g}_{3}\neq(n_{2},\top)$. This implies $n_{0}<n_{2}$. This contradicts $n_{0}=f_{n_{2}\overline{h}_{1}}(\lambda_{1})>n_{2}$.

(6.2) $n_{3}=n_{1}=\overline{h}_{1}$ and $\overline{k}_{3}\in\mathbb{N}^{<\omega}$.

Then $\{(n_{0},\delta_{1}\delta_{2}\cdot\cdot\cdot\delta_{k_{0}}),(n_{1},k_{1}),(n_{2},\lambda_{1}), (n_{1},\overline{k}_{3})\}$ has an upper bound $\overline{g}'_{3}$. This is impossible since $n_{0},n_{1},n_{2}$ are different and $\delta_{1}\delta_{2}\cdot\cdot\cdot\delta_{k_{0}},\lambda_{1},\overline{k}_{3}\in\mathbb{N}^{<\omega}$.

(6.3) $n_{3}=n_{2}$ and $\overline{k}_{3}\in\mathbb{N}$.

Then $\{(n_{0},\delta_{1}\delta_{2}\cdot\cdot\cdot\delta_{k_{0}}),(n_{1},k_{1}),(n_{2},\lambda_{1}), (n_{2},\overline{k}_{3})\}$ has an upper bound $\overline{g}_{4}$. Since $\delta_{1}\delta_{2}\cdot\cdot\cdot\delta_{k_{0}},\lambda_{1}\in\mathbb{N}^{<\omega}$, $\overline{g}_{4}=(n_{0},\top)$ and $\overline{g}_{4}=(n_{2},\top)$. In addition, as $k_{1},\overline{k}_{3}\in\mathbb{N}$, $\overline{g}_{4}=(n_{1},\top)$ and $\overline{g}_{4}=(n_{2},\top)$. Consequently, $\overline{g}_{4}=(n_{2},\top)$ and then $(n_{0},\delta_{1}\delta_{2}\cdot\cdot\cdot\delta_{k_{0}})\prec_{2}\overline{g}_{4}=(n_{2},\top)$. By strict order $\prec_{2}$, $n_{0}<n_{2}$. However, $n_{0}=f_{n_{2}\overline{h}_{1}}(\lambda_{1})>n_{2}$, a contradiction.

(6.4) $n_{3}=n_{2}$ and $\overline{k}_{3}\in\mathbb{N}^{<\omega}\setminus\{\lambda_{1}\}$.

Then  $\{(n_{0},\delta_{1}\delta_{2}\cdot\cdot\cdot\delta_{k_{0}}),(n_{1},k_{1}),(n_{2},\lambda_{1}), (n_{2},\overline{k}_{3})\}$ has an upper bound $\overline{g}'_{4}$. Clearly, $\overline{g}'_{4}=(n_{0},\top)$ and $\overline{g}'_{4}=(n_{2},\top)$. As $(n_{2},\overline{k}_{3})\in A\setminus\dn(f_{n_{2}\overline{h}_{1}}(\lambda_{1}),\top)=\dn(n_{0},\top)$, $\overline{g}'_{4}\neq(n_{0},\top)$. So $\overline{g}'_{4}=(n_{2},\top)$ and $(n_{0},\delta_{1}\delta_{2}\cdot\cdot\cdot\delta_{k_{0}})\prec_{2}\overline{g}'_{4}=(n_{2},\top)$. Similarly, $n_{0}<n_{2}$ but $n_{0}=f_{n_{2}\overline{h}_{1}}(\lambda_{1})>n_{2}$, a contradiction.

In other words, $A\subseteq\dn(n_{0},\top)=\dn(f_{n_{2}\overline{h}_{1}}(\lambda_{1}),\top)$ in case (6).

{\bf 5. The proof of case (8) in Lemma 4.10.}

Choose a $(n_{0},\rho_{1}\rho_{2}\cdot\cdot\cdot\rho_{u_{0}})\in A\cap(\{n_{0}\}\times\mathbb{N}^{<\omega})$, a $(n_{1},m'_{1})\in A\cap(\{n_{1}\}\times\mathbb{N})$ and a $(n_{2},m'_{2})\in A\cap(\{n_{2}\}\times\mathbb{N})$.
Clearly, $\{(n_{0},\rho_{1}\rho_{2}\cdot\cdot\cdot\rho_{u_{0}}),(n_{1},m'_{1}),(n_{2},m'_{2})\}$
has an upper bound $v_{1}$.
Since $A\cap (\{n_{0}\}\times \mathbb{N}^{<\omega})$ is infinite, we have $v_{1}=(n_{0},\top)$ and the following 2 cases:

(1) $(n_{1},m'_{1})\prec_{3}v_{1}$ and $(n_{2},m'_{2})\prec_{4}v_{1}$;

(2) $(n_{1},m'_{1})\prec_{4}v_{1}$ and $(n_{2},m'_{2})\prec_{3}v_{1}$.

 Case 1: $(n_{1},m'_{1})\prec_{3}v_{1}$ and $(n_{2},m'_{2})\prec_{4}v_{1}$.

 By $(n_{1},m'_{1})\prec_{3}v_{1}$, we have $v_{1}=(n_{0},\top)=(f_{d,n_{1}}(m'_{1}),\top)$ for some $d<n_{1}$. Since $(n_{2},m'_{2})\prec_{4}v_{1}=(n_{2},\top)$, $n_{2}=f_{a,b}(s)$ and $n_{0}=f_{a,b}(sm'_{2})$ for some $a,b\in\mathbb{N}$ and $s\in\mathbb{N}^{<\omega}$. So we know $n_{0}=f_{a,b}(sm'_{2})=f_{d,n_{1}}(m'_{1})$. Whence, $m'_{1}=sm'_{2}$, impossible.

Case 2: $(n_{1},m'_{1})\prec_{4}v_{1}$ and $(n_{2},m'_{2})\prec_{3}v_{1}$.

Impossibly, it is  similar to case 1.

All in all, the case (8) is always impossible.


\newpage
\begin{thebibliography}{s9}







\small {\baselineskip 0.07in


\bibitem{RTT45} S. Abramsky, A. Jung, Domain theory, in: S. Abramsky, D. Gabbay, T. Maibaum (Eds.), Semantic Structures, in: Handbook
of Logic in Computer Science, vol. 3, Clarendon Press, 1994, pp. 1-168.

\bibitem{WM34} M. Brecht, T. Kawai, On the commutativity of the powerspace constructions, Logical Methods in Computer Science, 15(3) (2019)
13:1-13:25.

\bibitem{GHY30} Y. Chen, H. Kou, Z. Lyu, Two topologies on the lattice of Scott closed subsets, Topology and its Applications, 306 (2022) 107918.

\bibitem{AQ21} M. Ern\'{e}, Infinite distributive laws versus local connectedness and compactness properties, Topology and its Applications, 156(12)(2009) 2054-2069.

\bibitem{KDFGY35} M. Ern\'{e}, Categories of locally hypercompact spaces and quasicontinuous posets, Applied Categorical Structures, 26 (2018) 823-854.


\bibitem{RE77} R. Engelking, General Topology, Polish Scientific Publishers, Warszawa, 1977.

\bibitem{GG03} G. Gierz, K. H. Hofmann, K. Keimel, J. D. Lawson, M. W. Mislove and D. S. Scott, Continuous Lattices and Domains, Encyclopedia of Mathematics and its Applications, vol. 93, Cambridge University Press, Cambridge, 2003.

\bibitem{JGKD13} J. Goubault-Larrecq, Non-Hausdorff Topology and Domain Theory, New Mathematical Monographs, vol. 22, Cambridge University Press, New York, 2013.


\bibitem{GGL23}Y. H. Han, S. S. Hong, C. K. Lee, P. U. Park, A generalization of posets, Comm. Korean Math. Soc,4(1) (1989), 129-138.


\bibitem{AGL23} R. Heckmann, An upper power domain construction in terms of strongly compact sets, in: Lecture Notes in Computer
Science, vol. 598, Springer, Berlin Heidelberg New York, 1992, 272-293.



\bibitem{MU13} R. Heckmann, K. Keimel, Quasicontinuous Domains and the Symth Power-domain, Electronic Notes in Theoretical Computer Science, 298(2013) 215-232.

\bibitem{BZ35} Z. He, K. Wang, A compact space is not always $SI$-compact, Rocky Mountain Journal of Mathematics, 52(6)(2022) 2041-2051.


\bibitem{BWR30} Z. He, B. Zhao, Some problems about co-consonant spaces, submitted to Acta Mathematica Sinica, English Series, 2023.

\bibitem{IJ82}  J. Isbell, Completion of a construction of Johnstone, Proceedings of the American Mathematical Society,  85
(1982) 333-334.

\bibitem{SOJ81} P. T. Johnstone, Scott is not always sober, in Continuous Lattices, Lecture Notes in Mathematics, 871 (1981) 282-283.


\bibitem{WSMJ81} H. Kou, Uk-admitting dcpo's need not be sober, in: Domains and Processes, Semantic Structure on Domain Theory, vol.
1, Kluwer, 2001, 41-50.


\bibitem{GOY03} J. D. Lawson, G. Wu, X. Xi. Well-filtered spaces, compactness, and the lower topology, Houston Journal of Mathematics, 46(2020)
    283-294.

\bibitem{KOY35}  S. O. Lee, On countably approximating lattices, J. Korean Math. Soc, 25(1)(1988), 11-23.




\bibitem{GOUY03} H. Miao, Q. Li and D. Zhao, On two problems about sobriety of topological spaces, Topology and its Applications, 295 (2021) 107667.

\bibitem{AB75} H. Miao, X. Xi, Q. Li and D. Zhao. Not every countable complete distributive lattice is sober. Mathematical Structures in Computer Science, 1-23 (2023).

\bibitem{ANVB75} A. Schalk, Algebras for Generalized Power Constructions, PhD Thesis, Technische Hochschule Darmstadt, 1993.

\bibitem{ARU175} W. Thron, Lattice-equivalence of topological spaces, Duke Mathematical Journal, 29 (1962), 671-679.



\bibitem{XSW87} X. Xu, Compact subsets of the Sorgenfrey line (in Chinese), Journal of Jiangxi Polytechnlc University, 12(4)(1990) 1-2.

\bibitem{SDF03} X. Xi, J. D. Lawson, On well-filtered spaces and ordered sets, Topology and its Applications, 228(2017)139-144.

\bibitem{KU13} X. Xu, C. Shen, X. Xi, D. Zhao, On $T_{0}$ spaces determined by well-filtered spaces, Topology and its Applications, 282 (2020) 107323.

\bibitem{HZ16} X. Xu, X. Xi, D. Zhao, A complete Heyting algebra whose Scott space is non-sober, Fundamenta Mathematicae, 252 (2021) 315-323.

\bibitem{HZT96}X. Xu, D. Zhao, Some open problems on well-filtered spaces and sober spaces, Topology and its Applications, 301(2021) 107540.

\bibitem{HQT116} J. Yang, J. Shi, Countably sober spaces, Electronic Notes in Theoretical Computer Science,  333 (2017) 143-151.

\bibitem{DME15} D. Zhao, W. K. Ho, On topologies defined by irreducible sets, Journal of Logical and Algebraic Methods in Programming, 84 (2015) 185-195.
\bibitem{DFG15} D. Zhao, X. Xi and Y. Chen, A new dcpo whose Scott topology is well-filtered but not sober,  Topology and its Applications, 252 (2019),
97-102.

}








\end{thebibliography}
\end{document}